\documentclass[11pt]{article}

\usepackage[margin=1in]{geometry}
\usepackage{amsmath,amssymb,amsfonts,amsthm,mathtools,mathrsfs}
\usepackage{bm}
\usepackage{enumitem}
\usepackage{xcolor}
\usepackage{graphicx}
\usepackage{booktabs}
\usepackage{tikz}
\usetikzlibrary{arrows.meta}
\usepackage[colorlinks=true,linkcolor=blue,citecolor=blue,urlcolor=blue]{hyperref}
\providecommand{\keywords}[1]{%
  \par\smallskip\noindent\textbf{Keywords: }#1\par}
\providecommand{\subjclass}[2][]{%
  \par\smallskip\noindent\textbf{#1 Mathematics Subject Classification: }#2\par}

\numberwithin{equation}{section}

\theoremstyle{plain}
\newtheorem{theorem}{Theorem}[section]
\newtheorem{proposition}[theorem]{Proposition}
\newtheorem{lemma}[theorem]{Lemma}

\theoremstyle{definition}
\newtheorem{assumption}[theorem]{Assumption}
\newtheorem{definition}[theorem]{Definition}

\theoremstyle{remark}
\newtheorem{remark}[theorem]{Remark}

\newcommand{\R}{\mathbb R}
\newcommand{\eps}{\varepsilon}
\renewcommand{\Cap}{\operatorname{Cap}}
\newcommand{\dist}{\operatorname{dist}}
\newcommand{\E}{\mathbb E}
\newcommand{\Pbb}{\mathbb P}
\newcommand{\calD}{\mathcal D}
\newcommand{\calL}{\mathcal L}
\newcommand{\loc}{\mathrm{loc}}

\title{An Eyring--Kramers Law for the Hypoelliptic Third-Order\\
Langevin Diffusion}
\author{
Yingli Wang\thanks{School of Mathematical Sciences, Fudan University, Shanghai, People's Republic of China; \texttt{yingliwang@fudan.edu.cn}}
\and
Lingjiong Zhu\thanks{Department of Mathematics, Florida State University, Tallahassee, Florida, United States of America; \texttt{zhu@math.fsu.edu}}}
\date{\today}

\begin{document}
\maketitle

\begin{abstract}
We prove an Eyring--Kramers law for metastable transitions of the
hypoelliptic third-order Langevin diffusion in the low-temperature limit.
This diffusion is a three-level Markovian lifting of Langevin dynamics:
the Brownian noise acts only on the highest auxiliary variable and reaches
the position variable through a third-order H\"ormander chain.  For a
double-well potential with a unique index-one transition saddle, we
determine both the Arrhenius exponential scale and the sharp prefactor of
the mean transition time.  The prefactor is governed by the unique positive
unstable rate of the deterministic linearization at the saddle,
equivalently the positive root of a cubic polynomial.  Our proof combines
a weak-capacity framework with a saddle-adapted boundary layer, an explicit
Gaussian current calculation, committor localization, and intrawell
flatness.
Under matched kinetic normalizations, the resulting metastable prefactor is
strictly smaller than its underdamped counterpart.
A numerical experiment for a one-dimensional double well illustrates the
Arrhenius scaling and the predicted prefactor comparison.
\end{abstract}



\tableofcontents

\section{Introduction}

Langevin diffusions form a central class of Markov processes for sampling from
target distributions.  They have applications in
Bayesian learning
\cite{gelman1995bayesian,stuart2010inverse,andrieu2003introduction,teh2016consistency,DistMCMC19,GIWZ2024}
and in non-convex optimization problems arising in machine learning
\cite{Raginsky,xu2018global,Chau2019,GGZ,Chau2022,Zhang2019}.
Given a potential function $U:\R^d\to\R$, the target distribution, also known as the Gibbs distribution, is of the form
\begin{align}\label{eqn:target}
\pi_\eps^\theta(d\theta)
=
Z_{\eps}^{-1} e^{-U(\theta)/\eps}\,d\theta ,
\end{align}
where $\eps>0$ is the temperature parameter and $Z_{\eps}>0$ is a normalizing factor.
The classical \textit{overdamped Langevin diffusion} satisfies the following stochastic differential equation (SDE) \cite{Dalalyan,DM2017,DK2017,Raginsky}:
\begin{equation}\label{eq:overdamped-2}
d\theta_{t}=-\nabla U(\theta_{t})dt+\sqrt{2\varepsilon}dB_{t},
\end{equation}
where $B_{t}$ is a standard $d$-dimensional Brownian motion,
and it is known that under mild conditions,
the overdamped Langevin diffusion \eqref{eq:overdamped-2} admits
\eqref{eqn:target} as the unique invariant distribution \cite{chiang1987diffusion,stroock-langevin-spectrum}.

When $\eps$ is small and $U$ is non-convex, the distribution concentrates near the global minima of $U$,
which is why Langevin diffusions have also been widely used
to obtain global convergence guarantees for solving non-convex optimization problems that often arise in machine learning \cite{Raginsky,xu2018global,Chau2019,Zhang2019}.
In the meantime, the dynamics exhibits metastability: the process remains trapped near a local minimum for a long time before making a rare transition through a saddle point to another well \cite{BovierDenHollander2015}.

At the logarithmic scale, rare transitions between wells are described by the
Freidlin--Wentzell theory of small random perturbations
\cite{FreidlinWentzell2012}.  For reversible overdamped Langevin dynamics, the
sharp low-temperature transition law is the classical Eyring--Kramers formula,
and the potential-theoretic approach expresses metastable transition times in
terms of equilibrium potentials, equilibrium measures, and capacities
\cite{Berglund2013Kramers,BovierEckhoffGayrardKlein2004,BovierDenHollander2015}.
More precisely, suppose that $m$ and $s$ are local minima separated by a
unique index-one saddle $\sigma$, and let $-\lambda_\sigma<0$ be the unique
negative eigenvalue of $\nabla^2U(\sigma)$.  If
$\tau_{B(s,\eps)}$ denotes the first hitting time of the ball $B(s,\eps)$,
then the Eyring--Kramers law for the overdamped Langevin diffusion \eqref{eq:overdamped-2} provides the leading order asymptotics
for the expected time of hitting the ball $B(s,\varepsilon)$ starting at the local minimum $m$:
\begin{equation}
\label{eq:overdamped-EK-intro}
\E_m\left[\tau_{B(s,\eps)}\right]
=
[1+o_\eps(1)]\,
\frac{2\pi}{\lambda_\sigma}
\sqrt{\frac{-\det\nabla^2U(\sigma)}{\det\nabla^2U(m)}}
\exp\!\left(\frac{U(\sigma)-U(m)}{\eps}\right),
\end{equation}
as $\eps\downarrow0$; see also
\cite[Section~1.1]{LeeRamilSeo2026}.  Thus the barrier height determines the
Arrhenius exponential scale, while the Hessians at the starting minimum and
transition saddle determine the sharp prefactor.
Geometric capacity methods provide a complementary route and also accommodate
more general critical-point geometries \cite{AvelinJulinViitasaari2023}.

Non-reversibility and degeneracy require further tools.  Landim and Seo derived
an Eyring--Kramers transition formula for non-reversible random walks in a
potential field \cite{LandimSeo2018}.  For non-self-adjoint elliptic operators,
Landim, Mariani, and Seo established Dirichlet and Thomson principles and
applied them to non-reversible metastable diffusions
\cite{LandimMarianiSeo2019}, while Lee and Seo developed a test-function method
for sharp capacities of non-reversible Gibbs diffusions \cite{LeeSeo2022}.
Related irreversible Eyring--Kramers and exit-time asymptotics appear in
\cite{BouchetReygner2016,LePeutrecMichelNectoux2024}.

For kinetic Langevin diffusions and their associated Fokker--Planck operators,
spectral approaches give sharp
low-temperature information through the low-lying spectrum
\cite{HerauHitrikSjostrand2008,BonyLePeutrecMichel2025}.
Bony, Le~Peutrec, and Michel proved spectral Eyring--Kramers formulas for a
broad class of Fokker--Planck type operators and noted that their framework
applies to generalized Langevin generators
\cite[Theorem~2 and p.~4359]{BonyLePeutrecMichel2025}.  For the present
third-order generator, the local Kalman chain $r\to p\to\theta$ is compatible
with their condition \textup{(Harmo)}, and their saddle matrix yields the same
unstable rate $\mu_\sigma$.  Thus their theorem gives the corresponding
whole-space spectral formula when its additional global coefficient,
hypocoercivity, finite-critical-set, and generic-labeling assumptions hold.
The present theorem works under
Assumptions~\ref{ass:double-well} and~\ref{ass:growth-lyapunov} and gives the
directed, pointwise mean hitting time of the shrinking set $S_\eps$ for
either ordering of the two well depths.

Delande developed semiclassical hypocoercive estimates and
Eyring--Kramers asymptotics for the low-lying eigenvalues of a class of
degenerate Kramers--Fokker--Planck operators
\cite[Assumption~6 and Theorems~1--2]{Delande2026KFP}.
His stated sufficient condition requires a one-step averaged transport matrix
to be uniformly positive definite.  That matrix is singular at the lifted
saddle for the present third-order chain, where the noise reaches
$\theta$ only through the two successive links $r\to p\to\theta$;
hence the present generator lies outside that stated sufficient condition.
Delande's result, like \cite{BonyLePeutrecMichel2025}, concerns low-lying
whole-space eigenvalues, while Theorem~\ref{thm:EK} concerns a pointwise mean
hitting-time asymptotic.
Metastable hypocoercivity and simulated annealing were studied in
\cite{Monmarche2018}.
Lee, Ramil, and Seo obtained detailed asymptotic-stability and cutoff estimates
for the underdamped dynamics \cite{LeeRamilSeo2023Cutoff}.  Most directly
related to the present work, they then proved an Eyring--Kramers law for the
kinetic Langevin diffusion, also known as the \textit{underdamped Langevin diffusion} and developed a weak-equilibrium-measure treatment
of its characteristic-boundary issues \cite{LeeRamilSeo2026}:
\begin{equation}\label{eqn:underdamped}
\begin{cases}
dr_{t}=-\gamma r_{t}dt-\nabla U(\theta_{t})dt+\sqrt{2\gamma\varepsilon}dB_{t},\\
d\theta_{t}=r_{t}dt,
\end{cases}
\end{equation}
where $B_{t}$ is a standard $d$-dimensional Brownian motion
with $\gamma>0$ being the friction coefficient and $\varepsilon>0$ being a scaling parameter.
Under some mild assumptions on $U$, the diffusion \eqref{eqn:underdamped} admits a unique stationary distribution with the density $\pi^{\varepsilon}(\theta,r) \propto e^{-(U(\theta)+\frac{1}{2}|r|^{2})/\varepsilon}$ \cite{Eberle}, whose $\theta$-marginal distribution is exactly \eqref{eqn:target},
which is the stationary distribution of the overdamped Langevin diffusion \eqref{eq:overdamped-2}.
It is known that the underdamped Langevin diffusion \eqref{eqn:underdamped} might converge to the target distribution \eqref{eqn:target} faster than the overdamped Langevin diffusion \eqref{eq:overdamped-2} \cite{Eberle,JianfengLu}.
When $\eps>0$ is small, 
the $\theta$-marginal of the invariant distribution concentrates around the global minima of $U$, and the stochastic algorithms based on the discretizations of underdamped Langevin diffusion \eqref{eqn:underdamped} have been used
to obtain non-asymptotic global convergence guarantees for non-convex optimization that outperform
the overdamped Langevin-based 
algorithms \cite{GGZ,Chau2022}.

For the same double-well geometry, define the phase-space hitting time
for the underdamped Langevin diffusion \eqref{eqn:underdamped}:
\[
\tau_\eps^{\rm ud}
:=\inf\{t\ge0:(\theta_t,r_t)\in B((s,0),\eps)\}.
\]
Lee--Ramil--Seo~\cite[Theorem~2.6]{LeeRamilSeo2026} proved that as $\eps\downarrow 0$,
\begin{equation}
\label{eq:underdamped-EK-intro}
\E_{(m,0)}\left[\tau_\eps^{\rm ud}\right]
=
[1+o_\eps(1)]\,
\frac{2\pi}{\nu_\sigma}
\sqrt{\frac{-\det\nabla^2U(\sigma)}{\det\nabla^2U(m)}}
\exp\!\left(\frac{U(\sigma)-U(m)}{\eps}\right).
\end{equation}
Here
$\nu_\sigma=(\sqrt{\gamma^2+4\lambda_\sigma}-\gamma)/2>0$;
equivalently, it is the unstable saddle rate satisfying
$\nu_\sigma(\nu_\sigma+\gamma)=\lambda_\sigma$. Hence, the overdamped and
underdamped laws have the same Arrhenius exponent and Hessian factor, but the
unstable rate $\lambda_\sigma$ in \eqref{eq:overdamped-EK-intro} is replaced
by $\nu_\sigma$ in \eqref{eq:underdamped-EK-intro}.
The formula \eqref{eq:underdamped-EK-intro}
was suggested by Remark 5.2 in \cite{BouchetReygner2016}
and also discussed in \cite{GGZ2}.

The present paper studies the next member of this hierarchy: the \textit{third-order Langevin diffusion} \cite{MouMaWainwrightBartlettJordan2021}:
\begin{equation}
\label{eq:third-order}
\begin{cases}
d\theta_t = p_t\,dt,\\
dp_t = -L^{-1}\nabla U(\theta_t)\,dt+\gamma r_t\,dt,\\
dr_t = -\gamma p_t\,dt-\gamma r_t\,dt+\sqrt{2\gamma\eps/L}\,dB_t,
\end{cases}
\end{equation}
where $\theta_t,p_t,r_t\in\R^d$, $L>0$, $\gamma>0$, $\varepsilon>0$, and $B_t$ is a standard $d$-dimensional Brownian motion.  The general third-order family may use separate coupling and thermostat-friction parameters.  In the notation of \cite{MouMaWainwrightBartlettJordan2021}, these are $\gamma$ and $\xi$, respectively.  We adopt the one-parameter convention of \cite{DangGurbuzbalabanIslamYaoZhu2025} and set $\xi=\gamma$, so that $\gamma$ is the dissipative parameter, exactly as in the underdamped Langevin diffusion \eqref{eqn:underdamped}.  This diffusion is motivated by high-order Langevin dynamics introduced for accelerated sampling algorithms \cite{MouMaWainwrightBartlettJordan2021}.
The construction has since been extended to general higher-order Langevin
Monte Carlo schemes \cite{DangGurbuzbalabanIslamYaoZhu2025,high-order-Liu-2025},
while related generalized Langevin diffusions have also been studied by
contraction methods \cite{Monmarche2023AlmostSure}.
The weak-capacity framework for the third-order diffusion \eqref{eq:third-order}, including the
capacity--hitting and weak test-function identities in
Propositions~\ref{prop:hitting-identity}
and~\ref{prop:weak-test-function-capacity}, was subsequently established in
\cite{HeLiWangZhu2026WeakEquilibrium}.

The infinitesimal generator of \eqref{eq:third-order} is given by
\begin{equation}
\label{eq:generator}
\calL_\eps f
=p\cdot\nabla_\theta f
+\left(-L^{-1}\nabla U(\theta)+\gamma r\right)\cdot\nabla_p f
+(-\gamma p-\gamma r)\cdot\nabla_r f
+\frac{\gamma\eps}{L}\Delta_r f.
\end{equation}
Moreover, for the third-order Langevin diffusion \eqref{eq:third-order}, its Hamiltonian is
\begin{equation}
\label{eq:Hamiltonian}
H(\theta,p,r)
=
U(\theta)+\frac L2|p|^2+\frac L2|r|^2,
\end{equation}
and its invariant density is
\begin{equation}
\label{eq:invariant-density}
\pi_\eps(dz)
=
Z_\eps^{-1}e^{-H(z)/\eps}\,dz,
\qquad z=(\theta,p,r)\in\R^{3d}.
\end{equation}

We work with a double-well potential having two non-degenerate local minima
$m,s\in\R^d$ and a unique index-one saddle $\sigma\in\R^d$ realizing the
communication height between them; the precise hypotheses are given in
Assumption~\ref{ass:double-well}.  Their
phase-space lifts are
\[
x_m:=(m,0,0),\qquad x_s:=(s,0,0),\qquad z_\sigma:=(\sigma,0,0),
\]
as illustrated in Figure~\ref{fig:ek-metastable-geometry}.

Let $Z^{\eps,*}$ denote the Markov diffusion whose generator is the
$L^2(\pi_\eps)$-adjoint $\calL_\eps^*$ of $\calL_\eps$.  The measure
$\pi_\eps$ is invariant for both $\calL_\eps$ and $\calL_\eps^*$; thus the
adjoint diffusion is stationary when it is initialized with law $\pi_\eps$.
For a Borel set
$D\subset\R^{3d}$, we write
\[
\tau_D:=\inf\{t\ge0:Z_t^\eps\in D\},
\qquad
\tau_D^*:=\inf\{t\ge0:Z_t^{\eps,*}\in D\},
\]
for the forward and adjoint hitting times, respectively.

The third-order Langevin diffusion \eqref{eq:third-order} is more degenerate than the underdamped Langevin diffusion \eqref{eqn:underdamped}.  The Brownian noise acts only in the $r$-coordinate, and randomness propagates along the chain $r\to p\to\theta$.  Thus the generator is hypoelliptic in the sense of H\"ormander but not elliptic \cite{Hormander1967}.  For metastable sets given by phase-space balls around $(m,0,0)$ and $(s,0,0)$, the boundary contains characteristic points where the diffusion direction is tangent to the boundary.  This characteristic geometry motivates the weak-equilibrium formulation of \cite{LeeRamilSeo2026,HeLiWangZhu2026WeakEquilibrium}.
Related probabilistic foundations for absorbed underdamped Langevin processes,
including boundary behavior, killed transition densities, and quasi-stationary
distributions, were developed in
\cite{LelievreRamilReygner2022KFP,LelievreRamilReygner2022QSD}; the associated
Hill relation for transition statistics is analyzed in
\cite{LelievreRamilReygner2022Hill}.

The related work \cite{HeLiWangZhu2026WeakEquilibrium} constructs weak
equilibrium measures and weak capacities for \eqref{eq:third-order}, using
elliptic regularization on bounded exhaustions.  It provides the weak
capacity--hitting identity and weak test-function representation used here.
The purpose of the present paper is to prove the low-temperature ingredients
for the third-order model under double-well communication geometry:
Hamiltonian sublevel localization, saddle test functions, non-central
negligibility, mass localization, intrawell flatness, and the Gaussian saddle
current.  Starting from the weak-capacity identity established in
\cite{HeLiWangZhu2026WeakEquilibrium}, we then perform the third-order saddle
computation directly.

The resulting transition law has the same exponential (Arrhenius) scale
$\exp\{(U(\sigma)-U(m))/\eps\}$ as the overdamped and underdamped formulas.
The difference lies in the prefactor.  The third-order dynamics changes the
unstable rate at the saddle: if $-\lambda_\sigma<0$ is the unique negative
eigenvalue of $\nabla^2U(\sigma)$, the relevant unstable rate $\mu_\sigma$ is
the unique positive root of
\begin{equation}
\label{eq:cubic-intro}
\mu^3+\gamma\mu^2+\left(\gamma^2-\frac{\lambda_\sigma}{L}\right)\mu
-\frac{\gamma\lambda_\sigma}{L}=0,
\end{equation}
obtained by linearizing \eqref{eq:third-order} at the saddle in the unstable
direction (Subsection~\ref{sec:saddle-reduction}).
As quantified in Remark~\ref{rem:prefactor-comparison}, comparison with the
underdamped dynamics using the same kinetic scaling $L$ shows that the
third-order prefactor is strictly smaller for every
$L,\gamma,\lambda_\sigma>0$.

\subsubsection*{Reduction strategy}
The reduction is potential-theoretic.  Let $A_\eps$ be a small phase-space
ball around $(m,0,0)$ and $S_\eps$ a small phase-space ball around $(s,0,0)$.
Figure~\ref{fig:ek-metastable-geometry} summarizes the relation between the
double-well landscape and these lifted phase-space sets.
\begin{figure}[t]
    \centering
    \begin{tikzpicture}[x=1cm,y=1cm,>=Latex]
        \begin{scope}[shift={(-4.1,-1.30)},x=1.05cm,y=1cm]
            \fill[gray!10]
                (-2.35,-0.05) --
                plot[smooth,tension=0.7] coordinates {
                    (-2.35,2.35) (-1.85,0.95) (-1.25,0.20)
                    (-0.65,0.65) (0,1.50) (0.65,0.55)
                    (1.25,0.08) (1.85,0.85) (2.35,2.25)
                }
                -- (2.35,-0.05) -- cycle;
            \draw[very thick]
                plot[smooth,tension=0.7] coordinates {
                    (-2.35,2.35) (-1.85,0.95) (-1.25,0.20)
                    (-0.65,0.65) (0,1.50) (0.65,0.55)
                    (1.25,0.08) (1.85,0.85) (2.35,2.25)
                };
            \draw[-{Latex[length=1.8mm]},thin]
                (-2.55,-0.05) -- (2.55,-0.05) node[right] {$\theta$};
            \draw[-{Latex[length=1.8mm]},thin]
                (-2.55,-0.05) -- (-2.55,2.65) node[above] {$U(\theta)$};

            \fill[blue!65!black] (-1.25,0.20) circle (2pt);
            \fill (0,1.50) circle (2pt);
            \fill[red!70!black] (1.25,0.08) circle (2pt);
            \node[blue!65!black,below=2pt] at (-1.25,0.20) {$m$};
            \node[above=2pt] at (0,1.50) {$\sigma$};
            \node[red!70!black,below=2pt] at (1.25,0.08) {$s$};

            \draw[-{Latex[length=2mm]},black!65,thick,dashed]
                (-1.0,0.38) .. controls (-0.65,0.72) and (-0.35,1.25)
                .. (-0.08,1.45);
            \draw[-{Latex[length=2mm]},black!65,thick,dashed]
                (0.08,1.45) .. controls (0.35,1.18) and (0.7,0.55)
                .. (1.02,0.22);
        \end{scope}

        \begin{scope}[shift={(3.45,0)},x=0.95cm,y=0.95cm]
            \filldraw[fill=blue!9,draw=blue!45!black,thick]
                plot[smooth cycle,tension=0.75] coordinates {
                    (-3.05,0) (-2.65,1.05) (-1.65,1.35)
                    (-0.72,0.92) (-0.20,0.28) (-0.08,0)
                    (-0.20,-0.28) (-0.72,-0.92) (-1.65,-1.35)
                    (-2.65,-1.05)
                };
            \filldraw[fill=red!9,draw=red!55!black,thick]
                plot[smooth cycle,tension=0.75] coordinates {
                    (3.05,0) (2.65,1.05) (1.65,1.35)
                    (0.72,0.92) (0.20,0.28) (0.08,0)
                    (0.20,-0.28) (0.72,-0.92) (1.65,-1.35)
                    (2.65,-1.05)
                };

            \shade[ball color=blue!30,opacity=0.92]
                (-1.75,-0.12) ellipse (0.48 and 0.38);
            \draw[blue!65!black,thick]
                (-1.75,-0.12) ellipse (0.48 and 0.38);
            \shade[ball color=red!32,opacity=0.92]
                (1.75,-0.12) ellipse (0.48 and 0.38);
            \draw[red!70!black,thick]
                (1.75,-0.12) ellipse (0.48 and 0.38);

            \node[blue!65!black] at (-1.75,-0.12) {$A_\eps$};
            \node[blue!65!black,font=\scriptsize]
                at (-1.75,-0.63) {$x_m=(m,0,0)$};
            \node[red!70!black] at (1.75,-0.12) {$S_\eps$};
            \node[red!70!black,font=\scriptsize]
                at (1.75,-0.63) {$x_s=(s,0,0)$};
            \node[blue!55!black] at (-1.65,0.92) {$W_m$};
            \node[red!60!black] at (1.65,0.92) {$W_s$};

            \fill (0,0) circle (2pt);
            \node[font=\scriptsize,fill=white,fill opacity=0.75,
                  text opacity=1,inner sep=1pt] at (0,0.48)
                {$z_\sigma=(\sigma,0,0)$};
            \draw[-{Latex[length=2.2mm]},black!70,thick]
                (-1.24,-0.08) .. controls (-0.75,0.25) and (-0.35,0.25)
                .. (0,0) .. controls (0.35,-0.25) and (0.75,-0.25)
                .. (1.24,-0.08);

            \coordinate (O) at (0,-1.20);
            \draw[-{Latex[length=1.5mm]},thin] (O) -- ++(0.72,0)
                node[right,font=\scriptsize] {$\theta$};
            \draw[-{Latex[length=1.5mm]},thin] (O) -- ++(-0.38,0.24)
                node[above left=-1pt,font=\scriptsize] {$p$};
            \draw[-{Latex[length=1.5mm]},thin] (O) -- ++(0,0.65)
                node[above,font=\scriptsize] {$r$};
        \end{scope}

        \node[font=\small] at (-4.1,-1.83)
            {(a) Double-well potential landscape};
        \node[font=\small] at (3.45,-1.83)
            {(b) Lifted phase-space geometry};
    \end{tikzpicture}
    \caption{Schematic double-well and phase-space metastable geometry.
    Panel (a) shows the two minima $m,s$ and the unique index-one saddle
    $\sigma$ along a representative transition section of $U$.  Panel (b)
    shows the shrinking balls $A_\eps$ and $S_\eps$ inside the Hamiltonian
    sublevel components $W_m$ and $W_s$; the lifted saddle
    $z_\sigma=(\sigma,0,0)$ is the transition bottleneck.  The right panel is
    a three-dimensional projection of the full space $\mathbb R^{3d}$.
    }
    \label{fig:ek-metastable-geometry}
\end{figure}
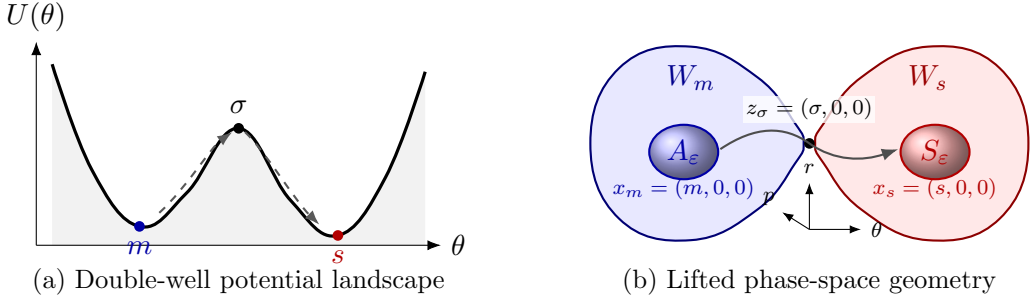

The weak capacity--hitting identity
\cite[Proposition~2.49]
{HeLiWangZhu2026WeakEquilibrium}, restated as
Proposition~\ref{prop:hitting-identity}, yields that
\begin{equation}
\label{eq:capacity-hitting-intro}
\int_{\partial A_\eps} \E_z[\tau_{S_\eps}]\,\nu_\eps(dz)
=
\frac{1}{\Cap_\eps(A_\eps,S_\eps)}
\int_{\R^{3d}} h_\eps^*(z)\,\pi_\eps(dz),
\end{equation}
where $h_\eps^*$ is the adjoint committor and $\nu_\eps$ is the normalized weak
equilibrium measure on $\partial A_\eps$.  The weak capacity in the denominator
is evaluated through smooth admissible test functions.  The Eyring--Kramers proof is organized around
the two sharp integral asymptotics
\begin{align}
\label{eq:numerator-asymptotic-intro}
\int_{\R^{3d}} h_\eps^*(z)\,\pi_\eps(dz)
&=
[1+o_\eps(1)]
\frac{1}{Z_\eps}
\frac{(2\pi\eps)^{3d/2}}{L^d\sqrt{\det\nabla^2U(m)}}
e^{-U(m)/\eps},\\
\label{eq:capacity-asymptotic-intro}
\Cap_\eps(A_\eps,S_\eps)
&=
[1+o_\eps(1)]
\frac{1}{Z_\eps}
\frac{(2\pi\eps)^{3d/2}}{L^d}
\frac{\mu_\sigma}{2\pi\sqrt{-\det\nabla^2U(\sigma)}}
e^{-U(\sigma)/\eps},
\end{align}
as $\varepsilon\downarrow 0$.
The first is proved by Hamiltonian sublevel committor localization and Laplace
asymptotics near $m$ (Subsection~\ref{sec:laplace-min}).  The second is proved by
inserting a saddle-adapted smooth test function in the weak capacity
representation and reducing the integral to a Gaussian saddle current
(Subsections~\ref{sec:saddle-reduction}--\ref{sec:gaussian-current}).  The proof
uses the linear coordinate $\Lambda_\sigma$ for the original saddle profile and
current, and patches the saddle box along physical side faces controlled by the
sublevel-component estimates.

\subsubsection*{Contributions}

The contributions of this paper can be summarized as follows.

\begin{itemize}
    \item We identify the complete linear saddle structure needed by the
    weak-capacity problem.  In particular, the Eyring--Kramers rate is the
    unique positive eigenvalue $\mu_\sigma$ of the unstable triple,
    equivalently the positive root of a cubic equation, and we evaluate the
    associated constrained Gaussian current explicitly.  Under the additional
    global hypotheses of \cite{BonyLePeutrecMichel2025}, the same
    $\mu_\sigma$ is already encoded in the whole-space spectral prefactor;
    our calculation supplies the saddle profile and flux needed for the
    directed mean hitting-time problem under the present assumptions.
    \item We construct a compactly supported global saddle test function for
    the weak capacity.  Hamiltonian sublevel components and physical
    side-layer collapse reduce its weak-capacity integral to the local saddle
    current through smooth test-function identities.
    \item We establish the well localization and intrawell flatness estimates
    needed for the numerator and for replacing the weak-equilibrium boundary
    average by the mean transition time.  The proof combines a local
    hypocoercive Lyapunov function with third-order heat-kernel, Malliavin, and
    uniform Harnack estimates.
\end{itemize}

\subsubsection*{Technical novelty relative to the underdamped case}
The underdamped Langevin diffusion \eqref{eqn:underdamped} has a two-level noise-propagation chain, whereas the
present dynamics is a third-order chain $r\to p\to\theta$.  This additional
level changes the saddle algebra from a quadratic to the cubic equation
\eqref{eq:cubic-intro}, produces the anisotropic short-time scales
$t^{1/2},t^{3/2},t^{5/2}$, and requires uniform third-order density and
Malliavin-covariance estimates.  It also changes the global capacity
construction: the local saddle profile must be patched to the two Hamiltonian
sublevel components through physical side layers, and the resulting current
must be transported to a central linear section before the Gaussian
calculation applies.  These steps form the third-order extension of the
two-level underdamped construction.

\subsubsection*{Organization}
Section~\ref{sec:main} presents the assumptions, the main theorem, the
weak-capacity results from \cite{HeLiWangZhu2026WeakEquilibrium}, and the
reduction to three sharp estimates.  Section~\ref{sec:saddle-capacity}
derives the saddle capacity asymptotic, from the unstable linear coordinate to
the Gaussian current.  Section~\ref{sec:well-proof} proves the well-mass and
intrawell-flatness estimates and then combines them with the capacity
asymptotic to prove the main theorem.  Section~\ref{sec:numerics} gives a
numerical illustration, and Section~\ref{sec:conclude} concludes the paper.

The technical arguments are organized in five appendices.  Appendix~\ref{app:weak-capacity-verification}
verifies the regularization interface, Appendix~\ref{app:local-saddle-analysis}
contains the local saddle algebra, Appendix~\ref{app:global-saddle-test}
constructs the global saddle test function, Appendix~\ref{app:well-localization}
develops the well and committor estimates, and Appendix~\ref{app:prefactor-numerics}
contains prefactor expansions and complete numerical details.
\section{Setting, Main Result, and Weak-Capacity Framework}
\label{sec:main}

\subsection{Assumptions}

Throughout, $Z_t^\eps=(\theta_t^\eps,p_t^\eps,r_t^\eps)$ denotes the solution
of \eqref{eq:third-order}, and $\calL_\eps$ its infinitesimal generator.
The weak-capacity construction uses auxiliary elliptic regularizations
of the forward and adjoint dynamics.  Their definition, non-explosion, and
uniform local-moment bounds are verified in Appendix~\ref{app:weak-capacity-verification}.  We
state here only the assumptions on the original third-order process.

Next, we introduce the assumptions that will be used throughout the rest of the paper.

\begin{assumption}[Double-well communication geometry]
\label{ass:double-well}
The potential $U\in C^\infty(\R^d)$ is a Morse function with exactly two local
minima $m,s\in\R^d$.  Let
\[
\Gamma(m,s)
:=\left\{\omega\in C([0,1];\R^d):\omega(0)=m,\ \omega(1)=s\right\},
\]
and define the communication height
\[
h_{m,s}:=\inf_{\omega\in\Gamma(m,s)}\max_{t\in[0,1]}U(\omega(t)).
\]
We assume that $h_{m,s}=U(\sigma)>\max\{U(m),U(s)\}$ and that the height is
realized by a unique index-one saddle point $\sigma\in\R^d$.  Write
$H_m=\nabla^2U(m)$ and $H_\sigma=\nabla^2U(\sigma)$.  Then
$H_m$ is positive definite, and $H_\sigma$ has exactly one negative eigenvalue
$-\lambda_\sigma<0$ and $d-1$ positive eigenvalues.  Let
$\Omega_m,\Omega_s$ be the components of $\{U<U(\sigma)\}$ containing $m,s$.
We assume
\[
\operatorname{Crit}(U)\cap\Omega_i=\{i\},
\qquad i\in\{m,s\},
\]
that $\sigma\in\partial\Omega_m\cap\partial\Omega_s$, and that the only
critical point of $U$ in
$\partial\Omega_i\cap\{U=U(\sigma)\}$ is $\sigma$.
\end{assumption}

\begin{assumption}[Confinement and weak-capacity growth conditions]
\label{ass:growth-lyapunov}
The following conditions are imposed.
\begin{description}[leftmargin=3.2em,labelwidth=2.6em]
\item[\textup{(WC1)}] \label{ass:wc-confinement}
The potential $U$ is bounded below and, after adding a constant, $U\ge0$, with
compact sublevel sets.  There are $c_0,R_0>0$ such that
\[
\theta\cdot\nabla U(\theta)
\ge c_0\left(|\theta|^2+U(\theta)\right),
\qquad |\theta|\ge R_0,
\]
and the force and all its derivatives grow at most polynomially.  In addition,
there exists $\beta>0$ such that the one-sided Laplacian control
\[
\liminf_{|\theta|\to\infty}
\left(|\nabla U(\theta)|-\beta\Delta U(\theta)\right)>0
\]
holds.  This condition controls the It\^o correction created by the auxiliary
$\theta$-noise in the auxiliary regularization in Appendix~\ref{app:weak-capacity-verification}; compare
\cite[Assumption~2.3]{LeeRamilSeo2026}.

\item[\textup{(WC2)}] \label{ass:wc-whole-space-growth}
The whole-space growth condition holds:
\[
\frac{|\nabla U(\theta)|^2}
{(1+U(\theta))(1+\theta\cdot\nabla U(\theta))}
\longrightarrow0,
\qquad |\theta|\longrightarrow\infty.
\]
\end{description}
\end{assumption}

\begin{remark}[Regularization interface]
The regularized well-posedness and local-moment hypotheses required by the
weak-capacity construction follow from \textup{(WC1)}; see Supplementary
Section~S.2.
\end{remark}

\begin{lemma}[Isolated wells and saddle contact]
\label{lem:derived-double-well-geometry}
Under Assumptions~\ref{ass:double-well} and~\ref{ass:growth-lyapunov}, for each
$i\in\{m,s\}$ and every neighborhood $N_i$ of $i$ with compact closure in
$\Omega_i$,
\begin{equation}
\label{eq:derived-no-low-energy-plateau}
\inf_{\theta\in\Omega_i\setminus N_i}U(\theta)>U(i).
\end{equation}
Moreover, every neighborhood of $\sigma$ intersects both $\Omega_m$ and
$\Omega_s$.
\end{lemma}

\begin{proof}
This elementary consequence of the standing geometry is proved in
Appendix~\ref{app:weak-capacity-verification}.
\end{proof}

\begin{definition}[Third-order unstable rate]
\label{def:unstable-rate}
Let $-\lambda_\sigma<0$ be the unique negative eigenvalue of $H_\sigma$.  Define
$\mu_\sigma>0$ as the unique positive root of
\begin{equation}
\label{eq:cubic}
\mu^3+\gamma\mu^2+\left(\gamma^2-\frac{\lambda_\sigma}{L}\right)\mu
-\frac{\gamma\lambda_\sigma}{L}=0.
\end{equation}
\end{definition}

\begin{remark}
Equation \eqref{eq:cubic} is the characteristic equation of the linearized
deterministic dynamics at the saddle in the unstable eigendirection of
$H_\sigma$; existence and uniqueness of the positive root are proved in
Lemma~\ref{lem:saddle-characteristic-polynomial}.
\end{remark}

\subsection{Notation and Conventions}
\label{sec:notation}

A generic phase-space point is denoted by
$z=(\theta,p,r)\in\R^{3d}$.  We write
\[
x_m:=(m,0,0),\qquad x_s:=(s,0,0),\qquad
z_\sigma:=(\sigma,0,0),
\]
for the lifted minima and saddle.  The symbol $B(x,r)$ denotes the open
Euclidean ball, while $\overline D$ and $\partial D$ denote the closure and
boundary of a set $D$, respectively.  Unless otherwise stated,
$i\in\{m,s\}$.

For a Borel set $D\subset\R^{3d}$, the forward and adjoint hitting times are
\[
\tau_D:=\inf\left\{t\ge0:Z_t^\eps\in D\right\},\qquad
\tau_D^*:=\inf\left\{t\ge0:Z_t^{\eps,*}\in D\right\}.
\]
The metastable balls are
\[
A_\eps:=B(x_m,\eps),\qquad S_\eps:=B(x_s,\eps).
\]
Throughout, $\eps>0$ is taken sufficiently small that
$\overline{A_\eps}$ and $\overline{S_\eps}$ are disjoint and are contained in
fixed disjoint neighborhoods of $x_m$ and $x_s$, respectively.
The forward and adjoint committors are
\[
h_\eps(z):=\Pbb_z\left(\tau_{A_\eps}<\tau_{S_\eps}\right),\qquad
h_\eps^*(z):=\Pbb_z^*\left(\tau_{A_\eps}<\tau_{S_\eps}\right),
\]
extended by $1$ on $\overline{A_\eps}$ and by $0$ on
$\overline{S_\eps}$.

We denote the sublevel set below the communication height by
\[
W:=\left\{z\in\R^{3d}:H(z)<U(\sigma)\right\},
\]
and its components containing $x_m,x_s$ by $W_m,W_s$.  The corresponding
killed domains are
\[
D_\eps^m:=W_m\setminus\overline{A_\eps},\qquad
D_\eps^s:=W_s\setminus\overline{S_\eps},
\]
or, in unified notation,
\[
D_\eps^i:=W_i\setminus\overline{B(x_i,\eps)},\qquad i\in\{m,s\}.
\]
For $\varrho\in(1/2,1]$, we also set
$G_{\eps,\varrho}:=B(x_m,\eps^\varrho)$.

The momentum flip is
\[
\Theta_*(\theta,p,r):=(\theta,-p,r).
\]
It satisfies $H\circ\Theta_*=H$ and $\Theta_*W_i=W_i$.  The symbols
$\eta_\eps$, $\Cap_\eps(A_\eps,S_\eps)$, and
$\nu_\eps=\eta_\eps/\Cap_\eps(A_\eps,S_\eps)$ denote, respectively, the
weak equilibrium measure, weak capacity, and normalized weak equilibrium
measure introduced in Subsection~\ref{sec:weak-capacity}.

Constants denoted by $c,C$ may change from line to line and are independent
of $\eps$ unless a dependence is explicitly indicated.  We write $a_\eps\asymp b_\eps$
if $c b_\eps\le a_\eps\le Cb_\eps$ for constants independent of small
$\eps$, and $o_\eps(1)$ for a quantity tending to zero as
$\eps\downarrow0$.  Any additional uniformity of an error term is stated
where it is used.

\subsection{Main Results}

In this section, we first state the main result of the paper,
which is the Eyring--Kramers law for the third-order Langevin diffusion \eqref{eq:third-order}.

\begin{theorem}[Eyring--Kramers law]
\label{thm:EK}
Under Assumptions~\ref{ass:double-well} and~\ref{ass:growth-lyapunov}, as
$\eps\downarrow0$,
\begin{equation}
\label{eq:EK-main}
\E_{(m,0,0)}[\tau_{S_\eps}]
=
[1+o_\eps(1)]
\frac{2\pi}{\mu_\sigma}
\sqrt{\frac{-\det H_\sigma}{\det H_m}}
\exp\!\left(\frac{U(\sigma)-U(m)}{\eps}\right),
\end{equation}
where $\mu_\sigma>0$ is the unique positive root of the cubic equation \eqref{eq:cubic}.
\end{theorem}

\begin{proof}
The proof will be provided in Section~\ref{sec:proof-main}.
\end{proof}

\begin{remark}
The exponential factor in \eqref{eq:EK-main} is the usual Arrhenius term.
The prefactor in \eqref{eq:EK-main} differs from
the overdamped and underdamped formulas through the unstable rate $\mu_\sigma$,
determined by the third-order linearized dynamics at the saddle.
\end{remark}

\begin{remark}[Comparison with underdamped prefactors]
\label{rem:prefactor-comparison}
For an intrinsic comparison, keep the kinetic scaling in
\eqref{eq:Hamiltonian} fixed and consider the underdamped Langevin diffusion
\[
d\theta_t=p_t\,dt,
\qquad
dp_t=-L^{-1}\nabla U(\theta_t)\,dt-\gamma p_t\,dt
      +\sqrt{2\gamma\eps/L}\,dB_t,
\]
where $\gamma,\eps,L>0$ and $B_{t}$ is a standard $d$-dimensional Brownian motion, 
and set $a_\sigma:=\lambda_\sigma/L$.  Along the unstable eigenvector of
$H_\sigma$, the linearized deterministic dynamics are
\[
    \dot q=v,
    \qquad
    \dot v=a_\sigma q-\gamma v.
\]
Substitution of $(q,v)=e^{\nu t}(q_0,v_0)$ shows that the positive unstable
saddle rate $\nu_{\sigma,L}$ satisfies
\[
\nu_{\sigma,L}(\nu_{\sigma,L}+\gamma)=a_\sigma.
\]
Equivalently,
\[
    \nu_{\sigma,L}
    =\frac{\sqrt{\gamma^2+4a_\sigma}-\gamma}{2},
\]
which is the unstable rate appearing in the underdamped Eyring--Kramers
prefactor; see, for example, \cite[Theorem~2.6]{LeeRamilSeo2026} (with the
corresponding normalization of the kinetic energy).
The third-order rate $\mu_\sigma$ is the unique positive zero of
\[
p_{a_\sigma}(x)
:=x^3+\gamma x^2+(\gamma^2-a_\sigma)x-\gamma a_\sigma.
\]
Using the equation for $\nu_{\sigma,L}$ gives
\[
p_{a_\sigma}(\nu_{\sigma,L})=-\gamma\nu_{\sigma,L}^2<0.
\]
Since $p_{a_\sigma}$ is negative before its unique positive zero, it follows
that
\[
\mu_\sigma>\nu_{\sigma,L},
\]
for all positive choices of $L$, $\gamma$, and $\lambda_\sigma$.  Thus, under
the matched normalization,
the third-order Eyring--Kramers prefactor is always strictly smaller.  The
comparison concerns the leading continuous-time metastable prefactors.
\end{remark}

Further small- and large-damping expansions of the prefactor improvement are
collected in Appendix~\ref{app:prefactor-numerics}.

\begin{remark}[Hamiltonian-sublevel patching]
The global patching is organized by the
Hamiltonian sublevel decomposition $W=\{H<U(\sigma)\}$, its components
$W_m,W_s$, a third-order Lee--Ramil--Seo committor estimate
(Proposition~\ref{prop:third-order-lrs-committor-localization}), and the
path-exclusion/saddle-localization argument near the unique saddle.  These
ingredients give the exponential negligibility away from the saddle.
\end{remark}

\subsection{Weak-Capacity Results}
\label{sec:weak-capacity}

We use several results established in
\cite{HeLiWangZhu2026WeakEquilibrium}.  There the metastable sets are two
phase-space balls $A,B$; we apply the whole-space results at each fixed
temperature with $A=A_\eps$ and $B=S_\eps$.  The forward and adjoint
committors $h_\eps,h_\eps^*$ are as defined in
Subsection~\ref{sec:notation}.
We recall the following fixed-temperature facts.  Their proofs, including the
bounded-domain elliptic regularization and the subsequent outer exhaustion,
are given in \cite{HeLiWangZhu2026WeakEquilibrium}.
The assumptions in the following proposition are those used in
\cite[Propositions~2.45 and~2.49]{HeLiWangZhu2026WeakEquilibrium}.

\begin{proposition}[Imported whole-space weak-capacity consequences]
\label{prop:basic-recurrence-weak-capacity-interface}
Under Assumptions~\ref{ass:double-well} and~\ref{ass:growth-lyapunov},
condition \textup{(WC1)} supplies the
confinement and coefficient control,
Appendix~\ref{app:weak-capacity-verification} verifies the required
regularization/local-moment hypothesis, and condition~\textup{(WC2)} supplies
the whole-space growth hypothesis.  Then the
forward and adjoint processes are non-explosive, admit $\pi_\eps$ as an
invariant probability measure, and return to a compact Lyapunov set in finite
mean time.  Moreover,
\[
\Pbb_z(\tau_{S_\eps}<\infty)=1,
\qquad
\Pbb_z^*(\tau_{A_\eps}\wedge\tau_{S_\eps}<\infty)=1,
\qquad z\in\R^{3d}.
\]
\end{proposition}

\emph{Weak equilibrium measure and weak capacity.}  By
\cite[Proposition~2.49]
{HeLiWangZhu2026WeakEquilibrium} there is a unique finite positive Borel
measure $\eta_\eps$ on $\partial A_\eps$ (the \emph{weak equilibrium measure})
such that, for every $\varphi\in C^2(\partial A_\eps)$ and every
$\Phi\in C_c^2(\R^{3d})$ with
$\Phi|_{\partial A_\eps}=\varphi$ and $\Phi=0$ near $\overline{S_\eps}$,
\[
\int_{\partial A_\eps}\varphi\,d\eta_\eps
=\int_{\R^{3d}}h_\eps^*(z)\left(-\calL_\eps\Phi\right)(z)\,\pi_\eps(dz),
\]
the right-hand side being independent of the compactly supported admissible
extension.
The \emph{weak capacity} is its total mass,
\[
\Cap_\eps(A_\eps,S_\eps)=\eta_\eps(\partial A_\eps)\in(0,\infty),
\]
and is defined entirely through the preceding weak measure identity.

At each fixed $\eps$, the cited weak-capacity framework yields the following
two identities.  The small-noise ingredients needed with them are established
in Propositions~\ref{prop:mass-asymptotic},
\ref{prop:capacity-asymptotic}, and~\ref{prop:local-flatness} for the
unregularized third-order process.
The following proposition 
is the normalized whole-space form of
\cite[Proposition~2.49]
{HeLiWangZhu2026WeakEquilibrium}; its bounded-domain counterpart is
\cite[Theorem~2.42 and Corollary~2.43]
{HeLiWangZhu2026WeakEquilibrium}.

\begin{proposition}[Weak capacity--hitting identity]
\label{prop:hitting-identity}
The normalized weak equilibrium measure
$\nu_\eps=\eta_\eps/\Cap_\eps(A_\eps,S_\eps)$ is a probability measure on
$\partial A_\eps$, and
\begin{equation}
\label{eq:hitting-identity}
\int_{\partial A_\eps} \E_z[\tau_{S_\eps}]\,\nu_\eps(dz)
=
\frac{1}{\Cap_\eps(A_\eps,S_\eps)}
\int_{\R^{3d}} h_\eps^*(z)\,\pi_\eps(dz).
\end{equation}
\end{proposition}

\begin{remark}
In \cite{HeLiWangZhu2026WeakEquilibrium} the identity is first proved on a
bounded domain $\calD$ with $\E_z[\tau_{S_\eps}\wedge\tau_{\partial\calD}]$ on
the left.  The whole-space form \eqref{eq:hitting-identity} follows from the
Lyapunov return-and-controlled-path accessibility argument summarized in
Proposition~\ref{prop:basic-recurrence-weak-capacity-interface}.  Together with
this argument, if $\calD_R=B(0,R)$, non-explosion gives
$\tau_{\partial\calD_R}\uparrow\infty$ as $R\to\infty$, while the
almost-sure target-hitting statements justify removal of the final time
truncation.  In particular,
$\E_z[\tau_{S_\eps}\wedge\tau_{\partial\calD_R}]
\uparrow\E_z[\tau_{S_\eps}]$.
\end{remark}

The following proposition is obtained from
\cite[Proposition~2.49]
{HeLiWangZhu2026WeakEquilibrium} by taking boundary trace
$\varphi\equiv1$.  The bounded-domain version is
\cite[Theorem~2.41]
{HeLiWangZhu2026WeakEquilibrium}.

\begin{proposition}[Weak test-function representation of capacity]
\label{prop:weak-test-function-capacity}
Let $f_\eps\in C_c^2(\R^{3d})$ satisfy $f_\eps=1$ near
$\overline{A_\eps}$ and $f_\eps=0$ near $\overline{S_\eps}$.  Then
\[
\Cap_\eps(A_\eps,S_\eps)
=\int_{\R^{3d}}h_\eps^*(-\calL_\eps f_\eps)\,\pi_\eps(dz).
\]
\end{proposition}

For the capacity computation we fix the order of limits once and for all.  Set
\begin{equation}
\label{eq:Ksigma-def}
\mathfrak K_{\sigma,\eps}
=
Z_\eps^{-1}\frac{(2\pi\eps)^{3d/2}}{L^d}e^{-U(\sigma)/\eps}.
\end{equation}

The quantity $\mathfrak K_{\sigma,\eps}$ is the natural Gibbs scale of a
$3d$-dimensional saddle neighborhood.  It contains the common exponential
and Gaussian-volume factors but not the dimensionless unstable-rate and
Hessian contribution.  In particular, Proposition~\ref{prop:quadratic-gaussian-current}
will show that
\[
\mathfrak C_{\sigma,\eps}^{\rm quad}
=\frac{\mu_\sigma}{2\pi\sqrt{-\det H_\sigma}}\,
\mathfrak K_{\sigma,\eps},
\]
and the capacity errors in
Proposition~\ref{prop:saddle-boundary-layer-reduction} and
Lemma~\ref{lem:uniform-saddle-current-stability} are measured relative to
this scale.  Whenever
an auxiliary side-layer parameter $\vartheta$ is present, the convention is to
take $\eps\downarrow0$ first and then $\vartheta\downarrow0$.

\emph{Parameter hierarchy.}
The constants used in the saddle construction are fixed in the following
order.  First choose the fixed saddle-neighborhood width $b$ and the geometric
constants $c_{\rm box}$ and $a_{\rm sf}$ from
Lemmas~\ref{lem:path-exclusion-saddle-box} and
\ref{lem:side-face-alternative}.  Next choose
$0<c_K<\min\{c_{\rm box},1\}$, record the fixed committor-loss exponent
$M_{\rm com}$ from
Proposition~\ref{prop:third-order-lrs-committor-localization}, and determine
the finitely many polynomial-loss exponents $N$ and derivative orders $k$
required by the test function.  Only then choose $K$ larger than all
corresponding thresholds $K_0(N,k)$ and large enough that the Gaussian and
energetic factors dominate those polynomial losses.  Finally fix
$\vartheta>0$, take $\eps\downarrow0$, and only afterward let
$\vartheta\downarrow0$.  All constants without an $\eps$ subscript are
independent of $\eps$; later increases of $K$ do not alter the previously
chosen geometric constants.

\begin{remark}[Reduction to the three asymptotics]
Theorem~\ref{thm:EK} follows by combining the weak capacity--hitting identity
(Proposition~\ref{prop:hitting-identity}) with three sharp estimates: the mass
asymptotic at the well (Proposition~\ref{prop:mass-asymptotic}), the capacity
asymptotic through the saddle (Proposition~\ref{prop:capacity-asymptotic}), and
the intrawell flatness of the mean transition time
(Proposition~\ref{prop:local-flatness}).  The combination is carried out in
Subsection~\ref{sec:proof-main}.
\end{remark}

\section{Saddle Capacity Asymptotic}
\label{sec:saddle-capacity}

\subsection{Linear Saddle Reduction}
\label{sec:saddle-reduction}

This subsection records the saddle facts used downstream: the
linearization and its rate $\mu_\sigma$, the splitting into one unstable and
$d-1$ stable directions, the unstable dual coordinate $\Lambda_\sigma$, and the
effective one-dimensional operator.  All matrix computations, eigenvector
formulas, and the normal form are in
Appendix~\ref{app:local-saddle-analysis}.

Let $e_\sigma\in\R^d$ be a unit eigenvector of $H_\sigma$ for the eigenvalue
$-\lambda_\sigma$, and set, in this unstable configuration direction,
\[
q_1=\langle\theta-\sigma,e_\sigma\rangle,\qquad
v_1=\langle p,e_\sigma\rangle,\qquad
w_1=\langle r,e_\sigma\rangle .
\]

The Eyring--Kramers prefactor is determined by the exponential
growth along the unique unstable saddle mode.  We therefore first identify
the corresponding linearized block and its positive eigenvalue.

\begin{lemma}[Linearization and unstable rate]
\label{lem:saddle-characteristic-polynomial}
The deterministic part of \eqref{eq:third-order}, linearized at
$(\sigma,0,0)$ and projected onto the unstable direction, is
$\dot X_1=\mathsf A_\sigma X_1$ with $X_1=(q_1,v_1,w_1)^\top$ and
\[
\mathsf A_\sigma=
\begin{pmatrix}
0&1&0\\
\lambda_\sigma/L&0&\gamma\\
0&-\gamma&-\gamma
\end{pmatrix},
\qquad
\chi_\sigma(\mu)=\det(\mu I_{3}-\mathsf A_\sigma)
=\mu^3+\gamma\mu^2+\left(\gamma^2-\frac{\lambda_\sigma}{L}\right)\mu
-\frac{\gamma\lambda_\sigma}{L}.
\]
The polynomial $\chi_\sigma$ has a unique positive root, namely the number
$\mu_\sigma$ of Definition~\ref{def:unstable-rate}.
\end{lemma}

\begin{proof}
The determinant computation and root count are given in
Appendix~\ref{app:local-saddle-analysis}.
\end{proof}

The unstable calculation must next be embedded into the full
phase space.  The following lemma shows that all remaining normal modes are
strictly stable and records the compatible quadratic normal form.

\begin{lemma}[Unstable/stable splitting]
\label{lem:full-saddle-spectrum}
In an eigenbasis of $H_\sigma$, the full $3d\times3d$ linearized deterministic
matrix block-diagonalizes into one unstable triple $\mathsf A_\sigma$ (with
the single positive eigenvalue $\mu_\sigma$) and $d-1$ stable normal triples
whose spectra lie in the open left half-plane.  Lifting the orthogonal
diagonalization of $H_\sigma$ to phase space leaves
$|p|^2+|r|^2$ and the diffusion $\Delta_r$ invariant, so the saddle generator
decomposes accordingly and the quadratic Hamiltonian becomes
\begin{equation}
\label{eq:saddle-quadratic-Hamiltonian}
H_\sigma^{\rm quad}
=
U(\sigma)-\frac{\lambda_\sigma}{2}q_1^2
+\frac12\sum_{i=2}^d\lambda_i^\sigma q_i^2
+\frac L2|v|^2+\frac L2|w|^2 ,
\qquad
-\det H_\sigma=\lambda_\sigma\prod_{i=2}^d\lambda_i^\sigma .
\end{equation}
\end{lemma}

\begin{proof}
The block diagonalization, spectral factorization, and Routh--Hurwitz
calculation are given in Appendix~\ref{app:local-saddle-analysis}.
\end{proof}

\begin{definition}[Unstable dual coordinate]
\label{def:linear-saddle-coordinate}
Let $\ell_\sigma=(\alpha_\sigma,\beta_\sigma,\zeta_\sigma)$ be the left
eigenvector of $\mathsf A_\sigma$ for $\mu_\sigma$, normalized by
$\beta_\sigma=1$, so that (Appendix~\ref{app:local-saddle-analysis})
\[
\alpha_\sigma=\frac{\lambda_\sigma}{L\mu_\sigma},\qquad
\beta_\sigma=1,\qquad
\zeta_\sigma=\frac{\gamma}{\mu_\sigma+\gamma}.
\]
Define the unstable dual coordinate
\[
\Lambda_\sigma(\theta,p,r)
=
\alpha_\sigma\langle\theta-\sigma,e_\sigma\rangle
+\langle p,e_\sigma\rangle
+\zeta_\sigma\langle r,e_\sigma\rangle
=\ell_\sigma\cdot X_1 .
\]
Along the linearized flow
$\frac{d}{dt}(\ell_\sigma\cdot X_{1,t})
=\mu_\sigma(\ell_\sigma\cdot X_{1,t})$.
\end{definition}

The left unstable eigenvector provides a scalar coordinate that
evolves autonomously under the linearized drift.  This reduces saddle
profiles depending on $\Lambda_\sigma$ to an effective one-dimensional
operator.

\begin{lemma}[Effective one-dimensional operator]
\label{lem:local-profile-ode}
For $\varphi\in C^2(\R)$ set
$\psi_\eps=\varphi(\Lambda_\sigma/\sqrt\eps)$.  Then the linearized saddle
generator $\calL_{\sigma,\eps}$ acts as
\[
\calL_{\sigma,\eps}\psi_\eps
=
\left(c_\sigma\partial_s^2+\mu_\sigma s\,\partial_s\right)\varphi(s)\Big|_{s=\Lambda_\sigma/\sqrt\eps},
\qquad
c_\sigma=\frac{\gamma}{L}\zeta_\sigma^2=\frac{\gamma^3}{L(\mu_\sigma+\gamma)^2}.
\]
\end{lemma}

\begin{proof}
The left-eigenvector and chain-rule calculation is given in
Appendix~\ref{app:local-saddle-analysis}.
\end{proof}

\begin{definition}[Normalized saddle profile]
\label{def:normalized-saddle-profile}
Let $\Phi_\sigma:\R\to(0,1)$ be the unique bounded $C^2$ solution of
$c_\sigma\Phi_\sigma''+\mu_\sigma s\Phi_\sigma'=0$,
$\Phi_\sigma(-\infty)=1$, $\Phi_\sigma(+\infty)=0$, i.e.
\[
\Phi_\sigma(s)
=
\frac{\int_s^\infty e^{-\mu_\sigma y^2/(2c_\sigma)}\,dy}
{\int_{-\infty}^\infty e^{-\mu_\sigma y^2/(2c_\sigma)}\,dy},
\qquad
\Phi_\sigma'(s)=-\sqrt{\frac{\mu_\sigma}{2\pi c_\sigma}}\,e^{-\mu_\sigma s^2/(2c_\sigma)},
\quad
\Phi_\sigma'(0)=-\sqrt{\frac{\mu_\sigma}{2\pi c_\sigma}} .
\]
\end{definition}

\begin{remark}[Linearized harmonicity]
By Lemma~\ref{lem:local-profile-ode},
$\calL_{\sigma,\eps}\Phi_\sigma(\Lambda_\sigma/\sqrt\eps)=0$.
\end{remark}

\begin{remark}[Orientation]
Depending on which side of the saddle is the $m$-well, the local test function
uses either $\Phi_\sigma(\Lambda_\sigma/\sqrt\eps)$ or its complement; the
capacity prefactor depends only on the absolute leading saddle integral and is
unaffected by this convention.
\end{remark}

\subsection{Gaussian Saddle Current and Capacity Asymptotic}
\label{sec:gaussian-current}

Write the linearized saddle generator in divergence form
\begin{align*}
\calL_{\sigma,\eps}f
&=
\eps e^{H_\sigma^{\rm quad}/\eps}\nabla\cdot\left(e^{-H_\sigma^{\rm quad}/\eps}\mathsf M\nabla f\right),
\\
\mathsf M
&=\mathsf D+\mathsf Q
=\begin{pmatrix}
0&L^{-1}I_{d}&0\\
-L^{-1}I_{d}&0&\gamma L^{-1}I_{d}\\
0&-\gamma L^{-1}I_{d}&\gamma L^{-1}I_{d}
\end{pmatrix}.
\end{align*}
with $\mathsf Q$ skew-symmetric and $\mathsf D$ carrying only the $r$-block.
With $\psi_\eps=\Phi_\sigma(\Lambda_\sigma/\sqrt\eps)$ and
$n_\sigma=\nabla\Lambda_\sigma/|\nabla\Lambda_\sigma|$, define the quadratic
saddle current through the linear section
$\Sigma_\sigma^{\rm lin}=\{\Lambda_\sigma=0\}$,
\begin{equation}
\label{eq:quad-current-def}
\mathfrak C_{\sigma,\eps}^{\rm quad}
=
\frac{1}{Z_\eps}
\int_{\Sigma_\sigma^{\rm lin}}
-\eps e^{-H_\sigma^{\rm quad}/\eps}\mathsf M\nabla\psi_\eps\cdot n_\sigma\,dS
\qquad(\text{positive orientation}).
\end{equation}

The normalized profile determines the flux across a section
transverse to the unstable direction.  The next proposition evaluates this
flux by an explicit constrained Gaussian integral.

\begin{proposition}[Gaussian saddle current]
\label{prop:quadratic-gaussian-current}
For the linearized saddle generator $\calL_{\sigma,\eps}$, the quadratic
Hamiltonian $H_\sigma^{\rm quad}$ in
\eqref{eq:saddle-quadratic-Hamiltonian}, and the current defined in
\eqref{eq:quad-current-def},
\[
\mathfrak C_{\sigma,\eps}^{\rm quad}
=
\frac{1}{Z_\eps}
\frac{(2\pi\eps)^{3d/2}}{L^d}
\frac{\mu_\sigma}{2\pi\sqrt{-\det H_\sigma}}
e^{-U(\sigma)/\eps}
=\frac{\mu_\sigma}{2\pi\sqrt{-\det H_\sigma}}\,
\mathfrak K_{\sigma,\eps},
\]
where $\mathfrak K_{\sigma,\eps}$ is defined in \eqref{eq:Ksigma-def}.
\end{proposition}

\begin{proof}
By the coarea identity proved in Appendix~\ref{app:local-saddle-analysis},
\begin{equation}
\label{eq:quadratic-current-proof-factorization}
\mathfrak C_{\sigma,\eps}^{\rm quad}
=\frac{\sqrt\eps\,c_\sigma|\Phi_\sigma'(0)|}{Z_\eps}
e^{-U(\sigma)/\eps}I_{\sigma,\eps},
\end{equation}
where
\[
I_{\sigma,\eps}
=\int_{\R^{3d}}\delta(\Lambda_\sigma)
\exp\!\left(-\frac1{2\eps}\left(
-\lambda_\sigma q_1^2+\sum_{i=2}^d\lambda_i^\sigma q_i^2
+L|v|^2+L|w|^2\right)\right)\,dq\,dv\,dw.
\]
For each stable direction $i\ge2$,
\begin{equation}
\label{eq:stable-triple-proof}
\int_{\R^3}
e^{-(\lambda_i^\sigma q_i^2+Lv_i^2+Lw_i^2)/(2\eps)}
\,dq_i\,dv_i\,dw_i
=\frac{(2\pi\eps)^{3/2}}{L\sqrt{\lambda_i^\sigma}}.
\end{equation}
For the unstable triple, the constraint
$\alpha_\sigma q_1+v_1+\zeta_\sigma w_1=0$ gives
$q_1=-(v_1+\zeta_\sigma w_1)/\alpha_\sigma$ and a Jacobian
$\alpha_\sigma^{-1}$.  Therefore,
\begin{equation}
\label{eq:unstable-triple-proof}
\begin{aligned}
&\int_{\R^3}\delta(\alpha_\sigma q_1+v_1+\zeta_\sigma w_1)
e^{-(-\lambda_\sigma q_1^2+Lv_1^2+Lw_1^2)/(2\eps)}
\,dq_1\,dv_1\,dw_1 \\
&\hspace{4em}
=\frac1{\alpha_\sigma}
\int_{\R^2}e^{-(v_1,w_1)K_\sigma(v_1,w_1)^\top/(2\eps)}
\,dv_1\,dw_1
=\frac1{\alpha_\sigma}\frac{2\pi\eps}{\sqrt{\det K_\sigma}}.
\end{aligned}
\end{equation}
The determinant identity proved in Appendix~\ref{app:local-saddle-analysis} and the definition of
$\Phi_\sigma$ imply that
\[
\det K_\sigma
=\frac{c_\sigma L^2\lambda_\sigma}{\alpha_\sigma^2\mu_\sigma},
\qquad
|\Phi_\sigma'(0)|=\sqrt{\frac{\mu_\sigma}{2\pi c_\sigma}}.
\]
Substitution into \eqref{eq:unstable-triple-proof} yields
\begin{equation}
\label{eq:unstable-current-factor-proof}
\sqrt\eps\,c_\sigma|\Phi_\sigma'(0)|
\frac1{\alpha_\sigma}\frac{2\pi\eps}{\sqrt{\det K_\sigma}}
=\frac{(2\pi\eps)^{3/2}}{L}
\frac{\mu_\sigma}{2\pi\sqrt{\lambda_\sigma}}.
\end{equation}
By multiplying \eqref{eq:unstable-current-factor-proof} by the $d-1$ factors in
\eqref{eq:stable-triple-proof} and using
$-\det H_\sigma=\lambda_\sigma\prod_{i=2}^d\lambda_i^\sigma$, we get
\[
\mathfrak C_{\sigma,\eps}^{\rm quad}
=\frac1{Z_\eps}\frac{(2\pi\eps)^{3d/2}}{L^d}
\frac{\mu_\sigma}{2\pi\sqrt{-\det H_\sigma}}
e^{-U(\sigma)/\eps},
\]
as claimed.
\end{proof}

We now state the central estimate \eqref{eq:capacity-asymptotic-intro}.

\begin{proposition}[Capacity asymptotic through the saddle]
\label{prop:capacity-asymptotic}
As $\eps\downarrow0$,
\begin{equation}
\label{eq:capacity-asymptotic}
\Cap_\eps(A_\eps,S_\eps)
=
[1+o_\eps(1)]
\frac{1}{Z_\eps}
\frac{(2\pi\eps)^{3d/2}}{L^d}
\frac{\mu_\sigma}{2\pi\sqrt{-\det H_\sigma}}
e^{-U(\sigma)/\eps}.
\end{equation}
\end{proposition}

\begin{proof}
Proposition~\ref{prop:weak-test-function-capacity-reduction} gives
\[
\Cap_\eps(A_\eps,S_\eps)
=\left[1+o_\eps(1)\right]
\mathfrak C_{\sigma,\eps}^{\rm quad}.
\]
Proposition~\ref{prop:quadratic-gaussian-current} evaluates the quadratic
current as
\[
\mathfrak C_{\sigma,\eps}^{\rm quad}
=\frac{\mu_\sigma}{2\pi\sqrt{-\det H_\sigma}}
\mathfrak K_{\sigma,\eps},
\]
and substituting \eqref{eq:Ksigma-def} gives
\eqref{eq:capacity-asymptotic}.
\end{proof}

\section{Well Asymptotics and Proof of the Main Theorem}
\label{sec:well-proof}

\subsection{Laplace Asymptotics at the Minimum}
\label{sec:laplace-min}

Choose $r_m>0$ sufficiently small that, with
\[
\mathcal W_m:=B(m,r_m)\times B(0,r_m)\times B(0,r_m),
\]
we have $\overline{\mathcal W_m}\subset W_m$.

The numerator of the capacity--hitting identity is concentrated
near $x_m$.  We begin by computing the local Gibbs mass of a fixed
neighborhood of this minimum.

\begin{lemma}[Local Laplace asymptotic at the well]
\label{lem:local-laplace-well}
As $\eps\downarrow0$,
\[
\int_{\mathcal W_m}e^{-H(z)/\eps}\,dz
=
[1+o_\eps(1)]
\frac{(2\pi\eps)^{3d/2}}{L^d\sqrt{\det H_m}}
e^{-U(m)/\eps}.
\]
\end{lemma}

\begin{proof}
Laplace's method in $\theta$ near $m$ (where
$U=U(m)+\frac12(\theta-m)^\top H_m(\theta-m)+\mathcal O(|\theta-m|^3)$,
{with $H_m=\nabla^2U(m)\in\R^{d\times d}$ positive definite}) gives the
factor $(2\pi\eps)^{d/2}/\sqrt{\det H_m}\,e^{-U(m)/\eps}$.  The two Gaussian
integrals are over fixed balls, and their complements have exponentially
small Gaussian mass.  In particular,
\[
\int_{B(0,r_m)}e^{-L|p|^2/(2\eps)}\,dp
=\left[1+\mathcal O\left(e^{-c/\eps}\right)\right]
\left(\frac{2\pi\eps}{L}\right)^{d/2},
\]
for some $c>0$, and the same estimate holds for the $r$-integral.  These two
factors produce $L^{-d}$.  Multiplying the three factors proves the claim.
\end{proof}

To pass from the local well integral to the global
committor-weighted mass, we also need an exponentially small bound for high
Hamiltonian levels.

\begin{lemma}[Global Gibbs tail]
\label{lem:global-gibbs-tail}
Fix $\eps_0>0$.  For every $h_0>\inf_{z\in\R^{3d}}H(z)$ and
$0<\eps<\eps_0$, where $\inf_{z\in\R^{3d}}H(z)=\inf_{\theta\in\R^d}U(\theta)$,
\[
\int_{\{H(z)\ge h_0\}}e^{-H(z)/\eps}\,dz
\le Z_{\eps_0}
\exp\!\left(-\left(\frac1\eps-\frac1{\eps_0}\right)h_0\right).
\]
In particular, if $h_0>U(m)$, this tail is
$o(e^{-U(m)/\eps}\eps^N)$ for every fixed $N$.
\end{lemma}

\begin{proof}
On $\{H\ge h_0\}$,
\[
e^{-H/\eps}=e^{-H/\eps_0}
e^{-(1/\eps-1/\eps_0)H}
\le e^{-H/\eps_0}
e^{-(1/\eps-1/\eps_0)h_0}.
\]
Integrating over $H\geq h_{0}$ proves the first assertion because $Z_{\eps_0}<\infty$; the fixed
gap $h_0-U(m)>0$ dominates every power of $\eps$, which proves the second assertion.
\end{proof}

The local Laplace expansion, global tail bound, and committor
localization now combine to identify the numerator in the capacity--hitting
identity.

\begin{proposition}[Mass asymptotic near the metastable well]
\label{prop:mass-asymptotic}
As $\eps\downarrow0$,
\begin{equation}
\label{eq:mass-asymptotic}
\int_{\R^{3d}} h_\eps^*(z)\,\pi_\eps(dz)
=
[1+o_\eps(1)]
\frac{1}{Z_\eps}
\frac{(2\pi\eps)^{3d/2}}{L^d\sqrt{\det H_m}}
e^{-U(m)/\eps}.
\end{equation}
\end{proposition}

\begin{proof}
Since $\overline{\mathcal W_m}\subset W_m$ and $U(s)<U(\sigma)$, choose
$\delta>0$ such that
\[
\max\left\{U(s),\sup_{z\in\overline{\mathcal W_m}}H(z)\right\}
<U(\sigma)-\delta.
\]
Every connected component of $W=\{H<U(\sigma)\}$ contains a local minimum of
$H$: indeed, compactness of the sublevel sets makes the infimum on a component
attainable below $U(\sigma)$.  The local minima of $H$ are exactly $x_m$ and
$x_s$, and the communication geometry places them in distinct components.
Thus $W=W_m\cup W_s$.  Since
$A_\delta:=\{H\le U(\sigma)-\delta\}\subset W$, we may split
$\int_{\R^{3d}}h_\eps^*(z)e^{-H(z)/\eps}\,dz$
over $A_\delta\cap W_m$, $A_\delta\cap W_s$, and $A_\delta^c$.

On $A_\delta\cap W_m$ the Lee--Ramil--Seo
localization (Proposition~\ref{prop:third-order-lrs-committor-localization})
gives 
\[
1-h_\eps^*\le C\eps^{-M_{\rm com}}e^{-\delta/\eps}, 
\]
and $H\ge U(m)$ on $W_m$. Therefore,
\[
\int_{A_\delta\cap W_m}(1-h_\eps^*)e^{-H/\eps}\,dz
\le C|A_\delta|\eps^{-M_{\rm com}}e^{-(U(m)+\delta)/\eps}
=o\left(e^{-U(m)/\eps}\eps^{3d/2}\right).
\]
Moreover,
Lemma~\ref{lem:derived-double-well-geometry} and the positive kinetic terms in
$H$ give a constant $\kappa_m>0$ such that
\[
\inf_{(A_\delta\cap W_m)\setminus\mathcal W_m}H
\ge U(m)+\kappa_m.
\]
Because $A_\delta$ is compact, the contribution of this complement is also
$o(e^{-U(m)/\eps}\eps^{3d/2})$; explicitly, it is bounded by
$|A_\delta|e^{-(U(m)+\kappa_m)/\eps}$.

On $A_\delta\cap W_s$ the localization bounds
$h_\eps^* e^{-H/\eps}\le C\eps^{-M_{\rm com}}e^{-U(\sigma)/\eps}$, and on $A_\delta^c$
one uses $0\le h_\eps^*\le1$ and
Lemma~\ref{lem:global-gibbs-tail} with
$h_0=U(\sigma)-\delta>U(m)$; both are
$o(e^{-U(m)/\eps}\eps^{3d/2})$.
Hence, as $\eps\downarrow 0$,
\[
\int_{\R^{3d}}h_\eps^*(z)e^{-H(z)/\eps}\,dz
=[1+o_\eps(1)]\int_{\mathcal W_m}e^{-H(z)/\eps}\,dz,
\]
and
Lemma~\ref{lem:local-laplace-well} with the factor $Z_\eps^{-1}$ proves
\eqref{eq:mass-asymptotic}.
\end{proof}

\subsection{Intrawell Flatness}
\label{sec:intrawell-flatness}

\subsubsection*{Whole-space weak-capacity assumptions}
Appendix~\ref{app:weak-capacity-verification} verifies the regularization and
local-moment conditions used in the bounded-domain regularization argument, while
condition~\textup{(WC2)} in
Assumption~\ref{ass:growth-lyapunov} supplies the additional Lyapunov condition
required for outer exhaustion.  Together they
allow us to apply the whole-space weak-capacity theorem of
\cite{HeLiWangZhu2026WeakEquilibrium}, as summarized in
Subsection~\ref{sec:weak-capacity}.  In that work the required global
recurrence property follows from a Lyapunov finite-mean return estimate,
explicit controlled-path accessibility, and the strong Markov property.  The
H\"ormander bracket condition supplies the local hypoelliptic smoothness used
in Appendix~\ref{app:lyapunov} and the local well-density estimates developed
in Appendix~\ref{app:well-localization}.

\subsubsection*{Flatness mechanism}
Proposition~\ref{prop:local-flatness} is proved by comparing the mean hitting
time from $G_{\eps,\varrho}=B((m,0,0),\eps^\varrho)$ with that from $x_m$ after the
well scaling $z=x_m+\sqrt\eps y$.  The rescaled mean hitting time solves a
uniformly H\"ormander Poisson equation on every fixed $y$-ball.  A uniform
Harnack--regularity estimate
(Lemma~\ref{lem:uniform-local-harnack-holder}), together with the fixed-time
lower bound (Lemma~\ref{lem:nondegenerate-mean-transition-time})
\[
\liminf_{\eps\downarrow0}\E_{x_m}[\tau_{S_\eps}]>0,
\]
yields
\eqref{eq:local-flatness}; see
Lemma~\ref{lem:intrawell-flatness-reduction}.  This follows the flatness
argument of \cite[Proposition~2.10 and Section~6]{LeeRamilSeo2026}.
The local kernel estimates and the killed-process analysis are developed
in Appendix~\ref{app:well-localization}.

We record the intrawell flatness estimate; it is proved in
Appendix~\ref{app:lyapunov} from a uniform local Harnack--regularity estimate
and the elementary fixed-time lower bound above.

\begin{proposition}[Local flatness of the mean transition time]
\label{prop:local-flatness}
For each $\varrho\in(1/2,1]$ there are $\delta,\eps_0>0$ such that, for all
$\eps\in(0,\eps_0)$ and $z\in B((m,0,0),\eps^\varrho)$,
\begin{equation}
\label{eq:local-flatness}
\E_z[\tau_{S_\eps}]
=
\E_{(m,0,0)}[\tau_{S_\eps}]\left[1+\mathcal O(\eps^\delta)\right],
\end{equation}
and consequently
$\int_{\partial A_\eps}\E_z[\tau_{S_\eps}]\,\nu_\eps(dz)
=\E_{(m,0,0)}[\tau_{S_\eps}]\,[1+o_\eps(1)]$.
\end{proposition}

\begin{proof}
The estimate \eqref{eq:local-flatness} is proved in
Appendix~\ref{app:lyapunov}.  To deduce the boundary-average conclusion, fix
$\varrho\in(1/2,1)$.  Since $\eps<\eps^\varrho$ for all sufficiently small
$\eps$, one has
\[
\partial A_\eps\subset B\left((m,0,0),\eps^\varrho\right).
\]
The error in \eqref{eq:local-flatness} is uniform on this ball.  Integrating
it against the probability measure $\nu_\eps$ on $\partial A_\eps$ proves the
stated boundary-average asymptotic.
\end{proof}

\subsection{Proof of the Main Theorem}
\label{sec:proof-main}

\begin{proof}[Proof of Theorem~\ref{thm:EK}]
We divide the proof into four steps.

\emph{Step 1 (capacity ratio).}  The weak capacity--hitting identity
\eqref{eq:hitting-identity} implies that
\[
\int_{\partial A_\eps} \E_z[\tau_{S_\eps}]\,\nu_\eps(dz)
=
\frac{\int_{\R^{3d}} h_\eps^*(z)\,\pi_\eps(dz)}{\Cap_\eps(A_\eps,S_\eps)}.
\]

\emph{Step 2 (numerator).}  Proposition~\ref{prop:mass-asymptotic} gives
\[
\int_{\R^{3d}}h_\eps^*\,d\pi_\eps
=\left[1+o_\eps(1)\right]
\frac{(2\pi\eps)^{3d/2}}{Z_\eps L^d\sqrt{\det H_m}}
e^{-U(m)/\eps},\qquad\text{as $\eps\downarrow 0$}.
\]

\emph{Step 3 (denominator).}  Proposition~\ref{prop:capacity-asymptotic}
evaluates the denominator by the Gaussian saddle current at $\sigma$:
\[
\Cap_\eps(A_\eps,S_\eps)
=\left[1+o_\eps(1)\right]
\frac{(2\pi\eps)^{3d/2}}{Z_\eps L^d}
\frac{\mu_\sigma}{2\pi\sqrt{-\det H_\sigma}}
e^{-U(\sigma)/\eps},\qquad\text{as $\eps\downarrow 0$}.
\]
The denominator is positive, so the quotient of the two relative error terms
is again $1+o_\eps(1)$.  Dividing and cancelling the common factor
$Z_\eps^{-1}(2\pi\eps)^{3d/2}L^{-d}$ gives
\[
\int_{\partial A_\eps} \E_z[\tau_{S_\eps}]\,\nu_\eps(dz)
=[1+o_\eps(1)]
\frac{2\pi}{\mu_\sigma}
\sqrt{\frac{-\det H_\sigma}{\det H_m}}
\exp\!\left(\frac{U(\sigma)-U(m)}{\eps}\right),\qquad\text{as $\eps\downarrow 0$}.
\]

\emph{Step 4 (flatness).}  Proposition~\ref{prop:local-flatness} gives
\[
\int_{\partial A_\eps}\E_z[\tau_{S_\eps}]\,\nu_\eps(dz)
=\E_{x_m}[\tau_{S_\eps}]\,[1+o_\eps(1)],\qquad\text{as $\eps\downarrow 0$}.
\]
Substituting this identity into the result of Step~3 proves
\eqref{eq:EK-main}.
\end{proof}

\section{Numerical Illustration}
\label{sec:numerics}

Consider the one-dimensional double well
\begin{equation}
\label{eq:numerical-potential}
U(\theta)=\frac14(\theta^2-1)^2,
\qquad m=-1,\quad s=1,\quad \sigma=0,
\end{equation}
and set $L=\gamma=1$.  Then $U(\sigma)-U(m)=1/4$,
$U''(m)=2$, and $U''(\sigma)=-1$.  The third-order and matched
underdamped unstable rates are
\[
\mu=0.7549,
\qquad
\nu=\frac{\sqrt5-1}{2}=0.6180,
\]
and hence
\begin{equation}
\label{eq:numerical-theory-constants}
C_3=\frac{2\pi}{\mu\sqrt2}=5.8856,
\qquad
C_2=\frac{2\pi}{\nu\sqrt2}=7.1887,
\qquad
\frac{C_3}{C_2}=\frac{\nu}{\mu}=0.8187.
\end{equation}

We simulated both diffusions from their lifted left minima using symmetric
splitting schemes with exact Ornstein--Uhlenbeck substeps.  The target was
$\{\theta\ge0.9\}$, the time step was $\Delta t=0.005$, and each model and
temperature used $20{,}000$ independent trajectories.  Figure
\ref{fig:numerical-prefactor-position} shows the Arrhenius plot, normalized
prefactors, and ratio of the two mean transition times.

\begin{figure}[t]
\centering
\includegraphics[width=\linewidth]{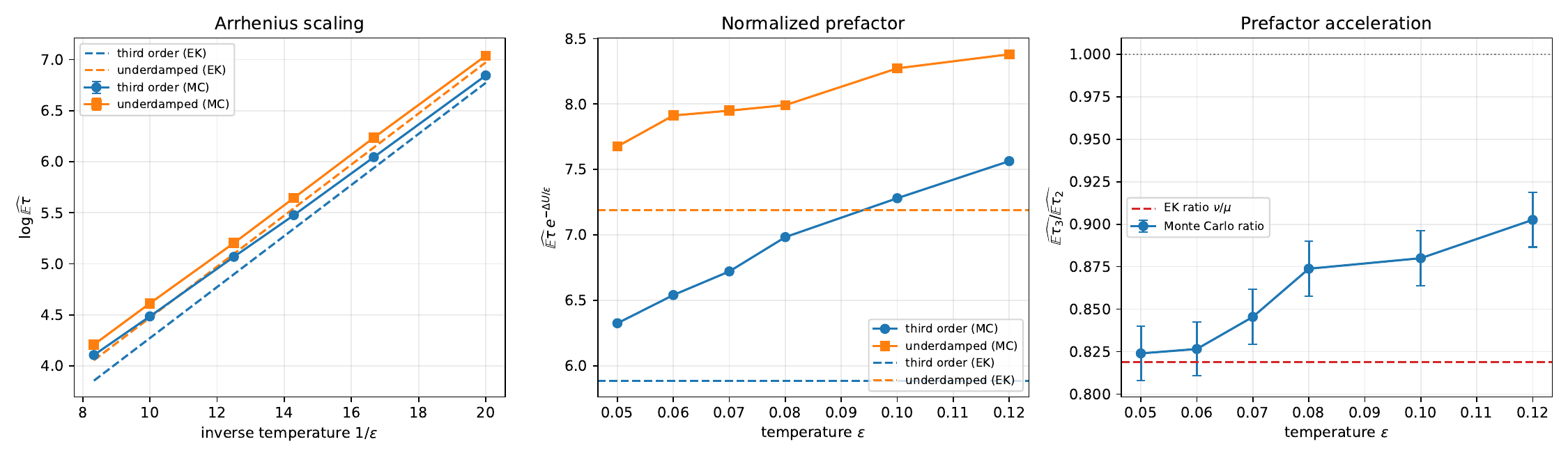}
\caption{Fixed right-well target.  Left: logarithms of the sample mean
transition times against $1/\eps$.  Center: normalized prefactors
$\overline\tau_j e^{-1/(4\eps)}$.  Right: ratio of the sample means.  Dashed
lines show the theoretical values in \eqref{eq:numerical-theory-constants};
error bars are pointwise $95\%$ confidence intervals.}
\label{fig:numerical-prefactor-position}
\end{figure}

The empirical ratio decreases toward $\nu/\mu$: it is
$0.827\;(0.811,0.842)$ at $\eps=0.06$ and
$0.824\;(0.808,0.840)$ at $\eps=0.05$, and both confidence intervals contain
the theoretical ratio $0.819$.  The full table, splitting formulas,
estimators, and step-size checks are given in Appendix~\ref{app:prefactor-numerics}.

\section{Conclusion}
\label{sec:conclude}

We established an Eyring--Kramers law for the hypoelliptic third-order
Langevin diffusion under double-well communication geometry.  For the
directed transition from $m$ to $s$, the mean hitting time has Arrhenius
factor
\[
\exp\!\left(\frac{U(\sigma)-U(m)}{\eps}\right),
\]
and sharp prefactor
\[
\frac{2\pi}{\mu_\sigma}
\sqrt{\frac{-\det H_\sigma}{\det H_m}},
\]
where $\mu_\sigma$ is the unique positive unstable rate of the third-order
linearization at the saddle, equivalently the positive root of the cubic
equation \eqref{eq:cubic}.  The result applies to either ordering of the two
well depths.
The proof combines the whole-space weak-capacity identity of
\cite{HeLiWangZhu2026WeakEquilibrium} with a saddle-adapted global test
function, an explicit Gaussian current calculation, Hamiltonian-sublevel
committor localization, and intrawell flatness.  The third-order propagation
chain $r\to p\to\theta$ appears both in the cubic saddle algebra and in the
anisotropic short-time scales underlying the local kernel and Harnack
estimates.  Under matched kinetic normalization, the third-order unstable
rate is larger than its underdamped counterpart, and hence its leading
metastable prefactor is strictly smaller.  The double-well simulations
illustrate the common Arrhenius exponent and this prefactor comparison.

\clearpage
\appendix

\section{Technical Verification for the Weak-Capacity Interface}
\label{app:weak-capacity-verification}
We record the auxiliary elliptic regularizations used by the
weak-capacity framework.  For
$\varsigma\in\{+1,-1\}$ and $0\le\delta\le1$, let
$Z_t^{\eps,\delta,\varsigma}
=\left(\theta_t^{\eps,\delta,\varsigma},p_t^{\eps,\delta,\varsigma},
r_t^{\eps,\delta,\varsigma}\right)$ solve
\begin{equation}\label{eq:ek-regularized-sde}
\begin{cases}
d\theta_t^{\eps,\delta,\varsigma}
=\left(\varsigma p_t^{\eps,\delta,\varsigma}
-\delta\nabla U\left(\theta_t^{\eps,\delta,\varsigma}\right)\right)dt
+\sqrt{2\eps\delta}\,dW_t^\theta,\\[1mm]
dp_t^{\eps,\delta,\varsigma}
=\left(-\varsigma L^{-1}\nabla U\left(\theta_t^{\eps,\delta,\varsigma}\right)
+\varsigma\gamma r_t^{\eps,\delta,\varsigma}
-\delta Lp_t^{\eps,\delta,\varsigma}\right)dt
+\sqrt{2\eps\delta}\,dW_t^p,\\[1mm]
dr_t^{\eps,\delta,\varsigma}
=\left(-\varsigma\gamma p_t^{\eps,\delta,\varsigma}
-\gamma r_t^{\eps,\delta,\varsigma}\right)dt
+\sqrt{2\gamma\eps/L}\,dB_t,
\end{cases}
\end{equation}
where $W^\theta,W^p,B$ are independent standard $d$-dimensional Brownian
motions.  We write
\[
    Z^{\eps,\delta}:=Z^{\eps,\delta,+1},
    \qquad
    Z^{\eps,\delta,*}:=Z^{\eps,\delta,-1},
\]
for the forward and adjoint regularized processes.  Their expectations when
started from $z$ are denoted by $\E_z^{\eps,\delta}$ and
$\E_z^{\eps,\delta,*}$, respectively.  Their generators are, respectively,
\begin{align}
\calL_{\eps,\delta}
&:=\calL_\eps
+\delta\left(\eps\Delta_\theta
-\nabla U(\theta)\cdot\nabla_\theta\right)
+\delta\left(\eps\Delta_p-Lp\cdot\nabla_p\right),
\label{eq:ek-forward-regularized-generator}\\
\calL_{\eps,\delta}^*
&:=\calL_\eps^*
+\delta\left(\eps\Delta_\theta
-\nabla U(\theta)\cdot\nabla_\theta\right)
+\delta\left(\eps\Delta_p-Lp\cdot\nabla_p\right),
\label{eq:ek-adjoint-regularized-generator}
\end{align}
where the second operator is the $L^2(\pi_\eps)$-adjoint of the first.
Equivalently, with
$\calL_{\eps,\delta}^{(+)}:=\calL_{\eps,\delta}$ and
$\calL_{\eps,\delta}^{(-)}:=\calL_{\eps,\delta}^*$,
\begin{equation}\label{eq:ek-signed-regularized-generator}
\begin{aligned}
\calL_{\eps,\delta}^{(\varsigma)}f
=&(\varsigma p-\delta\nabla U(\theta))\cdot\nabla_\theta f+\left(-\varsigma L^{-1}\nabla U(\theta)
    +\varsigma\gamma r-\delta Lp\right)\cdot\nabla_p f\\
&\qquad+(-\varsigma\gamma p-\gamma r)\cdot\nabla_r f
+\eps\delta(\Delta_\theta+\Delta_p)f
+\frac{\gamma\eps}{L}\Delta_r f .
\end{aligned}
\end{equation}

\begin{proposition}[Regularized well-posedness and local moments]
\label{prop:wc-regularized}
Suppose that $U\in C^2(\R^d)$ is non-negative with compact sublevel sets and
that, for some $c_0,R_0,\beta>0$,
\[
\theta\cdot\nabla U(\theta)
\ge c_0\left(|\theta|^2+U(\theta)\right),
\qquad |\theta|\ge R_0,
\]
and
\[
\liminf_{|\theta|\to\infty}
\left(|\nabla U(\theta)|-\beta\Delta U(\theta)\right)>0.
\]
For every fixed $\eps>0$ and $0\le\delta\le1$, the forward and adjoint
stochastic differential equations in \eqref{eq:ek-regularized-sde} then admit
unique non-explosive strong solutions.
Moreover, for every $T<\infty$, $q\ge1$, and compact
$K\subset\R^{3d}$,
\[
\sup_{0\le\delta\le1}\sup_{z\in K}
\E_z^{\eps,\delta}\left[
\sup_{0\le t\le T}\left(1+\left|Z_t^{\eps,\delta}\right|\right)^q
\right]<\infty,
\]
and
\[
\sup_{0\le\delta\le1}\sup_{z\in K}
\E_z^{\eps,\delta,*}\left[
\sup_{0\le t\le T}\left(1+\left|Z_t^{\eps,\delta,*}\right|\right)^q
\right]<\infty.
\]
The endpoint $\delta=0$ denotes the original degenerate forward and adjoint
processes.  Consequently, the regularization/local-moment hypothesis used in
the bounded-domain weak-capacity construction is a consequence of the
corresponding parts of \textup{(WC1)}.
\end{proposition}

\begin{proof}
Set $S=1+H$, where $H$ is defined in \eqref{eq:Hamiltonian}.  The
antisymmetric Hamiltonian part of either signed generator annihilates $H$.
Hence a direct computation from
\eqref{eq:ek-signed-regularized-generator} gives, for
$\varsigma\in\{+1,-1\}$,
\begin{equation}
\label{eq:ek-regularized-H-drift}
\calL_{\eps,\delta}^{(\varsigma)}H
=-\gamma L|r|^2-\delta|\nabla U(\theta)|^2
  -\delta L^2|p|^2
  +\gamma d\eps+\delta\eps\Delta U(\theta)+\delta\eps Ld.
\end{equation}
The one-sided Laplacian control above and $C^2$-regularity on compact sets
imply the global bound
\begin{equation}
\label{eq:ek-laplacian-gradient-bound}
\Delta U(\theta)\le \beta^{-1}|\nabla U(\theta)|+C_\beta,
\qquad \theta\in\R^d.
\end{equation}
Young's inequality therefore yields, for fixed $\eps>0$,
\begin{equation}
\label{eq:ek-gradient-trace-absorption}
\eps\Delta U(\theta)-|\nabla U(\theta)|^2
\le-\frac12|\nabla U(\theta)|^2+C_{\eps,\beta}.
\end{equation}

For an integer $m\ge1$, let $V_m=S^m$.  The chain rule and
\eqref{eq:ek-regularized-H-drift} give
\begin{align}
\calL_{\eps,\delta}^{(\varsigma)}V_m
=mS^{m-1}\calL_{\eps,\delta}^{(\varsigma)}H
+m(m-1)\eps S^{m-2}
\left(\delta|\nabla U|^2+\delta L^2|p|^2+\gamma L|r|^2\right).
\label{eq:ek-regularized-Vm-chain}
\end{align}
On the set $S\ge 4(m-1)\eps+1$, the second line is absorbed by the three
negative quadratic terms in the first line, after also using
\eqref{eq:ek-gradient-trace-absorption}.  The complementary set is compact,
because $H$ has compact sublevel sets, and all the coefficients in
\eqref{eq:ek-regularized-Vm-chain} are bounded there.  Consequently, there
are constants $C_{m,\eps}<\infty$ and $c_{m,\eps}>0$, independent of
$\delta\in[0,1]$ and $\varsigma\in\{+1,-1\}$, such that
\begin{equation}
\label{eq:ek-regularized-Vm-drift}
\calL_{\eps,\delta}^{(\varsigma)}V_m
\le C_{m,\eps}V_m
-c_{m,\eps}S^{m-1}
\left(\delta|\nabla U|^2+\delta L^2|p|^2+\gamma L|r|^2\right)
\end{equation}
outside a fixed compact set; enlarging $C_{m,\eps}$ makes the corresponding
upper bound valid on all of $\R^{3d}$.

Since $U\in C^2$, the coefficients of
\eqref{eq:ek-regularized-sde} are locally Lipschitz, so there is a unique
strong solution up to its explosion time.  Apply It\^o's formula to $V_m$
up to the exit time from a phase-space ball.  Estimate
\eqref{eq:ek-regularized-Vm-drift}, followed by localization and Gr\"onwall's
inequality, excludes explosion uniformly in $\delta$ and in the two signs.
The quadratic variation of the martingale term is controlled by
\[
2m^2\eps S^{2m-2}
\left(\delta|\nabla U|^2+\delta L^2|p|^2+\gamma L|r|^2\right)dt.
\]
The dissipative term in \eqref{eq:ek-regularized-Vm-drift}, together with the
Burkholder--Davis--Gundy and Young inequalities, therefore gives, with
$\E_z^{\eps,\delta,+1}:=\E_z^{\eps,\delta}$ and
$\E_z^{\eps,\delta,-1}:=\E_z^{\eps,\delta,*}$,
\begin{equation}
\label{eq:ek-regularized-Vm-moment}
\sup_{\substack{0\le\delta\le1,\ \varsigma\in\{+1,-1\}\\z\in K}}
\E_z^{\eps,\delta,\varsigma}
\left[\sup_{0\le t\le T}V_m
\left(Z_t^{\eps,\delta,\varsigma}\right)\right]
\le C_{m,T,K,\eps}.
\end{equation}
Finally, the radial bound above, integrated along rays, implies
$U(\theta)\ge c|\theta|^2-C$.  Thus $S$ controls
$1+|\theta|^2+|p|^2+|r|^2$.  Taking $m$ sufficiently large in
\eqref{eq:ek-regularized-Vm-moment} proves both asserted moment bounds.
\end{proof}

\begin{remark}[Global derivative growth and localization]
\label{rem:global-derivative-localization}
Proposition~\ref{prop:wc-regularized} follows from $C^2$ regularity and the
confinement and one-sided Laplacian conditions in its statement.  The polynomial growth of
higher force derivatives in \textup{(WC1)} enters through the hypotheses of
the cited weak-capacity theorem.  Locally, the density argument uses a finite
number of coefficient derivatives after a chain-compatible cutoff on a
bounded region; see Lemma~\ref{lem:finite-order-pigato}.  The finite-order
Malliavin estimates and uniform hypoelliptic regularity argument use local
smoothness above order two.  Thus the global polynomial-growth condition
serves the weak-capacity interface, while the local higher regularity serves
the density and Harnack estimates.
\end{remark}

\begin{proof}[Proof of Lemma~\ref{lem:derived-double-well-geometry}]
The second assertion follows immediately from
$\sigma\in\partial\Omega_m\cap\partial\Omega_s$.  Fix $i\in\{m,s\}$.  First,
$U\ge U(i)$ on $\Omega_i$: otherwise compactness of the sublevel sets would
give an interior minimizer of $U$ in $\Omega_i$ with value below $U(i)$,
contradicting $\operatorname{Crit}(U)\cap\Omega_i=\{i\}$.  If
\eqref{eq:derived-no-low-energy-plateau} failed, there would be
$\theta_n\in\Omega_i\setminus N_i$ with $U(\theta_n)\to U(i)$.  Compactness of
a fixed sublevel set gives, along a subsequence, $\theta_n\to\theta_*$.  Since
$U(\theta_*)=U(i)<U(\sigma)$, the limit belongs to $\Omega_i$; since $N_i$ is
open, it does not belong to $N_i$.  The preceding lower bound makes
$\theta_*$ an interior minimizer of $U$ in $\Omega_i$, hence a critical point.
Thus $\theta_*=i$, contradicting $i\in N_i$.
\end{proof}

\section{Local Saddle Analysis}
\label{app:local-saddle-analysis}

\subsection{Saddle Linear Algebra}
\label{app:linear-algebra}

We record the explicit computations behind Subsection~\ref{sec:saddle-reduction}.
In the unstable configuration direction $e_\sigma$ the linearized deterministic
dynamics is $\dot X_1=\mathsf A_\sigma X_1$, with
$X_1=(q_1,v_1,w_1)^\top$ and $\mathsf A_\sigma$ as in
Lemma~\ref{lem:saddle-characteristic-polynomial}; a direct evaluation of
$\det(\mu I_{3}-\mathsf A_\sigma)$ gives
\[
\chi_\sigma(\mu)=\mu^3+\gamma\mu^2+(\gamma^2-a_\sigma)\mu-\gamma a_\sigma,
\]
with
$a_\sigma=\lambda_\sigma/L>0$.
After omitting a possible zero coefficient, the coefficient sequence has
exactly one sign change, so Descartes' rule implies that there is at most one
positive root; since $\chi_\sigma(0)=-\gamma a_\sigma<0$ and
$\chi_\sigma(+\infty)=+\infty$, there is exactly one, namely $\mu_\sigma$.
Since $\chi_\sigma(\mu_\sigma)=0$, polynomial division implies that
\[
\chi_\sigma(x)
=(x-\mu_\sigma)
\left(x^2+(\gamma+\mu_\sigma)x
+\frac{\gamma a_\sigma}{\mu_\sigma}\right).
\]
Both coefficients of the quadratic factor are strictly positive.  If its
roots are real, their negative sum and positive product show that both are
negative; if they are non-real, their common real part is
$-(\gamma+\mu_\sigma)/2<0$.  Hence the unstable block has exactly one
eigenvalue in the open right half-plane, while its other two eigenvalues lie
in the open left half-plane.

To construct the unstable coordinate used in the saddle profile
and current, we record compatible right and left eigenvectors of the unstable
block.

\begin{proposition}[Right and left unstable eigenvectors]
\label{prop:saddle-eigenvectors}
Let $a_\sigma=\lambda_\sigma/L$ and let $\mu_\sigma>0$ be the unique positive
root of $\chi_\sigma$.  A right eigenvector of $\mathsf A_\sigma$ for
$\mu_\sigma$ is
\[
u_\sigma=\left(1,\ \mu_\sigma,\ -\frac{\gamma\mu_\sigma}{\mu_\sigma+\gamma}\right)^\top ,
\]
and a left eigenvector $\ell_\sigma=(\alpha_\sigma,\beta_\sigma,\zeta_\sigma)^\top$
with $\mathsf A_\sigma^\top\ell_\sigma=\mu_\sigma\ell_\sigma$ is given by
$\beta_\sigma=1$,
\[
\alpha_\sigma=\frac{a_\sigma}{\mu_\sigma}=\frac{\lambda_\sigma}{L\mu_\sigma},
\qquad
\zeta_\sigma=\frac{\gamma}{\mu_\sigma+\gamma}.
\]
\end{proposition}

\begin{proof}
Substituting the vectors into $\mathsf A_\sigma u_\sigma=\mu_\sigma u_\sigma$
and $\mathsf A_\sigma^\top\ell_\sigma=\mu_\sigma\ell_\sigma$ reduces the only
nontrivial component of each identity to $\chi_\sigma(\mu_\sigma)=0$.  Since
$\mu_\sigma+\gamma>0$ the normalizations are well defined.
\end{proof}

\emph{Stable spectrum and normal form.}  In the eigenbasis of $H_\sigma$ each
triple $(q_i,v_i,w_i)$ evolves independently; the unstable block is
$\mathsf A_\sigma$, and for $i\ge2$ the block has characteristic polynomial
\[
x^3+\gamma x^2+(\gamma^2+b_i)x+\gamma b_i,
\]
where
$b_i=\lambda_i^\sigma/L>0$.  Its coefficients are positive and
\[
\gamma(\gamma^2+b_i)-\gamma b_i=\gamma^3>0,
\]
so the cubic Routh--Hurwitz criterion
\cite[Chapter~XV, \S6]{Gantmacher1959} shows that it is Hurwitz.  This proves
Lemma~\ref{lem:full-saddle-spectrum}.  The lifted orthogonal transformation
\[
q=O_\sigma^\top(\theta-\sigma),
\quad v=O_\sigma^\top p,
\quad w=O_\sigma^\top r,
\]
has
Jacobian one and preserves $|p|^2+|r|^2$ and $\Delta_r$,
and hence preserves the normal form
\eqref{eq:saddle-quadratic-Hamiltonian}.

\emph{Effective one-dimensional operator.}  For
$\psi_\eps=\varphi(\Lambda_\sigma/\sqrt\eps)$, since $\Lambda_\sigma$ depends
only on $(q_1,v_1,w_1)$, denote the linearized saddle drift field by
$b_\sigma^{\rm lin}$.  Proposition~\ref{prop:saddle-eigenvectors} yields
\[
(\nabla\Lambda_\sigma)\cdot b_\sigma^{\rm lin}=\mu_\sigma\Lambda_\sigma.
\]
Write
\[
\calL_{\sigma,\eps}=b_\sigma^{\rm lin}\cdot\nabla+(\gamma\eps/L)\Delta_w.
\]
The chain rule and $b_\sigma^{\rm lin}\cdot\nabla\Lambda_\sigma
=\mu_\sigma\Lambda_\sigma$ imply that
\[
b_\sigma^{\rm lin}\cdot\nabla\psi_\eps
=\frac1{\sqrt\eps}
\varphi'\!\left(\frac{\Lambda_\sigma}{\sqrt\eps}\right)
b_\sigma^{\rm lin}\cdot\nabla\Lambda_\sigma
=\mu_\sigma\frac{\Lambda_\sigma}{\sqrt\eps}
\varphi'\!\left(\frac{\Lambda_\sigma}{\sqrt\eps}\right).
\]
Since $\Lambda_\sigma$ depends on $w$ only through
$\zeta_\sigma w_1$,
\[
\nabla_w\psi_\eps
=\frac{\zeta_\sigma}{\sqrt\eps}
\varphi'\!\left(\frac{\Lambda_\sigma}{\sqrt\eps}\right)e_1,
\qquad
\Delta_w\psi_\eps
=\frac{\zeta_\sigma^2}{\eps}
\varphi''\!\left(\frac{\Lambda_\sigma}{\sqrt\eps}\right).
\]
Consequently,
\[
\calL_{\sigma,\eps}\psi_\eps
=\mu_\sigma s\varphi'(s)+c_\sigma\varphi''(s),
\qquad
s=\frac{\Lambda_\sigma}{\sqrt\eps},
\quad c_\sigma=\frac{\gamma\zeta_\sigma^2}{L},
\]
which proves Lemma~\ref{lem:local-profile-ode}.

Imposing the constraint $\Lambda_\sigma=0$ reduces the unstable
Gaussian integral to a two-dimensional quadratic form.  We show it is positive definite and compute out its
determinant in the following lemma.

\begin{lemma}[Unstable Gaussian determinant]
\label{lem:unstable-gaussian-determinant}
Let
\[
K_\sigma=
\begin{pmatrix}
L-\frac{\lambda_\sigma}{\alpha_\sigma^2} & \frac{-\lambda_\sigma\zeta_\sigma}{\alpha_\sigma^2}\\
\frac{-\lambda_\sigma\zeta_\sigma}{\alpha_\sigma^2} & L-\frac{\lambda_\sigma\zeta_\sigma^2}{\alpha_\sigma^2}
\end{pmatrix}.
\]
Then $K_\sigma$ is positive definite and
$\det K_\sigma=\dfrac{c_\sigma L^2\lambda_\sigma}{\alpha_\sigma^2\mu_\sigma}$.
\end{lemma}

\begin{proof}
Set $a_\sigma=\lambda_\sigma/L$. 
Since $\alpha_\sigma=a_\sigma/\mu_\sigma$,
\[
\det K_\sigma
=L^2-\frac{L\lambda_\sigma}{\alpha_\sigma^2}\left(1+\zeta_\sigma^2\right)
=L^2\left\{1-\frac{\mu_\sigma^2}{a_\sigma}\left(1+\zeta_\sigma^2\right)\right\}.
\]
Using $\chi_\sigma(\mu_\sigma)=0$, equivalently rewriting
\eqref{eq:cubic} as
$a_\sigma(\mu_\sigma+\gamma)
=\mu_\sigma(\mu_\sigma^2+\gamma\mu_\sigma+\gamma^2)$,
and using $\zeta_\sigma=\gamma/(\mu_\sigma+\gamma)$,
\[
1-\frac{\mu_\sigma^2}{a_\sigma}(1+\zeta_\sigma^2)
=\frac{\gamma^3\mu_\sigma}
{a_\sigma(\mu_\sigma+\gamma)^2},
\qquad
\frac{c_\sigma\lambda_\sigma}{\alpha_\sigma^2\mu_\sigma}
=\frac{\gamma^3\mu_\sigma}
{a_\sigma(\mu_\sigma+\gamma)^2},
\]
so multiplying by $L^2$ gives the determinant.  The leading minor is
$L-\lambda_\sigma/\alpha_\sigma^2
=L\gamma^2/(\mu_\sigma^2+\gamma\mu_\sigma+\gamma^2)>0$, and the
determinant is positive, so $K_\sigma>0$.
\end{proof}

The side-face estimates require the kinetic energy forced by
the unstable constraint to dominate the negative potential curvature.  The
inequality in the following lemma provides exactly this comparison.

\begin{lemma}[Kinetic domination on the side faces]
\label{lem:side-face-kinetic-domination}
With $\alpha_\sigma=\lambda_\sigma/(L\mu_\sigma)$ and
$\zeta_\sigma=\gamma/(\mu_\sigma+\gamma)$,
we have $\dfrac{L\alpha_\sigma^2}{1+\zeta_\sigma^2}>\lambda_\sigma$.
\end{lemma}

\begin{proof}
With $a_\sigma=\lambda_\sigma/L$, the claimed inequality is equivalent to
$a_\sigma>\mu_\sigma^2(1+\zeta_\sigma^2)$.  The determinant computation in
Lemma~\ref{lem:unstable-gaussian-determinant} gives
\[
1-\frac{\mu_\sigma^2}{a_\sigma}(1+\zeta_\sigma^2)
=\frac{\gamma^3\mu_\sigma}
{a_\sigma(\mu_\sigma+\gamma)^2}>0,
\]
which is exactly the required inequality.
\end{proof}

\subsection{Saddle Scaling and the Compact Weak Interface}
\label{app:regularization}

Let $z_\sigma=(\sigma,0,0)$.  For $y=(Q,V,W)\in\R^{3d}$ set
$z=z_\sigma+\sqrt\eps\,y$ and, for a test function $f_\eps$, write its blow-up
\[
\widehat f_\eps(y):=f_\eps(z_\sigma+\sqrt\eps\,y).
\]

We now pass from physical variables to the natural
$\sqrt\eps$-scale near the saddle.  The first step is to identify the limiting
rescaled generator and its nonlinear remainder.

\begin{lemma}[Rescaled generator near the saddle]
\label{lem:rescaled-generator-saddle}
For smooth compactly supported $\widehat f$ with
$f(z_\sigma+\sqrt\eps\,y)=\widehat f(y)$, on compact subsets of $y$-space,
\[
(\calL_\eps f)(z_\sigma+\sqrt\eps\,y)
=(\calL_{\sigma,1}\widehat f)(y)+\sqrt\eps\,(\mathcal R_\eps\widehat f)(y),
\]
where
\begin{align*}
&\calL_{\sigma,1}
=V\cdot\nabla_Q-L^{-1}H_\sigma Q\cdot\nabla_V+\gamma W\cdot\nabla_V
+(-\gamma V-\gamma W)\cdot\nabla_W+\frac{\gamma}{L}\Delta_W,
\\
&\mathcal R_\eps\widehat f=\mathcal O_{\loc}\left(|Q|^2|\nabla_V\widehat f|\right).
\end{align*}
\end{lemma}

\begin{proof}
The chain rule gives
$\nabla_\theta f=\eps^{-1/2}\nabla_Q\widehat f$, and similarly for $p,r$, with
$\Delta_r f=\eps^{-1}\Delta_W\widehat f$.  Hence
\begin{align*}
&p\cdot\nabla_\theta f=V\cdot\nabla_Q\widehat f,
\\
&\gamma r\cdot\nabla_p f=\gamma W\cdot\nabla_V\widehat f,
\\
&(-\gamma p-\gamma r)\cdot\nabla_r f=(-\gamma V-\gamma W)\cdot\nabla_W\widehat f,
\\
&\frac{\gamma\eps}{L}\Delta_r f=\frac{\gamma}{L}\Delta_W\widehat f.
\end{align*}
Finally, $\nabla U(\sigma+\sqrt\eps Q)=H_\sigma\sqrt\eps Q+\mathcal O(\eps|Q|^2)$, so that
\[
-L^{-1}\nabla U\cdot\nabla_p f=-L^{-1}H_\sigma Q\cdot\nabla_V\widehat f
+\sqrt\eps\,\mathcal O_{\loc}\left(|Q|^2|\nabla_V\widehat f|\right).
\]
Adding the contributions of
$p\cdot\nabla_\theta$,
$-L^{-1}\nabla U\cdot\nabla_p$,
$\gamma r\cdot\nabla_p$,
$(-\gamma p-\gamma r)\cdot\nabla_r$, and
$(\gamma\eps/L)\Delta_r$ proves the expansion.
\end{proof}

The preceding compact-set expansion must remain valid on the
slowly growing logarithmic saddle box.  The next lemma records the required
uniform version.

\begin{lemma}[Normal-form stability under saddle blow-up]
\label{lem:normal-form-stability-saddle-blowup}
Fix $K>0$ and set $R_\eps=K\sqrt{\log(1/\eps)}$.  Let
$\widehat f$ be smooth on a neighborhood of $\overline{B_{R_\eps}}$, and
suppose that $f(z_\sigma+\sqrt\eps\,Y)=\widehat f(Y)$ there.  Then,
uniformly for $|Y|\le R_\eps$,
\[
(\calL_\eps f)\left(z_\sigma+\sqrt\eps\,Y\right)
=(\calL_{\sigma,1}\widehat f)(Y)+\mathcal O\left(\sqrt\eps R_\eps^2|\nabla_V\widehat f(Y)|\right).
\]
For the saddle profile
$j_\eps(z)=\Phi_\sigma(\Lambda_\sigma(z)/\sqrt\eps)$, set
\begin{equation}
\label{eq:normal-form-logarithmic-tube}
\mathcal T_\eps
:=
\left\{z_\sigma+\sqrt\eps\,Y:|Y|\le R_\eps\right\},
\end{equation}
and
\[
\mathcal R_\eps^{\rm nf}
:=
\int_{\mathcal T_\eps}h_\eps^*(z)
\left(\calL_\eps-\calL_{\sigma,\eps}\right)j_\eps(z)\,d\pi_\eps(z).
\]
Then
\[
|\mathcal R_\eps^{\rm nf}|
\le C\sqrt\eps R_\eps^3\mathfrak K_{\sigma,\eps}
=o_\eps(\mathfrak K_{\sigma,\eps}).
\]
\end{lemma}

\begin{proof}
This is the logarithmic-box version of
Lemma~\ref{lem:rescaled-generator-saddle}.
The Taylor remainder in Lemma~\ref{lem:rescaled-generator-saddle} implies that,
for $z=z_\sigma+\sqrt\eps Y$ with $|Y|\le R_\eps$,
\[
\left|\left(\calL_\eps-\calL_{\sigma,\eps}\right)j_\eps(z)\right|
\le C\sqrt\eps R_\eps^2
\left|\Phi_\sigma'(\Lambda_\sigma(Y))\right|.
\]
After the change of variables $z=z_\sigma+\sqrt\eps Y$, set
\[
B_\sigma(Y):=
-\frac{\lambda_\sigma}{2}Q_1^2
+\frac12\sum_{i=2}^d\lambda_i^\sigma Q_i^2
+\frac L2|V|^2+\frac L2|W|^2,
\]
and
\[
\mathscr Q_\sigma(Y)
:=B_\sigma(Y)+\frac{\mu_\sigma}{2c_\sigma}
\Lambda_\sigma(Y)^2.
\]
The form $B_\sigma$ is positive definite on $\ker\Lambda_\sigma$ by
Lemma~\ref{lem:coercivity-ker-Lambda}.  On the unstable triple, write
$D_\sigma=\operatorname{diag}(-\lambda_\sigma,L,L)$ and
$\ell_\sigma=(\alpha_\sigma,1,\zeta_\sigma)$.  The cubic identity
\eqref{eq:cubic} implies that
\[
\ell_\sigma^\top D_\sigma^{-1}\ell_\sigma
=-\frac{c_\sigma}{\mu_\sigma}.
\]
Thus the rank-one update
$D_\sigma+(\mu_\sigma/c_\sigma)\ell_\sigma\ell_\sigma^\top$ is singular.
Indeed, with $k_\sigma:=D_\sigma^{-1}\ell_\sigma$, we have
\[
\left(D_\sigma+\frac{\mu_\sigma}{c_\sigma}
\ell_\sigma\ell_\sigma^\top\right)k_\sigma
=\ell_\sigma+\frac{\mu_\sigma}{c_\sigma}\ell_\sigma
\left(\ell_\sigma^\top D_\sigma^{-1}\ell_\sigma\right)=0.
\]
Moreover,
$\ell_\sigma^\top k_\sigma=-c_\sigma/\mu_\sigma\ne0$, so every vector in
the unstable triple decomposes uniquely as $\eta+t k_\sigma$ with
$\eta\in\ker\ell_\sigma^\top$.  On this hyperplane the rank-one term
vanishes, and $D_\sigma$ is positive definite by
Lemma~\ref{lem:coercivity-ker-Lambda}.  Since the updated matrix annihilates
$k_\sigma$, the cross terms in this decomposition also vanish.  Hence the
updated matrix is positive semidefinite with kernel
$\operatorname{span}\{k_\sigma\}$.  The stable triples are positive definite,
so $\mathscr Q_\sigma$ is positive semidefinite with a one-dimensional
kernel.
Consequently, in linear coordinates $(u,\xi)\in\R\times\R^{3d-1}$ adapted to
this kernel,
\[
\mathscr Q_\sigma(Y)\ge c|\xi|^2,
\qquad
\int_{\{|Y|\le R\}}e^{-\mathscr Q_\sigma(Y)}\,dY\le CR
\quad(R\ge1).
\]
Since
$|\Phi_\sigma'(s)|=C_\sigma
e^{-\mu_\sigma s^2/(2c_\sigma)}$, $0\le h_\eps^*\le1$, and the cubic
Hamiltonian remainder is $\mathcal O(\sqrt\eps R_\eps^3)=o(1)$ uniformly on
$\mathcal T_\eps$ defined in \eqref{eq:normal-form-logarithmic-tube},
it follows that
\[
|\mathcal R_\eps^{\rm nf}|
\le C\sqrt\eps R_\eps^3\mathfrak K_{\sigma,\eps}
=o(\mathfrak K_{\sigma,\eps}).
\]
\end{proof}

The weak-capacity representation requires compactly supported
test functions.  We therefore show that compactifying the saddle ansatz
changes the integral only by an exponentially small error.

\begin{lemma}[Compact outer cutoff]
\label{lem:compact-outer-cutoff}
Let $g_{\eps,\vartheta}\in C^2(\R^{3d};[0,1])$ be any of the saddle test
functions constructed in Appendix~\ref{app:truncation}, before imposing compact
support.  There is a fixed $\Xi\in C_c^\infty(\R^{3d};[0,1])$, independent of
$\eps,\vartheta$, such that $\Xi=1$ on $\{H\le U(\sigma)+1\}$ and
\[
f_{\eps,\vartheta}:=\Xi g_{\eps,\vartheta}
\]
is admissible in Proposition~\ref{prop:weak-test-function-capacity}.  Moreover,
for every fixed $\vartheta>0$,
\begin{equation}
\label{eq:outer-cutoff-error}
\int_{\{\nabla\Xi\ne0\}}h_\eps^*
\left|\calL_\eps(\Xi g_{\eps,\vartheta})
-\Xi\calL_\eps g_{\eps,\vartheta}\right|\,d\pi_\eps
=o(\mathfrak K_{\sigma,\eps}).
\end{equation}
\end{lemma}

\begin{proof}
Compactness of the sublevels of $H$ gives a smooth cutoff equal to one on
$\{H\le U(\sigma)+1\}$ and supported in $\{H<U(\sigma)+2\}$.  The sets
$A_\eps,S_\eps$ and every saddle derivative support lie in the former set for
small $\eps$.  The product rule and the fact that diffusion acts only in the
$r$-variables give
\[
\calL_\eps(\Xi g_{\eps,\vartheta})
-\Xi\calL_\eps g_{\eps,\vartheta}
=g_{\eps,\vartheta}\calL_\eps\Xi
+\frac{2\gamma\eps}{L}
\nabla_r\Xi\cdot\nabla_rg_{\eps,\vartheta}.
\]
Let $E_\Xi=\operatorname{supp}\nabla\Xi$.  On this fixed compact shell,
$H(z)\ge U(\sigma)+1$, and for some fixed $N$,
\[
\sup_{z\in E_\Xi}\left(
|g_{\eps,\vartheta}\calL_\eps\Xi|
+\eps|\nabla_r\Xi\cdot\nabla_rg_{\eps,\vartheta}|
\right)
\le C_\vartheta\eps^{-N}.
\]
Since $0\le h_\eps^*\le1$,
\[
\int_{E_\Xi}h_\eps^*
\left|\calL_\eps(\Xi g_{\eps,\vartheta})
-\Xi\calL_\eps g_{\eps,\vartheta}\right|d\pi_\eps
\le
C_\vartheta Z_\eps^{-1}\eps^{-N}
e^{-(U(\sigma)+1)/\eps}\operatorname{Leb}(E_\Xi).
\]
After dividing by \eqref{eq:Ksigma-def}, the right-hand side in the above equation 
is bounded by
$C_\vartheta\eps^{-N-3d/2}e^{-1/\eps}$ and therefore tends to zero.
\end{proof}

\subsection{Coarea Form of the Quadratic Saddle Integral}
\label{app:coarea}

To evaluate the quadratic current, it is convenient to convert
the surface flux into an ambient Gaussian integral with a linear delta
constraint.

\begin{lemma}[Coarea form]
\label{lem:quadratic-flux-coarea}
The positive quadratic saddle current \eqref{eq:quad-current-def} satisfies
\[
\begin{aligned}
\mathfrak C_{\sigma,\eps}^{\rm quad}
&=\frac{1}{Z_\eps}\sqrt\eps\,c_\sigma|\Phi_\sigma'(0)|e^{-U(\sigma)/\eps} \\
&\quad\cdot
\int_{\R^{3d}}
\delta(\Lambda_\sigma)
\exp\!\left(-\frac1\eps\left(-\frac{\lambda_\sigma}{2}q_1^2
+\frac12\!\sum_{i=2}^d\!\lambda_i^\sigma q_i^2+\frac L2|v|^2+\frac L2|w|^2\right)\right)
dq\,dv\,dw .
\end{aligned}
\]
\end{lemma}

\begin{proof}
Since $\nabla\psi_\eps=\eps^{-1/2}\Phi_\sigma'(\Lambda_\sigma/\sqrt\eps)\nabla\Lambda_\sigma$
and $\Lambda_\sigma=0$ on $\Sigma_\sigma^{\rm lin}$, the profile derivative on
the surface is $\Phi_\sigma'(0)$. 
Since the skew part contributes zero,
\[
\nabla\Lambda_\sigma\cdot\mathsf M\nabla\Lambda_\sigma
=\nabla\Lambda_\sigma\cdot\mathsf D\nabla\Lambda_\sigma
=\frac{\gamma}{L}\zeta_\sigma^2=c_\sigma,
\]
and therefore,
\[
\mathfrak C_{\sigma,\eps}^{\rm quad}
=\frac{1}{Z_\eps}\sqrt\eps\,c_\sigma|\Phi_\sigma'(0)|
\int_{\Sigma_\sigma^{\rm lin}}
e^{-H_\sigma^{\rm quad}/\eps}\frac{1}{|\nabla\Lambda_\sigma|}\,dS .
\]
For every continuous $F:\R^{3d}\to\R$ whose restriction to
$\Sigma_\sigma^{\rm lin}$ is integrable, we use the linear coarea identity
to define
\[
\int_{\R^{3d}}F(Y)\delta(\Lambda_\sigma(Y))\,dY
:=\frac1{|\nabla\Lambda_\sigma|}
\int_{\Sigma_\sigma^{\rm lin}}F(Y)\,dS(Y),
\]
where, in our case,
\[
F(q,v,w)=
\exp\!\left[-\frac1\eps\left(-\frac{\lambda_\sigma}{2}q_1^2
+\frac12\sum_{i=2}^d\lambda_i^\sigma q_i^2
+\frac L2|v|^2+\frac L2|w|^2\right)\right].
\]
Its restriction to $\Sigma_\sigma^{\rm lin}$ is integrable because
$K_\sigma$ is positive definite by
Lemma~\ref{lem:unstable-gaussian-determinant}, while all stable Hessian
eigenvalues are positive.
This proves the claim.
\end{proof}

\section{Global Saddle Test Function and Capacity Reduction}
\label{app:global-saddle-test}

\subsection{Saddle Boxes, Truncation, and Test Functions}
\label{app:truncation}

\emph{Saddle scales and boxes.}  Let
$R_\eps=K\sqrt{\log(1/\eps)}$ ($K$ fixed, large),
$\delta_\eps=\sqrt\eps R_\eps$, and
$a_\eps=c_K\eps R_\eps^2$, where $c_K>0$ is fixed sufficiently small.
In local coordinates diagonalizing $H_\sigma$, define the physical saddle box
\begin{equation}
\label{eq:logarithmic-saddle-box}
\mathcal K_\eps:=\left\{|q_1|\le\frac{\delta_\eps}{\sqrt{\lambda_\sigma}},\
|q_i|\le\frac{2\delta_\eps}{\sqrt{\lambda_i^\sigma}}\ (i\ge2),\
|v_i|,|w_i|\le\frac{2\delta_\eps}{\sqrt L}\right\},
\end{equation}
with $m$-side face
\[
\partial\mathcal K_\eps^-
:=\partial\mathcal K_\eps\cap
\left\{q_1=-\delta_\eps/\sqrt{\lambda_\sigma}\right\},
\]
and
$s$-side face
\[
\partial\mathcal K_\eps^+
:=\partial\mathcal K_\eps\cap
\left\{q_1=+\delta_\eps/\sqrt{\lambda_\sigma}\right\}.
\]
For the auxiliary side-layer parameter $\vartheta>0$, define
\begin{equation}
\label{eq:physical-side-layers}
\begin{aligned}
\mathcal A_{\eps,\vartheta}^-
&:=
\left\{
-\frac{\delta_\eps+\vartheta\sqrt\eps}{\sqrt{\lambda_\sigma}}
\le q_1\le
-\frac{\delta_\eps}{\sqrt{\lambda_\sigma}},\
|q_i|\le\frac{2\delta_\eps}{\sqrt{\lambda_i^\sigma}}\ (i\ge2),\
|v_i|,|w_i|\le\frac{2\delta_\eps}{\sqrt L}
\right\},\\
\mathcal A_{\eps,\vartheta}^+
&:=
\left\{
\frac{\delta_\eps}{\sqrt{\lambda_\sigma}}
\le q_1\le
\frac{\delta_\eps+\vartheta\sqrt\eps}{\sqrt{\lambda_\sigma}},\
|q_i|\le\frac{2\delta_\eps}{\sqrt{\lambda_i^\sigma}}\ (i\ge2),\
|v_i|,|w_i|\le\frac{2\delta_\eps}{\sqrt L}
\right\},\\
\mathcal A_{\eps,\vartheta}^{\rm side}
&:=\mathcal A_{\eps,\vartheta}^-\cup\mathcal A_{\eps,\vartheta}^+.
\end{aligned}
\end{equation}
We fix the sign of $e_\sigma$ so that the negative face is locally connected
to $W_m$ and the positive face to $W_s$.
On $\partial\mathcal K_\eps^-$, let $n^-$ denote the unit normal pointing
toward increasing $q_1$, namely from the $m$-side layer into
$\mathcal K_\eps$.  Thus $n^-$ is the outward normal of the adjacent
$m$-side layer on its inner face; this is the positive $m$-to-$s$ current
orientation used in \eqref{eq:saddle-boundary-layer-sides} and
Lemma~\ref{lem:uniform-saddle-current-stability}.
Set $J_\eps=\{H<U(\sigma)+a_\eps\}$ with components
$J_\eps^m,J_\eps^s$ of $J_\eps\setminus\mathcal K_\eps$ containing $x_m,x_s$.
Denote the union of all remaining components by
\[
J_\eps^c
:=
\left(J_\eps\setminus\mathcal K_\eps\right)
\setminus\left(J_\eps^m\cup J_\eps^s\right).
\]
Lemma~\ref{lem:box-partition-low-sublevel} shows that
$J_\eps^c=\varnothing$ for all sufficiently small $\eps$.
Let $\widehat\ell_\sigma:=\nabla\Lambda_\sigma$ and choose the normalized quadratic
current direction
\[
\mathfrak u_\sigma:=\frac{\mathsf M\widehat\ell_\sigma}{c_\sigma}.
\]
Since
$\widehat\ell_\sigma^\top\mathsf M\widehat\ell_\sigma=c_\sigma$, we have
$\Lambda_\sigma(\mathfrak u_\sigma)=1$.  Set
$\Pi_\sigma Y=Y-\Lambda_\sigma(Y)\mathfrak u_\sigma$; write
$Y=(Q,V,W)$, $s=\Lambda_\sigma(Y)$, $\eta_s=\Pi_\sigma Y$.  For later use,
$\mathfrak K_{\sigma,\eps}$ is as in \eqref{eq:Ksigma-def}.

We now patch the local saddle profile to the constant well values
while localizing all additional derivatives to energetically negligible
regions.

\begin{lemma}[Energy-localized saddle ansatz]
\label{lem:global-saddle-test-function}
Let $\rho_\eps=\delta_\eps^3$.  There is
$G_\eps\in C_c^\infty(\R^{3d};[0,1])$ such that
\begin{equation}
\label{eq:energy-cutoff-properties}
G_\eps=1\ \hbox{ on }\{H\le U(\sigma)+a_\eps/2\},\qquad
G_\eps=0\ \hbox{ on }\{H\ge U(\sigma)+5a_\eps/6\},
\end{equation}
and $|\partial^\alpha G_\eps|\le C_\alpha\rho_\eps^{-|\alpha|}$.
Writing $j_\eps=\Phi_\sigma(\Lambda_\sigma/\sqrt\eps)$, for every $\vartheta>0$
there is an admissible $f_{\eps,\vartheta}\in C_c^2(\R^{3d};[0,1])$ which equals
$G_\eps j_\eps$ on $\mathcal K_\eps$, $G_\eps$ on the part of $J_\eps^m$
outside the $m$-side layer, and zero on the part of $J_\eps^s$ outside the
$s$-side layer.  Across the two physical faces it is
\begin{equation}
\label{eq:piecewise-energy-saddle-test}
f_{\eps,\vartheta}=
\begin{cases}
G_\eps\{j_\eps+\chi_{\eps,\vartheta}^-(1-j_\eps)\},&m\text{-side layer},\\
G_\eps j_\eps,&\mathcal K_\eps,\\
G_\eps(1-\chi_{\eps,\vartheta}^+)j_\eps,&s\text{-side layer},\\
G_\eps,&J_\eps^m\text{ away from the layer},\\
0,&J_\eps^s\text{ away from the layer or }J_\eps^c,
\end{cases}
\end{equation}
where $\chi_{\eps,\vartheta}^\pm$ are flat at both ends, equal to one at the
outer end and zero at the inner end.
Every derivative not contained in the
logarithmic saddle box $\mathcal K_\eps$ in
\eqref{eq:logarithmic-saddle-box} or a physical side layer
$\mathcal A_{\eps,\vartheta}^\pm$ in \eqref{eq:physical-side-layers} is supported in
$\{H\ge U(\sigma)+a_\eps/2\}$.
\end{lemma}

\begin{proof}
Take a non-negative $\Psi\in C_c^\infty(B(0,1))$ with $\int_{\mathbb{R}^{3d}}\Psi(y)dy=1$ and set
\begin{equation}
\label{eq:energy-cutoff-convolution}
G_\eps(z)=\rho_\eps^{-3d}\int_{\R^{3d}}
\mathbf 1_{\{H(z-y)\le U(\sigma)+2a_\eps/3\}}
\Psi(y/\rho_\eps)\,dy .
\end{equation}
Compact Hamiltonian sublevels make $G_\eps$ compactly supported.  On a fixed
compact set containing $\{H\le U(\sigma)+1\}$, $H$ is Lipschitz.  Since
$\rho_\eps=\delta_\eps^3=o(\delta_\eps^2)$ and
$a_\eps=c_K\delta_\eps^2$, for small $\eps$ one has
$|H(z-y)-H(z)|\le a_\eps/6$ whenever $|y|\le\rho_\eps$.
This proves \eqref{eq:energy-cutoff-properties}; differentiating the mollifier
proves the derivative bound.

Choose $\chi_{\eps,\vartheta}^\pm$ from a function
$\chi\in C^\infty(\R;[0,1])$ which is constant outside $(0,1)$ and whose
derivatives of every order vanish at $0,1$, rescaled to layers of width
$\vartheta\sqrt\eps$.  Lemma~\ref{lem:box-partition-low-sublevel} and the chosen
orientation identify the two prescriptions on the outer ends of the layers.
For fixed $\vartheta$, Taylor expansion on a tangential face of either layer
gives, uniformly as $\eps\downarrow0$,
\[
H-U(\sigma)
\ge
-\frac{\eps}{2}(R_\eps+\vartheta)^2
+2\eps R_\eps^2+o(\eps R_\eps^2)
\ge \eps R_\eps^2.
\]
On a tangential face of $\mathcal K_\eps$ itself, the same calculation is
stronger because $|q_1|\le\delta_\eps/\sqrt{\lambda_\sigma}$; namely,
\[
H-U(\sigma)
\ge -\frac{\eps}{2}R_\eps^2+2\eps R_\eps^2
+o(\eps R_\eps^2)
\ge \eps R_\eps^2.
\]
After decreasing $c_K$ so that $5c_K/6<1$, the cutoff $G_\eps$ therefore
vanishes in a neighborhood of every tangential face of the saddle box and
the side layers.  Thus the box, side-layer, and exterior prescriptions agree
there, together with all their derivatives.
At an inner face $\chi_{\eps,\vartheta}^\pm=0$, so both side-layer formulas
coincide with $G_\eps j_\eps$ in $\mathcal K_\eps$.  At the outer end,
$\chi_{\eps,\vartheta}^-=1$ gives $G_\eps$ on the $m$-side and
$\chi_{\eps,\vartheta}^+=1$ gives zero on the $s$-side.  Moreover,
\[
\partial_{q_1}^k\chi_{\eps,\vartheta}^\pm=0
\quad\text{at both layer endpoints},\qquad k=1,2,
\]
so the values and derivatives through order two also agree across both
physical interfaces.
Thus \eqref{eq:piecewise-energy-saddle-test} defines a $C^2$ function.  In
particular, the value on $J_\eps^m$ is $G_\eps$, rather than the bare constant
one; this makes the extension smooth across the outer energy level.  Near
$\overline{A_\eps}$, $G_\eps=1$, while the function vanishes near
$\overline{S_\eps}$.
The derivative-support assertion follows from
\eqref{eq:energy-cutoff-properties}.  The fixed cutoff $\Xi$ from
Lemma~\ref{lem:compact-outer-cutoff} satisfies
\[
\Xi=1\quad\text{on }\{H\le U(\sigma)+1\},
\]
which contains the support of $f_{\eps,\vartheta}$ for all small $\eps$ because
$a_\eps\to0$.  Hence $\Xi f_{\eps,\vartheta}=f_{\eps,\vartheta}$.
\end{proof}

The next lemma verifies that this patched ansatz is admissible
and identifies the only derivative supports that can contribute on the saddle
scale.

\begin{lemma}[Saddle localization admissibility]
\label{lem:saddle-localization-admissibility}
The ansatz $f_{\eps,\vartheta}$ of
Lemma~\ref{lem:global-saddle-test-function} is smooth, $[0,1]$-valued, equal
to one near $\overline{A_\eps}$ and zero near $\overline{S_\eps}$, with
derivatives satisfying, for suitable constants $C_\alpha,N_\alpha<\infty$,
\begin{equation}
\label{eq:global-saddle-derivative-bound}
\|\partial^\alpha f_{\eps,\vartheta}\|_\infty
\le C_\alpha\eps^{-N_\alpha}\vartheta^{-|\alpha|},
\qquad |\alpha|\le2.
\end{equation}
Apart
from the two physical side layers retained in
Lemma~\ref{lem:collapse-physical-side-layer}, whose union
$\mathcal A_{\eps,\vartheta}^{\rm side}$ is defined in
\eqref{eq:physical-side-layers}, set
\[
E_{\eps,\vartheta}^{\rm ext}
:=
\int_{(\R^{3d}\setminus\mathcal K_\eps)
\setminus\mathcal A_{\eps,\vartheta}^{\rm side}}
h_\eps^*(-\calL_\eps f_{\eps,\vartheta})\,d\pi_\eps.
\]
For every fixed $\vartheta>0$,
\begin{equation}
\label{eq:saddle-localization-fixed-theta}
\lim_{\eps\downarrow0}
\frac{|E_{\eps,\vartheta}^{\rm ext}|}{\mathfrak K_{\sigma,\eps}}=0.
\end{equation}
Moreover,
\begin{equation}
\label{eq:saddle-localization-ordered-limit}
\lim_{\vartheta\downarrow0}\limsup_{\eps\downarrow0}
\frac{|E_{\eps,\vartheta}^{\rm ext}|}{\mathfrak K_{\sigma,\eps}}=0.
\end{equation}
\end{lemma}

\begin{proof}
The derivative bounds for $G_\eps$, $j_\eps$, and the layer cutoffs give
\eqref{eq:global-saddle-derivative-bound}.  To prove the integral estimate,
split the derivative support outside $\mathcal K_\eps$ and the side
layers into
\[
E_G:=\operatorname{supp}\left(\left|\nabla G_\eps\right|+\left|\nabla^2G_\eps\right|\right),
\qquad
E_\Xi:=\operatorname{supp}\nabla\Xi.
\]
On all remaining low-energy components the function is identically one or
zero, so $\calL_\eps f_{\eps,\vartheta}=0$ there.  Since
$\rho_\eps^{-1}=(\eps^{1/2}R_\eps)^{-3}$, the derivatives on $E_G$ are bounded
by $C_\vartheta\eps^{-N}$ for some fixed $N$, while
$H\ge U(\sigma)+a_\eps/2$ on $E_G$.  Consequently,
\[
\int_{E_G}h_\eps^*|\calL_\eps f_{\eps,\vartheta}|\,d\pi_\eps
\le C_\vartheta Z_\eps^{-1}\eps^{-N}
e^{-[U(\sigma)+a_\eps/2]/\eps}
\le C_\vartheta\eps^{-N-3d/2}
e^{-a_\eps/(2\eps)}\mathfrak K_{\sigma,\eps},
\]
where
\[
\exp\left(-a_\eps/(2\eps)\right)
=\exp\left(-c_KR_\eps^2/2\right)=\eps^{c_KK^2/2}.
\]
Choose the fixed $K$ so that $c_KK^2/2>N+3d/2+1$; the normalized contribution
of $E_G$ then tends to zero.  Lemma~\ref{lem:compact-outer-cutoff} gives the
same conclusion on $E_\Xi$.  The only derivative supports not included in
this decomposition are the two physical side layers, which are excluded from
$E_{\eps,\vartheta}^{\rm ext}$ by definition.  Combining these estimates proves
\eqref{eq:saddle-localization-fixed-theta}.  Consequently,
\[
\lim_{\vartheta\downarrow0}\limsup_{\eps\downarrow0}
\frac{|E_{\eps,\vartheta}^{\rm ext}|}{\mathfrak K_{\sigma,\eps}}
=\lim_{\vartheta\downarrow0}0=0,
\]
which yields \eqref{eq:saddle-localization-ordered-limit}.
\end{proof}

We next quantify the error made by replacing the full
Hamiltonian with its quadratic Taylor approximation on the saddle box.

\begin{lemma}[Hamiltonian expansion in the saddle box]
\label{lem:saddle-hamiltonian-expansion}
For $(q,v,w)=\sqrt\eps(Q,V,W)$, with
$|Q|+|V|+|W|\le C R_\eps$ for a fixed $C<\infty$, we have
\[
\frac{H(\sigma+\sqrt\eps Q,\sqrt\eps V,\sqrt\eps W)-U(\sigma)}{\eps}
=-\frac{\lambda_\sigma}{2}Q_1^2+\frac12\sum_{i=2}^d\lambda_i^\sigma Q_i^2
+\frac L2|V|^2+\frac L2|W|^2
+\mathcal O\left(\sqrt\eps R_\eps^3\right),
\]
uniformly on this set.  In particular, the remainder is $o(1)$.
\end{lemma}

\begin{proof}
Taylor expansion gives
\[
U\left(\sigma+\sqrt\eps Q\right)=U(\sigma)+\frac\eps2 Q^\top H_\sigma Q+\mathcal O\left(\eps^{3/2}|Q|^3\right),
\]
and in the eigenbasis
\[
\frac12 Q^\top H_\sigma Q=-\frac{\lambda_\sigma}{2}Q_1^2+\frac12\sum_{i\ge2}\lambda_i^\sigma Q_i^2.
\]
The kinetic terms are exact because
\[
\frac L2|p|^2+\frac L2|r|^2
=\eps\left(\frac L2|V|^2+\frac L2|W|^2\right).
\]
Moreover, $|Q|\le C R_\eps$ on the box, and hence
\[
\frac{\mathcal O(\eps^{3/2}|Q|^3)}{\eps}
=\mathcal O\left(\sqrt\eps R_\eps^3\right)=o(1),
\]
because $R_\eps=K\sqrt{\log(1/\eps)}$.
\end{proof}

Although the saddle Hamiltonian is indefinite in the full
space, its restriction to the section $\Lambda_\sigma=0$ is coercive.  This
property controls the transverse Gaussian directions.

\begin{lemma}[Coercivity on $\ker\Lambda_\sigma$]
\label{lem:coercivity-ker-Lambda}
Let
\[
B_\sigma(Q,V,W)=-\frac{\lambda_\sigma}{2}Q_1^2+\frac12\sum_{i\ge2}\lambda_i^\sigma Q_i^2+\frac L2|V|^2+\frac L2|W|^2.
\]
Let $\mathsf B_\sigma(\cdot,\cdot)$ denote the symmetric bilinear
form associated with $B_\sigma$, so that
$B_\sigma(Y)=\mathsf B_\sigma(Y,Y)$.
There is $c_\sigma^{\rm st}>0$ with $B_\sigma(Y)\ge c_\sigma^{\rm st}|Y|^2$ on
$\ker\Lambda_\sigma$, and constants $c_1,c_2>0$ with
\begin{equation}
\label{eq:coercivity-splitting}
B_\sigma(s\mathfrak u_\sigma+\eta)\ge c_1|\eta|^2-c_2s^2,
\end{equation}
for all $s\in\R$,
$\eta\in\ker\Lambda_\sigma$.
\end{lemma}

\begin{proof}
On $\ker\Lambda_\sigma$, $Q_1=-(V_1+\zeta_\sigma W_1)/\alpha_\sigma$, and the
unstable part becomes $\frac12(V_1,W_1)K_\sigma(V_1,W_1)^\top$ with
$K_\sigma>0$ (Lemma~\ref{lem:unstable-gaussian-determinant}); the stable part is
positive definite.  Hence $B_\sigma$ is positive definite on
$\ker\Lambda_\sigma$.  To prove
\eqref{eq:coercivity-splitting},  
let $\mathsf B_\sigma$ denote the symmetric bilinear form obtained by
polarizing the quadratic form $B_\sigma$:
\[
\mathsf B_\sigma(Y_1,Y_2)
=\frac12\left[B_\sigma(Y_1+Y_2)-B_\sigma(Y_1)-B_\sigma(Y_2)\right],
\qquad B_\sigma(Y)=\mathsf B_\sigma(Y,Y).
\]
By using the expansion
\[
B_\sigma(s\mathfrak u_\sigma+\eta)=B_\sigma(\eta)+2s\mathsf B_\sigma(\mathfrak u_\sigma,\eta)+s^2B_\sigma(\mathfrak u_\sigma),
\]
using the bound $|\mathsf B_\sigma(\mathfrak u_\sigma,\eta)|\le C|\eta|$ and Young's inequality, we complete the proof.
\end{proof}

The tangential faces of the logarithmic side layers lie above the saddle
energy by a logarithmic amount.  This controls the auxiliary surface terms
created when the layers are collapsed.

\begin{lemma}[Tangential-face domination on the logarithmic layers]
\label{lem:stable-side-gaussian-domination}
Let $\Gamma_{\eps,\vartheta}^{\pm,\rm tan}$ denote the union of the tangential
faces of $\mathcal A_{\eps,\vartheta}^\pm$, excluding its inner and outer
$q_1$-faces.  For every fixed $\vartheta>0$, there is $c_{\rm tan}>0$ such that,
for all sufficiently small $\eps$,
\[
H(z)-U(\sigma)\ge c_{\rm tan}\eps R_\eps^2,
\qquad z\in
\Gamma_{\eps,\vartheta}^{-,\rm tan}\cup
\Gamma_{\eps,\vartheta}^{+,\rm tan}.
\]
Let
\[
\bar f_{\eps,\vartheta}^-
:=j_\eps+\chi_{\eps,\vartheta}^-(1-j_\eps),
\qquad
\bar f_{\eps,\vartheta}^+
:=(1-\chi_{\eps,\vartheta}^+)j_\eps.
\]
Then, for some $N<\infty$,
\[
\frac{\eps}{Z_\eps}
\sum_{\varsigma\in\{-,+\}}
\int_{\Gamma_{\eps,\vartheta}^{\varsigma,\rm tan}}
e^{-H/\eps}
\left|\mathsf M\nabla\bar f_{\eps,\vartheta}^\varsigma\right|\,dS
\le C_\vartheta R_\eps^N e^{-c_{\rm tan}R_\eps^2}
\mathfrak K_{\sigma,\eps}
=o_\eps(\mathfrak K_{\sigma,\eps}).
\]
\end{lemma}

\begin{proof}
On a tangential face, at least one transverse coordinate is at its box
threshold.  Its stable potential or kinetic quadratic contribution is
$2\eps R_\eps^2$, whereas the unstable potential contribution is bounded
below by
$-\eps(R_\eps+\vartheta)^2/2$.  Lemma~\ref{lem:saddle-hamiltonian-expansion}
therefore gives
\[
H-U(\sigma)
\ge
-\frac\eps2(R_\eps+\vartheta)^2
+2\eps R_\eps^2
+o(\eps R_\eps^2)
\ge c_{\rm tan}\eps R_\eps^2,
\]
for fixed $\vartheta$ and all sufficiently small $\eps$.

Moreover,
$|\nabla\bar f_{\eps,\vartheta}^\pm|\le C_\vartheta\eps^{-1/2}$,
and the surface measure of the tangential faces is bounded by
$C_\vartheta\eps^{(3d-1)/2}R_\eps^{3d-2}$.  Multiplication by the prefactor
$\eps$, followed by comparison with \eqref{eq:Ksigma-def}, proves the stated
surface-flux bound after increasing the harmless power $N$.
\end{proof}

The complementary case is controlled by the Gaussian tails of
the one-dimensional saddle profile.

\begin{lemma}[Profile-tail domination]
\label{lem:profile-tail-domination}
There exist $c,C>0$ such that, for every $s\in\R$,
\[
|\Phi_\sigma'(s)|\le Ce^{-cs^2}.
\]
Moreover, for $s\le-1$,
\[
1-\Phi_\sigma(s)\le Ce^{-cs^2},
\]
and, for $s\ge1$,
\[
\Phi_\sigma(s)\le Ce^{-cs^2}.
\]
\end{lemma}

\begin{proof}
Set $a=\mu_\sigma/(2c_\sigma)>0$.  The explicit formula gives
\[
|\Phi_\sigma'(s)|
=\sqrt{\frac{\mu_\sigma}{2\pi c_\sigma}}e^{-as^2}.
\]
Moreover, for $x>0$,
\[
\int_x^\infty e^{-ay^2}\,dy
\le\frac{1}{2ax}e^{-ax^2}.
\]
Applying this bound with $x=|s|\ge1$ to the defining Gaussian integrals for
$\Phi_\sigma(s)$ when $s>0$ and for $1-\Phi_\sigma(s)$ when $s<0$ proves the
two tail estimates, after changing $c,C>0$.
\end{proof}

\subsection{Sublevel Topology, Side Estimates, and Boundary-Layer Reduction}

We now turn from local analytic estimates to the topology of low
Hamiltonian sublevels.  The first lemma rules out low-energy paths that avoid
a fixed saddle neighborhood.

\begin{lemma}[Fixed saddle-box communication gap]
\label{lem:fixed-box-communication-gap}
For every sufficiently small fixed $b>0$, let $\mathcal K_b$ be the
box obtained from $\mathcal K_\eps$ by replacing $\delta_\eps$ by $b$.
There is $c_b>0$ such that every continuous path
$\omega:[0,1]\to\R^{3d}$ from $x_m$ to $x_s$ which avoids
$\mathcal K_b$ satisfies
{
\[
\max_{t\in[0,1]}H(\omega(t))\ge U(\sigma)+c_b .
\]}
\end{lemma}

\begin{proof}
Suppose there are paths $\omega_n$ avoiding $\mathcal K_b$ with
$\max_{t\in[0,1]}H(\omega_n(t))\le U(\sigma)+1/n$.
Let $C_n$ be the connected component
containing $x_m,x_s$ of the compact set
\[
\{H\le U(\sigma)+1/n\}\setminus\operatorname{int}\mathcal K_{b/2}.
\]
It exists because the image of $\omega_n$ is a connected subset of this set.
The components may be chosen nested, $C_{n+1}\subset C_n$.  Hence
$C_\infty=\cap_nC_n$ is compact and connected, contains $x_m,x_s$, lies in
$\{H\le U(\sigma)\}$, and avoids a neighborhood of $z_\sigma$.

The configuration projection of $C_\infty$ is connected and joins $m$ to
$s$ inside $\{U\le U(\sigma)\}$.  Two distinct components of
$\{U<U(\sigma)\}$ can meet through the level $U=U(\sigma)$ only at a critical
point: near a regular point that sublevel is a half-ball and is locally
connected.  The communication assumption therefore forces the projection to
contain $\sigma$.  At any point of $C_\infty$ above $\sigma$,
$H\le U(\sigma)$ forces $p=r=0$, so $C_\infty$ contains $z_\sigma$, the desired
contradiction.
\end{proof}

The fixed communication gap must next be transferred to the
shrinking saddle box used in the asymptotic construction.

\begin{lemma}[Shrinking saddle-box path exclusion]
\label{lem:path-exclusion-saddle-box}
There is $c_{\rm box}>0$ such that, after decreasing the fixed constant
$c_K$ so that $0<c_K<c_{\rm box}$, every continuous path
from $x_m$ to $x_s$ in $\{H<U(\sigma)+c_K\eps R_\eps^2\}$ meets $\mathcal K_\eps$
for any sufficiently small $\eps$; equivalently, any $x_m\to x_s$ path avoiding
$\mathcal K_\eps$ has
\[
\sup_{t\in[0,1]}H(\omega(t))\ge
U(\sigma)+c_{\rm box}\eps R_\eps^2.
\]
\end{lemma}

\begin{proof}
We divide the proof into three steps.

\emph{Step 1: fixed-face connections.}
Fix $b>0$ sufficiently small that the Taylor expansion of $H$ at
$z_\sigma$ is valid on a neighborhood of $\mathcal K_b$.  Choose
\[
0<c_0<\min\{c_b,c_{b/2}\},
\]
where $c_b,c_{b/2}$ are the communication gaps from
Lemma~\ref{lem:fixed-box-communication-gap}.  After decreasing $c_0$ and
$b$ if necessary,
\begin{equation}
\label{eq:fixed-box-low-boundary-faces}
\{H<U(\sigma)+c_0\}\cap\partial\mathcal K_b
\subset
\partial\mathcal K_b^-\cup\partial\mathcal K_b^+.
\end{equation}
On either physical face, $H$ is strictly convex in the transverse variables.
Let
\[
z_b^\pm
=\left(\pm b/\sqrt{\lambda_\sigma},0,0,0\right).
\]
For every
$z\in\partial\mathcal K_b^\pm\cap\{H<U(\sigma)+c_0\}$, the line segment
in the face from $z$ to $z_b^\pm$ is parameterized by
\[
\ell_{z,b}^\pm(u)=(1-u)z+uz_b^\pm,
\qquad u\in[0,1],
\]
and satisfies
\begin{equation}
\label{eq:fixed-face-segment-energy}
\max_{u\in[0,1]}H\left(\ell_{z,b}^\pm(u)\right)
\le\max\left\{H(z),H\left(z_b^\pm\right)\right\},
\qquad
H\left(z_b^\pm\right)
=U(\sigma)-\frac{b^2}{2}+\mathcal O(b^3)<U(\sigma).
\end{equation}
Continuing from $z_b^-$, respectively $z_b^+$, along the corresponding
unstable lobe and then inside $\Omega_m$, respectively $\Omega_s$, gives paths
$\kappa_z^-:[0,1]\to\R^{3d}$ from $z$ to $x_m$ and
$\kappa_z^+:[0,1]\to\R^{3d}$ from $z$ to $x_s$ such that
\begin{equation}
\label{eq:fixed-face-well-connection}
\kappa_z^\pm\cap\mathcal K_{b/2}=\varnothing,
\qquad
\max_{t\in[0,1]} H\left(\kappa_z^\pm(t)\right)
\le\max\{H(z),U(\sigma)\}.
\end{equation}

\emph{Step 2: a low-energy path contains a face-to-face excursion.}
Let $\omega:[0,1]\to\R^{3d}$ be an $x_m$--$x_s$ path with
\begin{equation}
\label{eq:fixed-box-low-path}
\max_{t\in[0,1]}H(\omega(t))<U(\sigma)+c_0.
\end{equation}
Lemma~\ref{lem:fixed-box-communication-gap} and the choice of $c_0$ imply
$\omega\cap\mathcal K_b\ne\varnothing$.  Decompose the portion of
$\omega$ in $\mathcal K_b$ into successive excursions.  By
\eqref{eq:fixed-box-low-boundary-faces}, every entrance or exit occurs through
a physical face.  The first face is $\partial\mathcal K_b^-$ and the last
is $\partial\mathcal K_b^+$: otherwise the corresponding initial or final
piece of $\omega$, concatenated with a path from
\eqref{eq:fixed-face-well-connection}, would give an $x_m$--$x_s$ path avoiding
$\mathcal K_{b/2}$ and satisfying \eqref{eq:fixed-box-low-path}, contrary
to the gap $c_{b/2}$.

Suppose that no excursion joins the negative face to the positive face.  The
face labels must then change between two consecutive excursions.  The
intervening piece outside $\mathcal K_b$, together with the two paths in
\eqref{eq:fixed-face-well-connection}, again yields an $x_m$--$x_s$ path
avoiding $\mathcal K_{b/2}$ below $U(\sigma)+c_0$, a contradiction.
Therefore some excursion joins the two physical faces.  By continuity, it
contains a point $z_0$ with $q_1(z_0)=0$.

\emph{Step 3: energy at the central crossing.}
On $\mathcal K_b\cap\{q_1=0\}$, the transverse quadratic form is positive
definite.  After decreasing $b$ to absorb the cubic remainder, there is
$c_\perp>0$ such that
\begin{equation}
\label{eq:central-section-transverse-coercivity}
H-U(\sigma)
\ge c_\perp\left(|q'|^2+|v|^2+|w|^2\right).
\end{equation}
If $\omega$ avoids $\mathcal K_\eps$, then at $z_0$ at least one transverse
coordinate exceeds its $\mathcal K_\eps$ threshold.  Hence, for a constant
$c_{\rm box}>0$ independent of $\eps$,
\[
H(z_0)-U(\sigma)
\ge c_{\rm box}\delta_\eps^2
=c_{\rm box}\eps R_\eps^2.
\]
Decrease $c_K$ so that $0<c_K<c_{\rm box}$.  Since
$\delta_\eps^2=\eps R_\eps^2=o(c_0)$, the preceding argument applies for all
small $\eps$ and proves both assertions.
\end{proof}

This path-exclusion estimate implies that the low sublevel
outside the saddle box splits into exactly the two well components.

\begin{lemma}[Hamiltonian box partition]
\label{lem:box-partition-low-sublevel}
For $K$ large and $\eps$ small,
$J_\eps=J_\eps^m\cup(J_\eps\cap\mathcal K_\eps)\cup J_\eps^s$, and any point of
$J_\eps\setminus\mathcal K_\eps$ lies in exactly one of $J_\eps^m,J_\eps^s$.
\end{lemma}

\begin{proof}
We divide the proof into three steps.

\emph{Step 1: disjointness of the well components.}
If $J_\eps^m\cap J_\eps^s\ne\varnothing$, path connectedness of the two open
components gives an $x_m$--$x_s$ path in
$J_\eps\setminus\mathcal K_\eps$.  Its energy is strictly below
\[
U(\sigma)+a_\eps
=U(\sigma)+c_K\eps R_\eps^2,
\]
contradicting Lemma~\ref{lem:path-exclusion-saddle-box}.  Thus
\begin{equation}
\label{eq:well-components-disjoint}
J_\eps^m\cap J_\eps^s=\varnothing.
\end{equation}

\emph{Step 2: components meeting the saddle box.}
The Taylor expansion of $H$ at $z_\sigma$ in the Hessian eigenbasis and the
choice of transverse side lengths give
\begin{equation}
\label{eq:low-sublevel-box-boundary}
J_\eps\cap\partial\mathcal K_\eps
\subset
\partial\mathcal K_\eps^-\cup\partial\mathcal K_\eps^+.
\end{equation}
On either physical face, scale the transverse variables by $\delta_\eps$.
The quadratic part is positive definite, while the $C^2$ norm of the rescaled
Taylor remainder is $o(1)$.  Hence each of the sets
\[
F_\eps^\pm
:=J_\eps\cap\partial\mathcal K_\eps^\pm
\]
is strictly convex and therefore connected for small $\eps$.  The orientation
of $e_\sigma$ and the same face-to-well construction as in
\eqref{eq:fixed-face-well-connection}, now applied at scale $\delta_\eps$,
imply
\begin{equation}
\label{eq:physical-faces-component-assignment}
F_\eps^-\subset J_\eps^m,
\qquad
F_\eps^+\subset J_\eps^s.
\end{equation}
It follows from \eqref{eq:low-sublevel-box-boundary}--
\eqref{eq:physical-faces-component-assignment} that every component of
$J_\eps\setminus\mathcal K_\eps$ whose closure meets the box is either
$J_\eps^m$ or $J_\eps^s$.

\emph{Step 3: exclusion of detached components.}
Suppose that $C$ is another component of
$J_\eps\setminus\mathcal K_\eps$.  By Step~2,
\[
\overline C\cap\partial\mathcal K_\eps=\varnothing.
\]
Since $\overline J_\eps$ is compact, $H$ attains its minimum on
$\overline C$.  The boundary of $C$ outside the box lies on
$\{H=U(\sigma)+a_\eps\}$, whereas $H<U(\sigma)+a_\eps$ in $C$; hence the
minimum is attained at an interior point $z_C$ and $z_C$ is a local minimum
of $H$. Therefore,
\[
0=\nabla H(z_C)=(\nabla U(\theta_C),Lp_C,Lr_C),
\]
so that $z_C=(i,0,0)$ for a local minimum $i$ of $U$.  By Assumption~
\ref{ass:double-well}, the only local minima of $U$ are $m$ and $s$.
Consequently $i\in\{m,s\}$, so $C$ must coincide with $J_\eps^m$ or
$J_\eps^s$, which leads to a contradiction.
Together with
\eqref{eq:well-components-disjoint}, this proves coverage and uniqueness.
\end{proof}

The physical side estimates use a simpler deterministic dichotomy: each face
point either lies in a Gaussian profile tail or pays a positive Hamiltonian
cost on the logarithmic scale.

\begin{lemma}[Side-face alternative]
\label{lem:side-face-alternative}
There is $a_{\rm sf}>0$ such that every point of
$\partial\mathcal K_\eps^-\cap J_\eps$ satisfies
\[
\frac{\Lambda_\sigma}{\sqrt\eps}\le-a_{\rm sf}R_\eps\qquad\text{or}\qquad
H\ge U(\sigma)+a_{\rm sf}\eps R_\eps^2,
\]
and every point of $\partial\mathcal K_\eps^+\cap J_\eps$ satisfies
\[
\frac{\Lambda_\sigma}{\sqrt\eps}\ge a_{\rm sf}R_\eps
\qquad\text{or}\qquad
H\ge U(\sigma)+a_{\rm sf}\eps R_\eps^2.
\]
\end{lemma}

\begin{proof}
On $q_1=-\delta_\eps/\sqrt{\lambda_\sigma}$,
$\alpha_\sigma q_1=-(\alpha_\sigma/\sqrt{\lambda_\sigma})\sqrt\eps R_\eps$.  Pick
$a_{\rm sf}$ small with
$\frac{L}{2(1+\zeta_\sigma^2)}(\alpha_\sigma/\sqrt{\lambda_\sigma}-a_{\rm sf})^2>\frac12$
(possible by Lemma~\ref{lem:side-face-kinetic-domination}).  If
$\Lambda_\sigma/\sqrt\eps>-a_{\rm sf}R_\eps$ then
\[
\frac{v_1+\zeta_\sigma w_1}{\sqrt\eps}\ge
\left(\frac{\alpha_\sigma}{\sqrt{\lambda_\sigma}}-a_{\rm sf}\right)R_\eps,
\]
so that by Cauchy--Schwarz inequality
\[
v_1^2+w_1^2\ge\left(1+\zeta_\sigma^2\right)^{-1}
\left(\frac{\alpha_\sigma}{\sqrt{\lambda_\sigma}}-a_{\rm sf}\right)^2\eps R_\eps^2.
\]
Consequently, uniformly on the negative face,
\[
\begin{aligned}
H-U(\sigma)
&\ge -\frac12\eps R_\eps^2
+\frac L2(v_1^2+w_1^2)+o(\eps R_\eps^2)\\
&\ge
\left[
\frac{L}{2(1+\zeta_\sigma^2)}
\left(\frac{\alpha_\sigma}{\sqrt{\lambda_\sigma}}-a_{\rm sf}\right)^2
-\frac12+o(1)
\right]\eps R_\eps^2.
\end{aligned}
\]
The limiting coefficient in brackets is strictly positive by the choice of
$a_{\rm sf}$.  Decreasing $a_{\rm sf}$ once more makes it at least
$a_{\rm sf}$ for any sufficiently small $\eps$, which proves the negative-face
alternative.  Replacing $q_1$ by $-q_1$ gives the positive-face statement.
\end{proof}

The side-face alternative and committor localization together
imply that the mismatch between the saddle profile and the well component
value is negligible.

\begin{lemma}[Side-face committor mismatch]
\label{lem:side-face-mismatch}
Let $j_\eps=\Phi_\sigma(\Lambda_\sigma/\sqrt\eps)$.  For every
$N\ge0$ and $k\in\{0,1,2\}$, there is $K_0=K_0(N,k)$, independent of
$\vartheta$, such that for $K\ge K_0$ and every fixed $\vartheta>0$,
\begin{align*}
&\lim_{\eps\downarrow0}
\frac{\eps}{Z_\eps\mathfrak K_{\sigma,\eps}}
\int_{\partial\mathcal K_\eps^-\cap J_\eps}
\eps^{-N}(\vartheta\sqrt\eps)^{-k}
(1-h_\eps^*)(1-j_\eps)e^{-H/\eps}\,dS=0,
\\
&\lim_{\eps\downarrow0}
\frac{\eps}{Z_\eps\mathfrak K_{\sigma,\eps}}
\int_{\partial\mathcal K_\eps^+\cap J_\eps}
\eps^{-N}(\vartheta\sqrt\eps)^{-k}
h_\eps^*j_\eps e^{-H/\eps}\,dS=0.
\end{align*}
\end{lemma}

\begin{proof}
Set
\[
F_\eps^-:=\partial\mathcal K_\eps^-\cap J_\eps,
\qquad
F_{\eps,0}^-:=F_\eps^-\cap\{H<U(\sigma)\},
\qquad
F_{\eps,1}^-:=F_\eps^-\cap\{H\ge U(\sigma)\}.
\]
By the orientation of the negative face and
Lemma~\ref{lem:box-partition-low-sublevel}, $F_{\eps,0}^-\subset W_m$.
Moreover, the high-energy alternative in
Lemma~\ref{lem:side-face-alternative} is impossible on $F_{\eps,0}^-$, and
hence
\[
\frac{\Lambda_\sigma}{\sqrt\eps}\le-a_{\rm sf}R_\eps
\quad\hbox{on }F_{\eps,0}^-.
\]
Lemmas~\ref{lem:profile-tail-domination} and
Proposition~\ref{prop:third-order-lrs-committor-localization} therefore imply
\begin{equation}
\label{eq:m-side-low-mismatch-bound}
\begin{aligned}
(1-h_\eps^*)(1-j_\eps)e^{-H/\eps}
&\le
C\eps^{-M_{\rm com}}e^{(H-U(\sigma))/\eps}
e^{-cR_\eps^2}e^{-H/\eps} \\
&=C\eps^{-M_{\rm com}}e^{-U(\sigma)/\eps}e^{-cR_\eps^2}
\qquad\hbox{on }F_{\eps,0}^-.
\end{aligned}
\end{equation}
On $F_{\eps,1}^-$, Lemma~\ref{lem:side-face-alternative} gives two cases.  If
$\Lambda_\sigma/\sqrt\eps\le-a_{\rm sf}R_\eps$, then
$1-j_\eps\le Ce^{-cR_\eps^2}$ and $e^{-H/\eps}\le e^{-U(\sigma)/\eps}$.
Otherwise,
$H\ge U(\sigma)+a_{\rm sf}\eps R_\eps^2$, and hence
$e^{-H/\eps}\le e^{-U(\sigma)/\eps}e^{-a_{\rm sf}R_\eps^2}$.  Thus, after
decreasing $c>0$ if necessary,
\begin{equation}
\label{eq:m-side-high-mismatch-bound}
(1-h_\eps^*)(1-j_\eps)e^{-H/\eps}
\le Ce^{-U(\sigma)/\eps}e^{-cR_\eps^2}
\qquad\hbox{on }F_{\eps,1}^-.
\end{equation}

The face $F_\eps^-$ has surface measure at most
$C\eps^{(3d-1)/2}R_\eps^{3d-1}$.  Therefore, for the prescribed $N$ and $k$
in the statement,
\begin{align*}
\frac{\eps}{Z_\eps}
\int_{F_\eps^-}\eps^{-N}(\vartheta\sqrt\eps)^{-k}
(1-h_\eps^*)(1-j_\eps)e^{-H/\eps}\,dS
\le
\frac{C\vartheta^{-k}}{Z_\eps}e^{-U(\sigma)/\eps}
\eps^{(3d+1-k)/2-N-M_{\rm com}}R_\eps^{3d-1}e^{-cR_\eps^2}.
\end{align*}
Since $R_\eps=K\sqrt{\log(1/\eps)}$, the normalized contribution satisfies
\[
\frac{1}{\mathfrak K_{\sigma,\eps}}
\frac{\eps}{Z_\eps}
\int_{F_\eps^-}\eps^{-N}(\vartheta\sqrt\eps)^{-k}
(1-h_\eps^*)(1-j_\eps)e^{-H/\eps}\,dS
\le
C\vartheta^{-k}\eps^{(1-k)/2-N-M_{\rm com}+cK^2}
\left(\log(1/\eps)\right)^{(3d-1)/2}.
\]
The right-hand side tends to zero once the fixed $K$ is chosen sufficiently
large.

For the positive face, define $F_{\eps,0}^+$ and $F_{\eps,1}^+$ analogously.
The same argument uses $F_{\eps,0}^+\subset W_s$, the bound for $h_\eps^*$
from Proposition~\ref{prop:third-order-lrs-committor-localization}, and
$j_\eps\le Ce^{-cR_\eps^2}$ on the profile-tail part.  It yields the second
estimate with the same polynomial losses.
\end{proof}

We extend the preceding face estimate uniformly across the thin
interpolation layers adjacent to the two physical faces.

\begin{lemma}[Side-layer committor mismatch]
\label{lem:side-layer-mismatch}
For $N\ge0$ and $k\in\{0,1,2\}$, define
\[
\begin{aligned}
M_{\eps,\vartheta}^{-,N,k}
&:=\frac{\eps^{-N}(\vartheta\sqrt\eps)^{-k}}{Z_\eps}
\int_{\mathcal A_{\eps,\vartheta}^-}
(1-h_\eps^*)(1-j_\eps)e^{-H/\eps}\,dz,\\
M_{\eps,\vartheta}^{+,N,k}
&:=\frac{\eps^{-N}(\vartheta\sqrt\eps)^{-k}}{Z_\eps}
\int_{\mathcal A_{\eps,\vartheta}^+}
h_\eps^*j_\eps e^{-H/\eps}\,dz,\\
D_{\eps,\vartheta}^{-,N,k}
&:=\frac{\eps^{-N}(\vartheta\sqrt\eps)^{-k}}{Z_\eps}
\int_{\mathcal A_{\eps,\vartheta}^-}
(1-h_\eps^*)
\left|\Phi_\sigma'(\Lambda_\sigma/\sqrt\eps)\right|
e^{-H/\eps}\,dz,\\
D_{\eps,\vartheta}^{+,N,k}
&:=\frac{\eps^{-N}(\vartheta\sqrt\eps)^{-k}}{Z_\eps}
\int_{\mathcal A_{\eps,\vartheta}^+}
h_\eps^*
\left|\Phi_\sigma'(\Lambda_\sigma/\sqrt\eps)\right|
e^{-H/\eps}\,dz.
\end{aligned}
\]
There exists some $K_0=K_0(N,k)$, independent of $\vartheta$, such that for
$K\ge K_0$ and every fixed $\vartheta>0$,
\[
\lim_{\eps\downarrow0}
\frac{M_{\eps,\vartheta}^{-,N,k}+M_{\eps,\vartheta}^{+,N,k}
+D_{\eps,\vartheta}^{-,N,k}+D_{\eps,\vartheta}^{+,N,k}}
{\mathfrak K_{\sigma,\eps}}=0.
\]
In particular,
\[
\lim_{\vartheta\downarrow0}\limsup_{\eps\downarrow0}
\frac{M_{\eps,\vartheta}^{-,N,k}+M_{\eps,\vartheta}^{+,N,k}
+D_{\eps,\vartheta}^{-,N,k}+D_{\eps,\vartheta}^{+,N,k}}
{\mathfrak K_{\sigma,\eps}}=0.
\]
\end{lemma}

\begin{proof}
Parametrize the negative layer $\mathcal A_{\eps,\vartheta}^-$ by
\[
q_1=q_1^-+\frac{\sqrt\eps}{\sqrt{\lambda_\sigma}}u,
\qquad
u\in[-\vartheta,0],
\qquad
q_1^-=-\frac{R_\eps\sqrt\eps}{\sqrt{\lambda_\sigma}},
\]
and let $\bar z$ be the projection of $z$ onto the inner face $u=0$.
Uniformly for $z\in\mathcal A_{\eps,\vartheta}^-$,
\[
\frac{|\Lambda_\sigma(z)-\Lambda_\sigma(\bar z)|}{\sqrt\eps}
\le C\vartheta,
\qquad
\frac{|H(z)-H(\bar z)|}{\eps}
\le C_\vartheta(R_\eps+1).
\]
For fixed $\vartheta$, both layers lie in the saddle neighborhood
$\mathcal N_\sigma$ of Lemma~\ref{lem:saddle-lobe-force-sign} when $\eps$ is
small.  Every component of $\{H<U(\sigma)\}$ contains a local minimum of $H$,
and the only such minima are $x_m,x_s$.  Hence
\begin{equation}
\label{eq:side-layer-well-membership}
\mathcal A_{\eps,\vartheta}^-\cap\{H<U(\sigma)\}\subset W_m,
\qquad
\mathcal A_{\eps,\vartheta}^+\cap\{H<U(\sigma)\}\subset W_s,
\end{equation}
where the signs follow from Lemma~\ref{lem:saddle-lobe-force-sign}.  Thus the
corresponding committor-localization estimate applies throughout the
low-energy part of each layer.

We also transfer the deterministic face alternative to the full layers.  If
$\bar z\notin J_\eps$, then
$H(\bar z)\ge U(\sigma)+c_K\eps R_\eps^2$.  If $\bar z\in J_\eps$, then
Lemma~\ref{lem:side-face-alternative} gives either
$\Lambda_\sigma(\bar z)/\sqrt\eps\le-a_{\rm sf}R_\eps$ or the same lower
bound for $H(\bar z)$ with $a_{\rm sf}$ in place of $c_K$.  The two
perturbation estimates above therefore imply, after decreasing a fixed
$a_{\rm lay}>0$, that every $z\in\mathcal A_{\eps,\vartheta}^-$ satisfies
\begin{equation}
\label{eq:negative-side-layer-alternative}
\frac{\Lambda_\sigma(z)}{\sqrt\eps}\le-a_{\rm lay}R_\eps
\qquad\hbox{or}\qquad
H(z)\ge U(\sigma)+a_{\rm lay}\eps R_\eps^2.
\end{equation}
The positive layer satisfies the sign-reversed alternative.  On the
low-energy part, \eqref{eq:side-layer-well-membership} and
Proposition~\ref{prop:third-order-lrs-committor-localization} control the
committor factor, while \eqref{eq:negative-side-layer-alternative} and
Lemma~\ref{lem:profile-tail-domination} control the profile factor.  On the
remaining part, either the same profile tail occurs or the Gibbs weight pays
$e^{-a_{\rm lay}R_\eps^2}$.

For fixed $\vartheta$, the perturbation cost is negligible on the logarithmic
scale:
\[
\frac{C_\vartheta(R_\eps+1)}{R_\eps^2}\longrightarrow0
\qquad\text{as $\eps\downarrow0$}.
\]
Consequently, after decreasing $c>0$ if necessary, the preceding
committor/profile and high-energy bounds are all bounded by
$C_\vartheta e^{-cR_\eps^2}$ for sufficiently small $\eps$.  Derivatives of the
interpolation cutoffs satisfy
\[
|\partial_{q_1}^k\chi_{\eps,\vartheta}^{\pm}|
\le C_k(\vartheta\sqrt\eps)^{-k},
\qquad k=1,2.
\]
Thus, for the prescribed $N$ and $k$, fixed $\vartheta$, and some fixed $M$,
each mismatch contribution is
bounded by
\[
C_{\vartheta,N,k}\eps^{-N-k/2}R_\eps^M e^{-c'R_\eps^2}
\mathfrak K_{\sigma,\eps}=o(\mathfrak K_{\sigma,\eps}),
\]
once $K$ dominates the finitely many polynomial losses.  The positive layer
is treated identically, with
$\mathcal A_{\eps,\vartheta}^+$ in place of
$\mathcal A_{\eps,\vartheta}^-$.  The same dichotomy and
Lemma~\ref{lem:profile-tail-domination} give the identical bound for the two
$D$-terms, since the profile derivative has Gaussian decay on the tail part.
This proves both limits in the statement.
\end{proof}

The preceding localization estimates can now be summarized by
removing every contribution except the central saddle box and the two
physical side layers.

\begin{proposition}[Exterior and side-layer negligibility]
\label{prop:cutoff-side-exterior-negligibility}
Let $E_{\eps,\vartheta}^{\rm ext}$ be defined in
Lemma~\ref{lem:saddle-localization-admissibility}.  For every $N\ge0$ and
$k\in\{0,1,2\}$, choose $K\ge K_0(N,k)$ as in
Lemma~\ref{lem:side-layer-mismatch}.  Then
\[
\lim_{\vartheta\downarrow0}\limsup_{\eps\downarrow0}
\frac{
|E_{\eps,\vartheta}^{\rm ext}|
+M_{\eps,\vartheta}^{-,N,k}
+M_{\eps,\vartheta}^{+,N,k}}
{\mathfrak K_{\sigma,\eps}}=0.
\]
\end{proposition}

\begin{proof}
Set
\[
D_{\eps,\vartheta}^{\rm ext}
:=\left(\R^{3d}\setminus\mathcal K_\eps\right)
\setminus\mathcal A_{\eps,\vartheta}^{\rm side}.
\]
Decompose the exterior derivative support as
\[
D_{\eps,\vartheta}^{\rm ext}=E_G\cup E_\Xi\cup E_{\rm const},
\]
where $E_G$ contains derivatives of the energy cutoff, $E_\Xi$ contains
derivatives of the fixed compactifying cutoff, and $E_{\rm const}$ is the
remainder outside $\mathcal K_\eps$ in
\eqref{eq:logarithmic-saddle-box} and the two physical side layers
$\mathcal A_{\eps,\vartheta}^\pm$ in \eqref{eq:physical-side-layers}.
For $D\in\{E_G,E_\Xi,E_{\rm const}\}$, define
\[
\mathscr E_{\eps,\vartheta}^{D}
:=\int_Dh_\eps^*(-\calL_\eps f_{\eps,\vartheta})\,d\pi_\eps.
\]
The three sets can be chosen disjoint, and hence
\[
E_{\eps,\vartheta}^{\rm ext}
=\mathscr E_{\eps,\vartheta}^{E_G}
+\mathscr E_{\eps,\vartheta}^{E_\Xi}
+\mathscr E_{\eps,\vartheta}^{E_{\rm const}}.
\]
On $E_{\rm const}$ the function is identically equal to its component value,
either zero or one, and hence
\[
\calL_\eps f_{\eps,\vartheta}=0\quad\hbox{on }E_{\rm const},
\qquad
\mathscr E_{\eps,\vartheta}^{E_{\rm const}}=0.
\]
On $E_G$, \eqref{eq:global-saddle-derivative-bound} and
$H\ge U(\sigma)+a_\eps/2$ give, for fixed $\vartheta$ and some $N<\infty$,
\[
\left|\mathscr E_{\eps,\vartheta}^{E_G}\right|
\le C_\vartheta Z_\eps^{-1}\eps^{-N}
e^{-[U(\sigma)+a_\eps/2]/\eps}
\le C_\vartheta\eps^{-N-3d/2}
e^{-a_\eps/(2\eps)}\mathfrak K_{\sigma,\eps}
=o(\mathfrak K_{\sigma,\eps}),
\]
because $a_\eps/(2\eps)=c_KR_\eps^2/2$ and $K$ was chosen to dominate
the fixed polynomial loss.  Lemma~\ref{lem:compact-outer-cutoff} gives the
same estimate on $E_\Xi$.  Therefore,
\[
\lim_{\vartheta\downarrow0}\limsup_{\eps\downarrow0}
\frac{
|\mathscr E_{\eps,\vartheta}^{E_G}|
+|\mathscr E_{\eps,\vartheta}^{E_\Xi}|
+|\mathscr E_{\eps,\vartheta}^{E_{\rm const}}|}
{\mathfrak K_{\sigma,\eps}}=0.
\]
By Lemma~\ref{lem:side-layer-mismatch},
\[
\lim_{\vartheta\downarrow0}\limsup_{\eps\downarrow0}
\frac{M_{\eps,\vartheta}^{-,N,k}+M_{\eps,\vartheta}^{+,N,k}}
{\mathfrak K_{\sigma,\eps}}=0.
\]
The triangle inequality now yields
\[
\lim_{\vartheta\downarrow0}\limsup_{\eps\downarrow0}
\frac{
|E_{\eps,\vartheta}^{\rm ext}|
+M_{\eps,\vartheta}^{-,N,k}
+M_{\eps,\vartheta}^{+,N,k}}
{\mathfrak K_{\sigma,\eps}}=0,
\]
which proves the proposition.
\end{proof}

\emph{Boundary-layer reduction.}  We now construct the boundary-layer test
function $f_{\eps,\vartheta}$ and reduce the weak integral to the central Gaussian
current.  It is the function in
\eqref{eq:piecewise-energy-saddle-test}: it equals $G_\eps j_\eps$ inside
$\mathcal K_\eps$, $G_\eps$ on $J_\eps^m$, and zero on $J_\eps^s$, and is
patched only through thin layers on $\partial\mathcal K_\eps^\pm$, with
\[
|\partial^\alpha f_{\eps,\vartheta}|\le C_\alpha\eps^{-N_\alpha}\vartheta^{-|\alpha|},
\qquad
|\alpha|\le2.
\]
After applying the fixed outer cutoff, set
\[
\mathcal I_\eps(f_{\eps,\vartheta})=
\int_{\R^{3d}}h_\eps^*(-\calL_\eps f_{\eps,\vartheta})\,d\pi_\eps.
\]

After this component replacement, the $m$-side interpolation
layer can be collapsed by the divergence theorem to a physical face current.

\begin{lemma}[Collapse of the physical side layer]
\label{lem:collapse-physical-side-layer}
On the $m$-side layer $\mathcal A_{\eps,\vartheta}^-$ defined in
\eqref{eq:physical-side-layers},
construct
\[
f_{\eps,\vartheta}=G_\eps\left\{j_\eps+\chi_{\eps,\vartheta}^-(1-j_\eps)\right\},
\]
where $\chi_{\eps,\vartheta}^-$ is the cutoff from
Lemma~\ref{lem:global-saddle-test-function}.  For fixed $\eps,\vartheta$,
\[
\begin{aligned}
I_{\eps,\vartheta}^-
&:=\int_{\mathcal A_{\eps,\vartheta}^-}
h_\eps^*(-\calL_\eps f_{\eps,\vartheta})\,d\pi_\eps,
\qquad
I_{\eps,\vartheta}^+
:=\int_{\mathcal A_{\eps,\vartheta}^+}
h_\eps^*(-\calL_\eps f_{\eps,\vartheta})\,d\pi_\eps,\\
\mathfrak J_{\sigma,\eps}^-
&:=\frac1{Z_\eps}\int_{\partial\mathcal K_\eps^-}
-\eps e^{-H/\eps}\mathsf M\nabla j_\eps\cdot n^-\,dS.
\end{aligned}
\]
Then
\begin{equation}
\label{eq:side-collapse-ordered-remainder}
\lim_{\vartheta\downarrow0}\limsup_{\eps\downarrow0}
\frac{|R_{\eps,\vartheta}^-|+|R_{\eps,\vartheta}^+|}
{\mathfrak K_{\sigma,\eps}}=0,
\end{equation}
where
\[
R_{\eps,\vartheta}^-:=I_{\eps,\vartheta}^--\mathfrak J_{\sigma,\eps}^-,
\qquad
R_{\eps,\vartheta}^+:=I_{\eps,\vartheta}^+.
\]
\end{lemma}

\begin{proof}
Set $\mathcal A^-:=\mathcal A_{\eps,\vartheta}^-$.

\emph{Step 1: removal of the energy cutoff and committor replacement.}
Set
\[
\bar f_{\eps,\vartheta}
=j_\eps+\chi_{\eps,\vartheta}^-(1-j_\eps).
\]
On $\mathcal A^-$ one has $f_{\eps,\vartheta}=G_\eps\bar f_{\eps,\vartheta}$.
The product rule gives
\[
\calL_\eps(G_\eps\bar f_{\eps,\vartheta})
=G_\eps\calL_\eps\bar f_{\eps,\vartheta}
+\bar f_{\eps,\vartheta}\calL_\eps G_\eps
+\frac{2\gamma\eps}{L}
\nabla_rG_\eps\cdot\nabla_r\bar f_{\eps,\vartheta}.
\]
All terms in the difference between
$\calL_\eps(G_\eps\bar f_{\eps,\vartheta})$ and
$\calL_\eps\bar f_{\eps,\vartheta}$ contain either $1-G_\eps$ or a derivative
of $G_\eps$ and are therefore supported in
$\{H\ge U(\sigma)+a_\eps/2\}$.  For fixed $\vartheta$, the explicit derivative
bounds for $G_\eps$, $j_\eps$, and $\chi_{\eps,\vartheta}^-$ consequently give, for some
fixed $N<\infty$,
\[
\left|I^-_{\eps,\vartheta}
-\int_{\mathcal A^-}h_\eps^*
(-\calL_\eps\bar f_{\eps,\vartheta})\,d\pi_\eps\right|
\le
C_\vartheta Z_\eps^{-1}\eps^{-N}
e^{-[U(\sigma)+a_\eps/2]/\eps}
\le C_\vartheta\eps^{-N-3d/2}e^{-c_KR_\eps^2/2}
\mathfrak K_{\sigma,\eps}
=o(\mathfrak K_{\sigma,\eps}),
\]
after increasing the fixed $K$ if necessary.  Thus
\begin{equation}
\label{eq:side-collapse-energy-cutoff-error}
I^-_{\eps,\vartheta}
=\int_{\mathcal A^-}h_\eps^*(-\calL_\eps \bar f_{\eps,\vartheta})\,d\pi_\eps
+o(\mathfrak K_{\sigma,\eps}),
\end{equation}
Next, set
\[
R^{\rm com}_{\eps,\vartheta}
:=\int_{\mathcal A^-}(h_\eps^*-1)
(-\calL_\eps\bar f_{\eps,\vartheta})\,d\pi_\eps.
\]
After expanding $\calL_\eps\bar f_{\eps,\vartheta}$, every term in
$R^{\rm com}_{\eps,\vartheta}$ contains one of
\[
(1-h_\eps^*)(1-j_\eps),
\qquad
(1-h_\eps^*)\Phi_\sigma'(\Lambda_\sigma/\sqrt\eps),
\]
multiplied by at most a fixed polynomial in
$\eps^{-1}$, $R_\eps$, and $\vartheta^{-1}$.
Lemma~\ref{lem:side-layer-mismatch} therefore yields
\begin{equation}
\label{eq:side-collapse-committor-error}
\left|R^{\rm com}_{\eps,\vartheta}\right|
\le
C_\vartheta Z_\eps^{-1}e^{-U(\sigma)/\eps}
\eps^{3d/2-N}R_\eps^N e^{-cR_\eps^2}
=o_{\eps,\vartheta}(\mathfrak K_{\sigma,\eps}).
\end{equation}
Combining \eqref{eq:side-collapse-energy-cutoff-error} and
\eqref{eq:side-collapse-committor-error},
\begin{equation}
\label{eq:side-collapse-component-replacement}
I^-_{\eps,\vartheta}
=\int_{\mathcal A^-}(-\calL_\eps\bar f_{\eps,\vartheta})\,d\pi_\eps
+o_{\eps,\vartheta}(\mathfrak K_{\sigma,\eps}).
\end{equation}

\emph{Step 2: divergence theorem.}
The symmetric diffusion matrix $\mathsf D$ has only an $r$-block, whereas
$\chi_{\eps,\vartheta}^-$ depends only on $q_1$.  Consequently,
\[
\nabla\chi_{\eps,\vartheta}^-\cdot\mathsf D\nabla(1-j_\eps)=0,
\]
and the product rule implies that
\begin{equation}
\label{eq:side-collapse-generator-product}
\calL_\eps\bar f_{\eps,\vartheta}
=(1-\chi_{\eps,\vartheta}^-)\calL_\eps j_\eps
+(1-j_\eps)\calL_\eps\chi_{\eps,\vartheta}^-.
\end{equation}
Using
\[
\calL_\eps f
=\eps e^{H/\eps}
\nabla\cdot\left(e^{-H/\eps}\mathsf M\nabla f\right)
\]
in \eqref{eq:side-collapse-component-replacement} implies that
\begin{equation}
\label{eq:side-collapse-boundary-decomposition}
\int_{\mathcal A^-}(-\calL_\eps\bar f_{\eps,\vartheta})\,d\pi_\eps
=-\frac{\eps}{Z_\eps}
\int_{\partial\mathcal A^-}
e^{-H/\eps}\mathsf M\nabla\bar f_{\eps,\vartheta}
\cdot n_{\mathcal A^-}\,dS.
\end{equation}
On the outer physical face, $\bar f_{\eps,\vartheta}=1$ in a neighborhood and
the flux is zero.  On the inner face,
$\bar f_{\eps,\vartheta}=j_\eps$; with the orientation $n^-$ used in the
statement, its contribution is
\[
\frac1{Z_\eps}
\int_{\partial\mathcal K_\eps^-}
-\eps e^{-H/\eps}\mathsf M\nabla j_\eps\cdot n^-\,dS.
\]
The remaining tangential faces contribute
$o_{\eps,\vartheta}(\mathfrak K_{\sigma,\eps})$ by
Lemma~\ref{lem:stable-side-gaussian-domination}.  Substitution into
\eqref{eq:side-collapse-boundary-decomposition} and then
\eqref{eq:side-collapse-component-replacement} gives
\[
I_{\eps,\vartheta}^-
=\mathfrak J_{\sigma,\eps}^-+R_{\eps,\vartheta}^-.
\]

\emph{Step 3: the $s$-side layer.}
Set
\[
\bar f_{\eps,\vartheta}^+=(1-\chi_{\eps,\vartheta}^+)j_\eps.
\]
The direct high-energy estimate in Step~1, with the signs reversed, first
gives
\[
I_{\eps,\vartheta}^+
=\int_{\mathcal A_{\eps,\vartheta}^+}h_\eps^*
(-\calL_\eps\bar f_{\eps,\vartheta}^+)\,d\pi_\eps
+o_{\eps,\vartheta}(\mathfrak K_{\sigma,\eps}).
\]
On its low-energy part, replace $h_\eps^*$ by its component value $0$.
Lemmas~\ref{lem:side-face-alternative},
\ref{lem:profile-tail-domination}, and
\ref{lem:side-layer-mismatch} bound the replacement error and the
profile-tail/high-energy complements by
$o_{\eps,\vartheta}(\mathfrak K_{\sigma,\eps})$.  Hence
\[
I_{\eps,\vartheta}^+
=o_{\eps,\vartheta}(\mathfrak K_{\sigma,\eps}),
\]
and therefore the same estimate holds for
$R_{\eps,\vartheta}^+=I_{\eps,\vartheta}^+$.
More precisely, the estimates in Steps 1--3 yield
\[
\lim_{\vartheta\downarrow0}\limsup_{\eps\downarrow0}
\frac{|R_{\eps,\vartheta}^-|+|R_{\eps,\vartheta}^+|}
{\mathfrak K_{\sigma,\eps}}=0,
\]
which proves \eqref{eq:side-collapse-ordered-remainder}.
\end{proof}

Combining the interior normal-form estimate with the side-layer
collapse reduces the full weak-capacity integral to the $m$-side face
current.

\begin{proposition}[Lee--Ramil--Seo-style saddle boundary-layer reduction]
\label{prop:saddle-boundary-layer-reduction}
As $\eps\downarrow0$ with $\vartheta$ fixed, and then $\vartheta\downarrow0$,
\[
\mathcal I_\eps(f_{\eps,\vartheta})
=\frac1{Z_\eps}\int_{\partial\mathcal K_\eps^-}-\eps e^{-H/\eps}\mathsf M\nabla j_\eps\cdot n^-\,dS
+o_{\eps,\vartheta}(\mathfrak K_{\sigma,\eps}),
\qquad
\lim_{\vartheta\downarrow0}\limsup_{\eps\downarrow0}\frac{|o_{\eps,\vartheta}(\mathfrak K_{\sigma,\eps})|}{\mathfrak K_{\sigma,\eps}}=0 .
\]
\end{proposition}

\begin{proof}
Split the integral into the logarithmic box, the two physical side layers,
and the exterior:
\begin{equation}
\label{eq:saddle-boundary-layer-decomposition}
\mathcal I_\eps(f_{\eps,\vartheta})
=I_\eps^{\rm box}+I_{\eps,\vartheta}^-+I_{\eps,\vartheta}^+
+E_{\eps,\vartheta}^{\rm ext}.
\end{equation}
Proposition~\ref{prop:cutoff-side-exterior-negligibility} gives
\begin{equation}
\label{eq:saddle-boundary-layer-exterior}
E_{\eps,\vartheta}^{\rm ext}
=o_{\eps,\vartheta}(\mathfrak K_{\sigma,\eps}).
\end{equation}
Inside the box, $f_{\eps,\vartheta}=G_\eps j_\eps$ and
$\calL_{\sigma,\eps}j_\eps=0$.  The product rule therefore gives the exact
decomposition
\[
I_\eps^{\rm box}
=-\int_{\mathcal K_\eps}h_\eps^*
G_\eps(\calL_\eps-\calL_{\sigma,\eps})j_\eps\,d\pi_\eps
+E_\eps^G,
\]
where
\[
E_\eps^G
:=-\int_{\mathcal K_\eps}h_\eps^*
\left(j_\eps\calL_\eps G_\eps
+\frac{2\gamma\eps}{L}\nabla_rG_\eps\cdot\nabla_rj_\eps\right)
\,d\pi_\eps.
\]
The integrand defining $E_\eps^G$ is supported in
$\{H\ge U(\sigma)+a_\eps/2\}$ and is bounded by a fixed power of
$\eps^{-1}$.  The same direct high-energy estimate used in
\eqref{eq:side-collapse-energy-cutoff-error} gives
$E_\eps^G=o(\mathfrak K_{\sigma,\eps})$.

There is a fixed geometric constant $C_{\rm box}$ such that
$\mathcal K_\eps$ is contained in
$\{z_\sigma+\sqrt\eps Y:|Y|\le C_{\rm box}R_\eps\}$.  Apply the pointwise
normal-form bound and Gaussian majorant from the proof of
Lemma~\ref{lem:normal-form-stability-saddle-blowup} with the fixed logarithmic
constant $C_{\rm box}K$.  Since $0\le h_\eps^*G_\eps\le1$, together with
Lemma~\ref{lem:saddle-hamiltonian-expansion} this yields
\[
\left|\int_{\mathcal K_\eps}h_\eps^*
G_\eps(\calL_\eps-\calL_{\sigma,\eps})j_\eps\,d\pi_\eps\right|
\le C\sqrt\eps R_\eps^3\mathfrak K_{\sigma,\eps}
=o(\mathfrak K_{\sigma,\eps}).
\]
Consequently,
\begin{equation}
\label{eq:saddle-boundary-layer-box}
I_\eps^{\rm box}=o_\eps(\mathfrak K_{\sigma,\eps}).
\end{equation}
By Lemma~\ref{lem:collapse-physical-side-layer},
\begin{equation}
\label{eq:saddle-boundary-layer-sides}
I_{\eps,\vartheta}^-
=\frac1{Z_\eps}\int_{\partial\mathcal K_\eps^-}
-\eps e^{-H/\eps}\mathsf M\nabla j_\eps\cdot n^-\,dS
+o_{\eps,\vartheta}(\mathfrak K_{\sigma,\eps}),
\qquad
I_{\eps,\vartheta}^+=o_{\eps,\vartheta}(\mathfrak K_{\sigma,\eps}).
\end{equation}
Substituting \eqref{eq:saddle-boundary-layer-exterior},
\eqref{eq:saddle-boundary-layer-box}, and
\eqref{eq:saddle-boundary-layer-sides} into
\eqref{eq:saddle-boundary-layer-decomposition}, and using
\eqref{eq:side-collapse-ordered-remainder} to control
$R_{\eps,\vartheta}^-=I_{\eps,\vartheta}^--\mathfrak J_{\sigma,\eps}^-$ and
$R_{\eps,\vartheta}^+=I_{\eps,\vartheta}^+$ in the ordered limit, proves both the
asserted expansion and its ordered-limit error bound.
\end{proof}

It remains to transport this physical face current to the
central linear section, where the Gaussian calculation applies.

\begin{lemma}[Physical side-face current equals central Gaussian current]
\label{lem:artificial-surface-current}
As $\eps\downarrow 0$,
\[
\frac1{Z_\eps}\int_{\partial\mathcal K_\eps^-}
-\eps e^{-H/\eps}\mathsf M\nabla j_\eps\cdot n^-\,dS
=[1+o_\eps(1)]\mathfrak C_{\sigma,\eps}^{\rm quad}.
\]
\end{lemma}

\begin{proof}
The physical side-face current in the statement is
$\mathfrak J_{\sigma,\eps}^{-,\rm phys}$ from
Lemma~\ref{lem:uniform-saddle-current-stability}.  By
\eqref{eq:current-stability-hamiltonian}--
\eqref{eq:current-stability-tail} and the triangle inequality,
\[
\left|\mathfrak J_{\sigma,\eps}^{-,\rm phys}
-\mathfrak C_{\sigma,\eps}^{\rm quad}\right|
\le
\left|\mathfrak J_{\sigma,\eps}^{-,\rm phys}
-\mathfrak J_{\sigma,\eps}^{-,\rm quad}\right|
+
\left|\mathfrak J_{\sigma,\eps}^{-,\rm quad}
-\mathfrak J_{\sigma,\eps}^{0,\rm quad,log}\right|
+
\left|\mathfrak J_{\sigma,\eps}^{0,\rm quad,log}
-\mathfrak C_{\sigma,\eps}^{\rm quad}\right|
=o(\mathfrak K_{\sigma,\eps}).
\]
Since $\mathfrak C_{\sigma,\eps}^{\rm quad}\asymp
\mathfrak K_{\sigma,\eps}$, the additive error is a relative $o_\eps(1)$ error,
which proves the claim.
\end{proof}

We now combine the preceding asymptotic estimate with the exact
weak test-function identity.  The exact identity is what removes the auxiliary
layer parameter from the capacity itself.

\begin{proposition}[Weak test-function reduction to the quadratic current]
\label{prop:weak-test-function-capacity-reduction}
For every fixed $\vartheta>0$ and all sufficiently small $\eps>0$,
\[
\Cap_\eps(A_\eps,S_\eps)=\mathcal I_\eps(f_{\eps,\vartheta}).
\]
Moreover, as $\eps\downarrow 0$,
\[
\Cap_\eps(A_\eps,S_\eps)
=[1+o_\eps(1)]\mathfrak C_{\sigma,\eps}^{\rm quad}.
\]
\end{proposition}

\begin{proof}
For fixed $\vartheta>0$ and all sufficiently small $\eps>0$,
Proposition~\ref{prop:weak-test-function-capacity} identifies the weak
integral exactly with the capacity.  On the other hand,
Proposition~\ref{prop:saddle-boundary-layer-reduction} and
Lemma~\ref{lem:artificial-surface-current}, together with
Proposition~\ref{prop:quadratic-gaussian-current}, give
\[
r(\vartheta):=\limsup_{\eps\downarrow0}
\left|
\frac{\mathcal I_\eps(f_{\eps,\vartheta})}
{\mathfrak C_{\sigma,\eps}^{\rm quad}}-1
\right|,
\qquad
\lim_{\vartheta\downarrow0}r(\vartheta)=0.
\]
If
$\mathscr R_\eps:=\Cap_\eps(A_\eps,S_\eps)/\mathfrak C_{\sigma,\eps}^{\rm quad}$,
then the exact weak identity makes $\mathscr R_\eps$ independent of $\vartheta$, and
\[
\limsup_{\eps\downarrow0}|\mathscr R_\eps-1|\le r(\vartheta)
\qquad\text{for every fixed }\vartheta>0.
\]
Letting $\vartheta\downarrow0$ proves $\mathscr R_\eps\to1$.  Finally, as in the
convention accompanying \eqref{eq:Ksigma-def}, the order is
$\eps\downarrow0$ first and $\vartheta\downarrow0$ second.
\end{proof}

\subsection{Current Stability}
\label{app:gaussian}

It remains to compare the nonlinear current on the physical face with the
quadratic current on the central linear section.  The following lemma performs
the comparison uniformly on the logarithmic saddle scale.
\begin{lemma}[Uniform saddle current stability]
\label{lem:uniform-saddle-current-stability}
Let $j_\eps=\Phi_\sigma(\Lambda_\sigma/\sqrt\eps)$ in the logarithmic saddle box
and define the physical and quadratic currents on the $m$-side face by
\[
\begin{aligned}
\mathfrak J_{\sigma,\eps}^{-,\rm phys}
&:=\frac1{Z_\eps}\int_{\partial\mathcal K_\eps^-}
-\eps e^{-H/\eps}\mathsf M\nabla j_\eps\cdot n^-\,dS,\\
\mathfrak J_{\sigma,\eps}^{-,\rm quad}
&:=\frac1{Z_\eps}\int_{\partial\mathcal K_\eps^-}
-\eps e^{-H_\sigma^{\rm quad}/\eps}
\mathsf M\nabla j_\eps\cdot n^-\,dS.
\end{aligned}
\]
Use the quadratic current direction
$\mathfrak u_\sigma=\mathsf M\widehat\ell_\sigma/c_\sigma$ fixed above, and let
\begin{align*}
&\Sigma_{\sigma,\eps}^{\rm log}
:=\left\{z-\Lambda_\sigma(z)\mathfrak u_\sigma:
z\in\partial\mathcal K_\eps^-\right\}
\subset\{\Lambda_\sigma=0\},
\\
&\mathfrak J_{\sigma,\eps}^{0,\rm quad,log}
:=\frac1{Z_\eps}\int_{\Sigma_{\sigma,\eps}^{\rm log}}
-\eps e^{-H_\sigma^{\rm quad}/\eps}
\mathsf M\nabla j_\eps\cdot n_\sigma\,dS.
\end{align*}
Then
\begin{align}
\label{eq:current-stability-hamiltonian}
\left|\mathfrak J_{\sigma,\eps}^{-,\rm phys}
-\mathfrak J_{\sigma,\eps}^{-,\rm quad}\right|
&=o_\eps(\mathfrak K_{\sigma,\eps}),\\
\label{eq:current-stability-transport}
\left|\mathfrak J_{\sigma,\eps}^{-,\rm quad}
-\mathfrak J_{\sigma,\eps}^{0,\rm quad,log}\right|
&=o_\eps(\mathfrak K_{\sigma,\eps}),\\
\label{eq:current-stability-tail}
\left|\mathfrak J_{\sigma,\eps}^{0,\rm quad,log}
-\mathfrak C_{\sigma,\eps}^{\rm quad}\right|
&=o_\eps(\mathfrak K_{\sigma,\eps}).
\end{align}
\end{lemma}

\begin{proof}
We first transport the quadratic current.  Since
$\nabla j_\eps=\eps^{-1/2}\Phi_\sigma'(\Lambda_\sigma/\sqrt\eps)\widehat\ell_\sigma$,
the quadratic current field is
\[
\mathcal J_\eps
:=-\eps e^{-H_\sigma^{\rm quad}/\eps}\mathsf M\nabla j_\eps
=-\sqrt\eps\,c_\sigma
e^{-H_\sigma^{\rm quad}/\eps}
\Phi_\sigma'(\Lambda_\sigma/\sqrt\eps)\mathfrak u_\sigma.
\]
Thus $\mathcal J_\eps$ is everywhere parallel to $\mathfrak u_\sigma$, and
the profile equation gives $\nabla\cdot\mathcal J_\eps=0$.  Moreover, the
$q_1$-component of $\mathfrak u_\sigma$ is
$(Lc_\sigma)^{-1}>0$.  Hence projection along $\mathfrak u_\sigma$ maps the
bounded face $\partial\mathcal K_\eps^-$ bijectively onto
$\Sigma_{\sigma,\eps}^{\rm log}$, and the swept set
\[
\Omega_{\sigma,\eps}^{\rm tr}
:=\left\{z-t\Lambda_\sigma(z)\mathfrak u_\sigma:
z\in\partial\mathcal K_\eps^-,\ 0\le t\le1\right\},
\]
is a Lipschitz tube.  Its lateral boundary is swept in the
$\mathfrak u_\sigma$ direction.  Consequently,
$\mathcal J_\eps\cdot n=0$ there.  Applying the divergence theorem separately
to the portions of the face on which $\Lambda_\sigma$ is positive and
negative handles the possible intersection of the two sections; their common
zero set has surface measure zero.  We therefore obtain the exact identity
\[
\mathfrak J_{\sigma,\eps}^{-,\rm quad}
=\mathfrak J_{\sigma,\eps}^{0,\rm quad,log},
\]
which proves \eqref{eq:current-stability-transport}.

The eigenvector identities
$\alpha_\sigma-\gamma\zeta_\sigma=\mu_\sigma$ and
$\gamma(1-\zeta_\sigma)=\mu_\sigma\zeta_\sigma$ give
\[
\mathfrak u_\sigma
=\frac{1}{Lc_\sigma}(1,-\mu_\sigma,-\mu_\sigma\zeta_\sigma)
\]
on the unstable triple.
The line through the origin in direction $\mathfrak u_\sigma$ meets the
blown-up negative face at the point whose unstable coordinates are
\[
\left(-\frac{R_\eps}{\sqrt{\lambda_\sigma}},
\frac{\mu_\sigma R_\eps}{\sqrt{\lambda_\sigma}},
\frac{\mu_\sigma\zeta_\sigma R_\eps}{\sqrt{\lambda_\sigma}}\right).
\]
Lemma~\ref{lem:side-face-kinetic-domination} implies
$\mu_\sigma^2<\lambda_\sigma/L$, while $0<\zeta_\sigma<1$.
This intersection point is therefore separated by a fixed multiple of $R_\eps$
from every tangential boundary of the blown-up face.  Since the projection is
a fixed linear isomorphism, there is $c_0>0$ such that
$\eps^{-1/2}\Sigma_{\sigma,\eps}^{\rm log}$ contains the ball of radius
$c_0R_\eps$ in $\ker\Lambda_\sigma$.  The restriction of
$H_\sigma^{\rm quad}-U(\sigma)$ to this hyperplane is positive definite by
Lemma~\ref{lem:coercivity-ker-Lambda}.  Its Gaussian tail consequently gives
\[
\left|\mathfrak C_{\sigma,\eps}^{\rm quad}
-\mathfrak J_{\sigma,\eps}^{0,\rm quad,log}\right|
\le C Z_\eps^{-1}e^{-U(\sigma)/\eps}
\eps^{3d/2}R_\eps^N e^{-cR_\eps^2}
=o(\mathfrak K_{\sigma,\eps}).
\]
This proves \eqref{eq:current-stability-tail} and also shows that
$\mathfrak J_{\sigma,\eps}^{-,\rm quad}=\mathcal O(\mathfrak K_{\sigma,\eps})$.

It remains to replace the quadratic Hamiltonian by the nonlinear one.  In
blown-up variables $z=z_\sigma+\sqrt\eps Y$, $|Y|\le CR_\eps$,
Lemma~\ref{lem:saddle-hamiltonian-expansion} gives
\[
\exp\!\left(-\left(H-H_\sigma^{\rm quad}\right)/\eps\right)
=1+\mathcal O\left(\sqrt\eps R_\eps^3\right)=1+o(1),
\]
uniformly on the face.  Since
$-\mathsf M\nabla j_\eps\cdot n^->0$ there, the preceding quadratic-current
bound yields
\[
\left|\mathfrak J_{\sigma,\eps}^{-,\rm phys}
-\mathfrak J_{\sigma,\eps}^{-,\rm quad}\right|
\le C\sqrt\eps R_\eps^3\,
\mathfrak J_{\sigma,\eps}^{-,\rm quad}
=o(\mathfrak K_{\sigma,\eps}),
\]
which proves \eqref{eq:current-stability-hamiltonian}.
\end{proof}

\section{Well Localization and Committor Estimates}
\label{app:well-localization}

\subsection{Common Finite-Time Tracking and Non-Degeneracy}

Recall $x_m=(m,0,0)$, $A_\eps=B(x_m,\eps)$, $S_\eps=B(x_s,\eps)$,
$G_{\eps,\varrho}=B(x_m,\eps^\varrho)$.

The proof separates into two chains.  The chain culminating in
Lemma~\ref{lem:intrawell-flatness-reduction} proves intrawell flatness through
non-degeneracy and a uniform Harnack estimate.  The chain culminating in
Proposition~\ref{prop:third-order-lrs-committor-localization} proves committor
localization through a local Lyapunov estimate, density bounds, killed-process
exit tails, and stopped time reversal.  Lemma~\ref{lem:uniform-fw-tracking} is
used in both chains.

We begin with a uniform finite-horizon estimate that transfers
deterministic entrance properties to the small-noise diffusion.  The compact
starting set is important because the estimate will later be applied after
several stopping times.

\begin{lemma}[Uniform finite-time Freidlin--Wentzell tracking]
\label{lem:uniform-fw-tracking}
Let $K$ be compact, $T<\infty$, and $\phi_t(z)$ the zero-noise flow of the
original or adjoint dynamics.  For every $\rho>0$ there is $c_{\rho,T,K}>0$ with
\[
\sup_{z\in K}\Pbb_z\left(\sup_{0\le t\le T}\left|Z_t^\eps-\phi_t(z)\right|>\rho\right)\le e^{-c_{\rho,T,K}/\eps},
\]
for any sufficiently small $\eps$.
\end{lemma}

\begin{proof}
Define the deterministic flow tube by
\begin{equation}
\label{eq:deterministic-flow-tube}
\mathscr T_{K,T}
:=
\left\{\phi_t(z):z\in K,\ 0\le t\le T\right\}.
\end{equation}
The set $\mathscr T_{K,T}$ is compact; choose $R$ with
$\mathscr T_{K,T}\subset B_{R/2}$ and let $\bar Z^\eps$ be the diffusion with a
globally Lipschitz drift agreeing with the true drift on $B_R$.  The classical
finite-time Freidlin--Wentzell estimate
\cite[Chapter~4]{FreidlinWentzell2012} implies
\[
\sup_{z\in K}\Pbb_z\left(\sup_{t\le T}\left|\bar Z_t^\eps-\phi_t(z)\right|>\eta\right)\le e^{-c_1/\eps},
\]
for
\[
\eta
=
\rho\wedge\frac12
\dist\left(\mathscr T_{K,T},\partial B_R\right)>0,
\]
where $\mathscr T_{K,T}$ is defined in \eqref{eq:deterministic-flow-tube}.  Couple
$Z^\eps$ and $\bar Z^\eps$ with the same Brownian motion and define their
common exit time from $B_R$ by
\[
\tau_R^\eps:=\inf\left\{t\ge0:Z_t^\eps\notin B_R\right\}
=\inf\left\{t\ge0:\bar Z_t^\eps\notin B_R\right\}.
\]
The equality follows from pathwise uniqueness because the two drifts agree
on $B_R$.  Up to $\tau_R^\eps$ the two processes agree, and
$\{\tau_R^\eps\le T\}$
is contained in
the tracking-failure event; combining proves the bound for the original process.
\end{proof}

The flatness argument begins with a uniform positive lower bound for the mean
transition time.  The next elementary consequence of finite-time tracking is
the third-order counterpart of
\cite[Lemma~6.1]{LeeRamilSeo2026}.

\begin{lemma}[Non-degenerate mean transition time]
\label{lem:nondegenerate-mean-transition-time}
There are $c_*>0$ and $\eps_*>0$ such that
\[
\E_{x_m}[\tau_{S_\eps}]\ge c_*
\qquad\text{for all }0<\eps<\eps_*.
\]
In particular,
$\liminf_{\eps\downarrow0}\E_{x_m}[\tau_{S_\eps}]>0$.
\end{lemma}

\begin{proof}
Choose $\rho>0$ so small that
$B(x_m,2\rho)\cap S_\eps=\varnothing$ for all sufficiently small $\eps$.
The zero-noise trajectory starting from the equilibrium $x_m$ is constant.
Lemma~\ref{lem:uniform-fw-tracking}, applied with $K=\{x_m\}$ and $T=1$,
implies that there exists a constant $c>0$ such that
\[
\Pbb_{x_m}\!\left(\sup_{0\le t\le1}\left|Z_t^\eps-x_m\right|\ge\rho\right)
\le e^{-c/\eps}.
\]
Consequently,
\[
\E_{x_m}[\tau_{S_\eps}]
\ge \Pbb_{x_m}(\tau_{S_\eps}>1)
\ge1-e^{-c/\eps}.
\]
The assertion follows, for example, with $c_*=1/2$ after decreasing
$\eps_*$.
\end{proof}

\subsection{Uniform Harnack Estimate and Intrawell Flatness}
\label{app:lyapunov}

The parameter-uniform part of the Harnack argument is a stability statement
for finitely many local Green kernels.  We isolate that statement before
applying Bony's fixed-operator estimates.

\begin{lemma}[Stability of shifted hypoelliptic Green kernels]
\label{lem:shifted-green-kernel-stability}
Let $\mathscr U\subset\R^N$ be a fixed smooth bounded domain.  For some
$\nu_\star<\infty$, let
\[
\mathscr P_\nu=\sum_{j=1}^dX_{j,\nu}^2+Y_\nu,
\qquad \nu\in[\nu_\star,\infty),
\]
be diffusion generators whose coefficients converge in
$C^\infty(\overline{\mathscr U})$ to those of $\mathscr P_0$ as
$\nu\to\infty$.  Assume that a
fixed finite bracket family has least singular value bounded below on
$\overline{\mathscr U}$, uniformly for $\nu\ge\nu_\star$.  Assume also that,
for every $\nu\ge\nu_\star$,
$\mathscr U$ is regular for the Dirichlet problems of
$\mathscr P_\nu-\beta$ and its formal adjoint, with forward and adjoint
barrier inequalities having a common strictly positive margin, where
$\beta>0$ is fixed and larger than a uniform upper bound
$\kappa_\dagger$ for the zero-order coefficient of the formal adjoints.
Then, for every compact
$K\subset(\mathscr U\times\mathscr U)\setminus\{(x,x):x\in\mathscr U\}$,
the Dirichlet Green kernels $G_{\nu,\beta}^{\mathscr U}(x,y)$ of
$\beta-\mathscr P_\nu$ satisfy
\begin{equation}
\label{eq:shifted-green-convergence}
G_{\nu,\beta}^{\mathscr U}\longrightarrow
G_{0,\beta}^{\mathscr U}
\qquad\text{in }C^\infty(K)\qquad\text{as $\nu\to\infty$}.
\end{equation}
Consequently, if $G_{0,\beta}^{\mathscr U}\ge c_0>0$ on $K$, then
$G_{\nu,\beta}^{\mathscr U}\ge c_0/2$ there for all sufficiently large
$\nu$.
\end{lemma}

\begin{proof}
The regular Dirichlet problems and the Green-operator construction of
\cite[Section~6]{Bony1969} give the operators
$\mathcal G_{\nu,\beta}^{\mathscr U}$ and their non-negative kernels
$G_{\nu,\beta}^{\mathscr U}$.  For bounded $f\ge0$, the function
$u=\mathcal G_{\nu,\beta}^{\mathscr U}f$ solves
\[
(\beta-\mathscr P_\nu)u=f\quad\text{in }\mathscr U,
\qquad u|_{\partial\mathscr U}=0.
\]
The maximum principle and comparison with the constant
$\beta^{-1}\|f\|_\infty$ give, for $f\ge0$,
\[
0\le\mathcal G_{\nu,\beta}^{\mathscr U}f
\le\beta^{-1}\|f\|_\infty.
\]
Approximating $f\equiv1$ by smooth functions implies that
\begin{equation}
\label{eq:green-resolvent-mass}
G_{\nu,\beta}^{\mathscr U}\ge0,
\qquad
\sup_{x\in\mathscr U}
\int_{\mathscr U}G_{\nu,\beta}^{\mathscr U}(x,y)\,dy
\le\beta^{-1}.
\end{equation}
Applying the same comparison to the adjoint Dirichlet Green operator yields
the complementary mass bound
\begin{equation}
\label{eq:green-two-sided-mass}
\sup_{y\in\mathscr U}\int_{\mathscr U}
G_{\nu,\beta}^{\mathscr U}(x,y)\,dx
\le(\beta-\kappa_\dagger)^{-1}.
\end{equation}
If the closure of $K'$ is compact and $\overline{K'}\subset(\mathscr U\times\mathscr U)\setminus\operatorname{diag}$,
the kernels solve the homogeneous forward and adjoint equations on a
neighborhood of $K'$.  The uniform finite-bracket bound, coefficient
convergence, \eqref{eq:green-resolvent-mass},
\eqref{eq:green-two-sided-mass}, and the interior estimates of
\cite{RothschildStein1976,BramantiZhu2013} give, for every $k$,
\begin{equation}
\label{eq:green-uniform-off-diagonal-Ck}
\sup_{\nu\ge\nu_\star}
\left\|G_{\nu,\beta}^{\mathscr U}\right\|_{C^k(K')}<\infty.
\end{equation}
Consequently every sequence $\nu_j\to\infty$ has a subsequence along which
the Green kernels converge in $C^\infty(K')$.

The forward and adjoint barriers, with their common positive margin, also
give a modulus $\omega(\delta)\downarrow0$ such that, whenever the closure of $K_y$ is compact and
$\overline{K_y}\subset\mathscr U$,
\begin{equation}
\label{eq:green-uniform-boundary-modulus}
\sup_{\nu\ge\nu_\star}
\sup_{\substack{\operatorname{dist}(x,\partial\mathscr U)\le\delta\\
y\in K_y}}G_{\nu,\beta}^{\mathscr U}(x,y)
+\sup_{\nu\ge\nu_\star}
\sup_{\substack{x\in K_y\\
\operatorname{dist}(y,\partial\mathscr U)\le\delta}}
G_{\nu,\beta}^{\mathscr U}(x,y)
\le\omega(\delta).
\end{equation}
Thus every subsequential limit has zero Dirichlet data in both variables.

Fix $y\in\mathscr U$ and an arbitrary sequence $\nu_j\to\infty$.  By
\eqref{eq:green-two-sided-mass}, after passing to a subsequence (not relabeled),
$G_{\nu_j,\beta}^{\mathscr U}(\cdot,y)\,dx$ converges weakly to a finite
measure $\mu_y$.  For every $\varphi\in C_c^\infty(\mathscr U)$,
\[
\int_{\mathscr U}G_{\nu_j,\beta}^{\mathscr U}(x,y)
\left(\beta-\mathscr P_{\nu_j}^*\right)\varphi(x)\,dx=\varphi(y),
\]
and
\[
\left|\int_{\mathscr U}G_{\nu_j,\beta}^{\mathscr U}(x,y)
\left(\mathscr P_{\nu_j}^*-\mathscr P_0^*\right)\varphi(x)\,dx\right|
\le\left(\beta-\kappa_\dagger\right)^{-1}
\left\|\left(\mathscr P_{\nu_j}^*-\mathscr P_0^*\right)\varphi\right\|_\infty
\longrightarrow0
\qquad\text{as }j\to\infty.
\]
Passing to the limit therefore gives
\[
\int_{\mathscr U}\left(\beta-\mathscr P_0^*\right)\varphi\,d\mu_y
=\varphi(y).
\]
Equations \eqref{eq:green-uniform-off-diagonal-Ck} and
\eqref{eq:green-uniform-boundary-modulus} identify the off-diagonal density
of $\mu_y$ with the smooth subsequential limit and give zero boundary data.
Uniqueness for the Dirichlet problem of $\beta-\mathscr P_0$ now yields
$\mu_y(dx)=G_{0,\beta}^{\mathscr U}(x,y)\,dx$.  Every subsequence has the
same limit; hence \eqref{eq:shifted-green-convergence} holds.  The final
lower bound follows from uniform convergence on $K$.
\end{proof}

To compare a positive harmonic function at different points, we combine the
Green-kernel stability above with a compact-family version of Bony's Harnack
estimate whose constant is uniform under the rescaled small-noise
perturbation.

\begin{lemma}[Compact-family form of Bony's estimate]
\label{lem:compact-family-bony}
Let $\mathcal Q$ be a bounded cylinder and
\[
 \mathscr P_\nu=\sum_{j=1}^d X_{j,\nu}^2+Y_\nu
\]
be a family of smooth operators whose coefficients converge in
$C^\infty(\overline{\mathcal Q})$.  Let $K_0,K_1\subset\mathcal Q$ be compact,
with $K_0\subset\operatorname{int}K_1$ and $K_1\subset\mathcal Q$, and fix a
distinguished base point $z_\bullet\in\mathcal Q$.
Assume that a finite family of controlled curves from $z_\bullet$ reaches
neighborhoods covering $K_1$, stays a fixed distance from
$\partial\mathcal Q$, and persists for every large $\nu$.  Assume also that a
fixed finite list of commutators of $(X_{j,\nu},Y_\nu)$ spans the tangent
space on a neighborhood of the corresponding control tubes, with least
singular value bounded away from zero, and that
\[
\inf_{z\in\mathcal Q}\sum_{j=1}^d|X_{j,\nu}(z)|^2\ge c_X>0,
\]
uniformly for large $\nu$.  Then there is a constant $C$ such that, for every sufficiently large
$\nu$, every positive function
$u_\nu:\mathcal Q\to(0,\infty)$ satisfying
$\mathscr P_\nu u_\nu=0$ in $\mathcal Q$ obeys
\[
 \sup_{z\in K_1}u_\nu(z)+\|\nabla u_\nu\|_{L^\infty(K_0)}
 \le C u_\nu(z_\bullet).
\]
The same conclusion holds for non-negative solutions by approximation.
\end{lemma}

\begin{proof}
We divide the proof into four steps.

\emph{Step 1: a finite system of regular boxes.}
For the limiting operator $\mathscr P_0$, choose finitely many oriented
control curves starting at $z_\bullet$ whose terminal neighborhoods cover
$K_1$.  The regular-domain construction of
\cite[Corollary~5.2 and Theorem~5.2]{Bony1969} gives, along these curves, a
finite collection of smooth domains whose closures are compact and satisfy
\[
\overline{\mathscr U_\ell'}\subset\mathscr U_\ell,
\qquad
\overline{\mathscr U_\ell}\subset\mathcal Q,
\qquad 1\le\ell\le N,
\]
such that consecutive domains overlap in the direction of the control flow.
After shrinking them, all control tubes and all comparison points in this
finite chain
have distance at least $d_0>0$ from the relevant boundaries.

Fix $\beta>\kappa_\dagger$, where $\kappa_\dagger$ is a uniform upper bound
for the zero-order coefficients of the formal adjoints.  The forward and
adjoint regularity of $\mathscr U_\ell$ provides barriers
$\rho_\ell^+,\rho_\ell^-$ and constants $c_\ell>0$ such that, on the
corresponding boundary neighborhoods,
\begin{equation}
\label{eq:bony-uniform-barriers-limit}
(\mathscr P_0-\beta)\rho_\ell^+\le-c_\ell,
\qquad
(\mathscr P_0^*-\beta)\rho_\ell^-\le-c_\ell.
\end{equation}

\emph{Step 2: uniformity of the local comparison constants.}
Since $\mathscr P_\nu\to\mathscr P_0$ in
$C^\infty(\overline{\mathcal Q})$, for all sufficiently large $\nu$,
\begin{equation}
\label{eq:bony-uniform-barriers-family}
(\mathscr P_\nu-\beta)\rho_\ell^+\le-\frac{c_\ell}{2},
\qquad
(\mathscr P_\nu^*-\beta)\rho_\ell^-\le-\frac{c_\ell}{2},
\end{equation}
on the same boundary neighborhoods, simultaneously for
$1\le\ell\le N$.  Thus every $\mathscr U_\ell$ is regular for the two shifted
Dirichlet problems with a margin independent of $\nu$.

For each link of the finite control chain, the proof of
\cite[Theorem~7.1]{Bony1969} uses a compact set
\[
E_\ell\subset
(\mathscr U_\ell\times\mathscr U_\ell)
\setminus\{(x,x):x\in\mathscr U_\ell\}
\]
of comparison pairs.  Positivity and continuity of the limiting Green kernel
give
\begin{equation}
\label{eq:bony-limit-green-lower-bound}
g_\ell:=\min_{(x,y)\in E_\ell}
G_{0,\beta}^{\mathscr U_\ell}(x,y)>0.
\end{equation}
By Lemma~\ref{lem:shifted-green-kernel-stability},
\eqref{eq:bony-uniform-barriers-family} and
\eqref{eq:bony-limit-green-lower-bound} imply
\begin{equation}
\label{eq:bony-family-green-lower-bound}
G_{\nu,\beta}^{\mathscr U_\ell}(x,y)
\ge\frac{g_\ell}{2},
\qquad (x,y)\in E_\ell,\quad 1\le\ell\le N,
\end{equation}
for every sufficiently large $\nu$.  Because the collection is finite,
$g_*:=\frac12\min_{1\le\ell\le N}g_\ell>0$ is a common lower bound.

\emph{Step 3: iteration along the control chains.}
Apply \cite[Proposition~7.1]{Bony1969} successively on the domains
$\mathscr U_\ell$.  The coefficient seminorms, bracket lower bound,
boundary margins in \eqref{eq:bony-uniform-barriers-family}, and Green lower
bound in \eqref{eq:bony-family-green-lower-bound} are uniform in $\nu$.
Consequently, the comparison constants in the finitely many links may be
chosen as numbers $C_1,\ldots,C_N$ independent of $\nu$.  Enumerate the links
in an order compatible with their control curves, and let
$p(\ell)\in\{0,\ldots,\ell-1\}$ denote the preceding link on the same curve
($p(\ell)=0$ for its first link).  Put $M_0=u_\nu(z_\bullet)$, and let
$M_\ell$ denote the supremum of $u_\nu$ on the terminal neighborhood of the
$\ell$th link.  The successive comparisons give
\[
M_\ell\le C_\ell M_{p(\ell)},
\qquad 1\le\ell\le N.
\]
Since the terminal neighborhoods cover $K_1$,
\begin{equation}
\label{eq:bony-family-sup-bound}
\sup_{z\in K_1}u_\nu(z)
\le\left(\prod_{\ell=1}^NC_\ell\right)u_\nu(z_\bullet).
\end{equation}

\emph{Step 4: derivative estimate and non-negative solutions.}
Choose fixed open sets $\mathcal O_0,\mathcal O_1$ such that
\[
K_0\subset\mathcal O_0,\qquad
\overline{\mathcal O_0}\subset\mathcal O_1,\qquad
\overline{\mathcal O_1}\subset\operatorname{int}K_1,
\]
and on which the same finite bracket family has a uniform least singular
value.  The parameter-uniform interior subelliptic
estimate of \cite{RothschildStein1976}, first with a common positive gain in
Sobolev regularity and then iterated, gives for every integer $k\ge0$
\[
\|u_\nu\|_{H^k(\mathcal O_0)}
\le C_k\|u_\nu\|_{L^2(\mathcal O_1)}.
\]
Here the constants are uniform in $\nu$: the domains are fixed, the relevant
coefficient seminorms are bounded by the $C^\infty$ convergence, and the
quantitative bracket lower bound is uniform.  Taking $k$ above the Sobolev
embedding threshold and using
\eqref{eq:bony-family-sup-bound} therefore yields
\[
\|\nabla u_\nu\|_{L^\infty(K_0)}
\le C\sup_{z\in K_1}u_\nu(z).
\]
Combining this with \eqref{eq:bony-family-sup-bound} proves the asserted
estimate.  Finally, if $u_\nu\ge0$, then
$\mathscr P_\nu(u_\nu+\delta)=0$ for every $\delta>0$.  Apply the positive
case to $u_\nu+\delta$ and let $\delta\downarrow0$.
\end{proof}

We apply the compact-family estimate to the rescaled generator
near $x_m$.  The third-order bracket structure supplies uniform
hypoellipticity, while a time lift converts the constant right-hand side into
a homogeneous equation.

\begin{lemma}[Uniform local Harnack--regularity estimate]
\label{lem:uniform-local-harnack-holder}
There is a fixed $R>4$ such that the following holds.  Write
\[
\mathscr L_\eps
=V\cdot\nabla_Q
-L^{-1}\frac{\nabla U(m+\sqrt\eps Q)}{\sqrt\eps}\cdot\nabla_V
+\gamma W\cdot\nabla_V
+(-\gamma V-\gamma W)\cdot\nabla_W
+\frac{\gamma}{L}\Delta_W .
\]
There exists some $C<\infty$, independent of small $\eps$, such that every
non-negative smooth solution of
$\mathscr L_\eps v=-a$ on $B_R$, $a\ge0$, satisfies
\begin{equation}
\label{eq:uniform-local-harnack-gradient}
\sup_{y\in B_2}v(y)\le C\left(v(0)+a\right),
\qquad
\|\nabla v\|_{L^\infty(B_1)}\le C\left(v(0)+a\right).
\end{equation}
Consequently, for every fixed $\alpha\in(0,1]$ after changing $C$,
\[
|v(y)-v(0)|\le C\left(v(0)+a\right)|y|^\alpha,
\quad y\in B_1.
\]
\end{lemma}

\begin{proof}
Let $Y_\eps$ denote the first-order part of $\mathscr L_\eps$.  On $B_R$ its
coefficients are bounded uniformly in every $C^k$ norm and converge to those
of the stable OU field $Y_0$.  With $X_i=\sqrt{\gamma/L}\,\partial_{W_i}$,
the diffusion fields and their first two commutators with $Y_\eps$ contain,
uniformly in $\eps$,
\begin{equation}
\label{eq:uniform-hormander-brackets}
\partial_{W_i},\qquad
\gamma\partial_{V_i}-\gamma\partial_{W_i},\qquad
\gamma\partial_{Q_i}+\text{a linear combination of }\partial_V,\partial_W.
\end{equation}
Equivalently, if $G$ is the noise matrix and $A_m=DY_0(0)$, then
$[G,A_mG,A_m^2G]$ has rank $3d$.  Its least singular value is positive, and
the bracket coefficients in
\eqref{eq:uniform-hormander-brackets}  show that the corresponding quantitative
rank bound persists on $B_R$ for any sufficiently small $\eps$.

We reduce the non-homogeneous equation to a positive homogeneous one before
using Bony's Harnack estimate.  On the cylinder
$\mathcal Q=B_R\times(-T,T/2)$ set
\[
\widehat{\mathscr L}_\eps=\mathscr L_\eps+\partial_t,
\qquad
w(y,t)=v(y)+a(t+T).
\]
Then $w\ge0$ and $\widehat{\mathscr L}_\eps w=0$ on $\mathcal Q$.  The lifted
fields $X_i$ and $Y_\eps+\partial_t$ satisfy the full-rank condition in
dimension $3d+1$: the spatial brackets span the $y$ variables, and subtracting
their smooth linear combination from $Y_\eps+\partial_t$ gives the time
direction.

We next formulate the estimate on an open space--time neighborhood so that
the interior derivative estimate applies.  The
controllability Gramian of the limiting linear system,
\[
  \Gamma_\tau=\int_0^\tau e^{sA_m}GG^\top e^{sA_m^\top}\,ds,
\]
is positive definite for every $\tau>0$.  Choose $T>0$, $R>4$, and
$0<\eta<T/4$ so that
the compact cylinders
\[
 K_0=\overline{B_1}\times[-\eta/2,\eta/2],
 \qquad K_1=\overline{B_3}\times[-2\eta,2\eta]
\]
lie in $\mathcal Q$.  Linear controllability provides controlled curves from
$(0,-T/2)$ to the points of $K_1$ over the time intervals
\[
I_{T,\eta}
:=\left[T/2-2\eta,T/2+2\eta\right]\subset(0,\infty).
\]
Since $\tau\mapsto\Gamma_\tau$ is continuous,
\[
\inf_{\tau\in I_{T,\eta}}
\lambda_{\min}(\Gamma_\tau)>0.
\]
Thus the controls may be chosen locally uniformly in the target point and
the terminal time.  After increasing $R$ if necessary, the corresponding
curves all stay a fixed
distance inside $\mathcal Q$.  Compactness of $K_1$ reduces these curves to a
finite family.  Since $Y_\eps\to Y_0$ in $C^k(B_R)$ for every fixed $k$,
continuous dependence of the controlled ODE and a small perturbation of the
controls give the same finite family of control tubes for every small $\eps$.

Taking $\nu=\eps^{-1}$, the hypotheses of
Lemma~\ref{lem:compact-family-bony} are now satisfied by
$\widehat{\mathscr L}_\eps$: the coefficient family converges in $C^\infty$,
the finite bracket family in
\eqref{eq:uniform-hormander-brackets}  has a uniform least singular value, and
the control cover and its distance from $\partial\mathcal Q$ are fixed.
Moreover,
$\sum_{i=1}^d|X_i|^2=\gamma d/L>0$, so the non-total-degeneracy constant is
uniform.
Applying that lemma to $w+\delta$ and sending $\delta\downarrow0$ gives the
following estimate.  Here $w=w(y,t)$ is a scalar function of the space--time variables;
the supremum ranges over both $y$ and $t$, whereas $\nabla_y$ differentiates
only in the spatial variable.  Thus
\[
\sup_{(y,t)\in K_1}w(y,t)
+\sup_{(y,t)\in K_0}|\nabla_yw(y,t)|
\le Cw(0,-T/2),
\]
with $C$ independent of $\eps$.  The open neighborhood
of $K_1$ supplies the domain required by the interior derivative
estimate.

Since $w(\cdot,0)=v+aT$,
$w(0,-T/2)=v(0)+aT/2$, and spatial derivatives of $w$ equal those of $v$, the
estimate \eqref{eq:uniform-local-harnack-gradient} follows after absorbing
the fixed $T$ into $C$.
For $y\in B_1$, the segment $\{sy:0\le s\le1\}$ lies in $B_1$, and hence
\[
|v(y)-v(0)|
=\left|\int_0^1\nabla v(sy)\cdot y\,ds\right|
\le \|\nabla v\|_{L^\infty(B_1)}|y|
\le C\left(v(0)+a\right)|y|.
\]
If $\alpha\in(0,1]$, then $|y|\le|y|^\alpha$ for $|y|\le1$, which proves the
asserted H\"older estimate.  If instead $v$ is initially given only as a
non-negative distributional solution of
$\mathscr L_\eps v=-a$ on $B_R$, then the right-hand side is smooth and the
bracket condition \eqref{eq:uniform-hormander-brackets} holds, so H\"ormander
hypoellipticity \cite{Hormander1967} implies that $v$ has a smooth
representative in the interior.  The preceding estimate then applies to that
representative directly.
\end{proof}

The uniform Harnack--regularity estimate now converts the
fixed-time lower bound of
Lemma~\ref{lem:nondegenerate-mean-transition-time} into spatial flatness on
shrinking neighborhoods of the well.  This is the final ingredient needed to
replace the weak equilibrium average by the value at $x_m$.

\begin{lemma}[Flatness of the hitting expectation]
\label{lem:intrawell-flatness-reduction}
For each $\varrho\in(1/2,1]$,
\[
\E_z[\tau_{S_\eps}]
=\E_{x_m}[\tau_{S_\eps}]
\left[1+\mathcal O\left(\eps^{\varrho-1/2}\right)\right], 
\]
uniformly on $G_{\eps,\varrho}$.
This is the pointwise estimate in Proposition~\ref{prop:local-flatness}.
\end{lemma}

\begin{proof}
We divide the proof into three steps.

\emph{Step 1: local finiteness and the Poisson equation.}
Set
\[
u_\eps(z):=\E_z[\tau_{S_\eps}].
\]
By Proposition~\ref{prop:hitting-identity},
\begin{equation}
\label{eq:flatness-finite-boundary-average}
\int_{\partial A_\eps}u_\eps(z)\,\nu_\eps(dz)<\infty.
\end{equation}
Since $u_\eps\ge0$ and $\nu_\eps$ is a probability measure, there is
$z_\eps^\partial\in\partial A_\eps$ such that
$u_\eps(z_\eps^\partial)<\infty$.

Let $(\mathscr D_R)_{R\ge1}$ be an increasing smooth bounded exhaustion of
$\R^{3d}$ containing $A_\eps\cup S_\eps$, and define
\[
u_{\eps,R}(z)
:=\E_z\left[\tau_{S_\eps}\wedge\tau_{\partial\mathscr D_R}\right].
\]
Then
\begin{equation}
\label{eq:flatness-exhaustion-monotonicity}
0\le u_{\eps,R}\uparrow u_\eps,
\qquad
u_{\eps,R}\left(z_\eps^\partial\right)
\le u_\eps\left(z_\eps^\partial\right)<\infty.
\end{equation}
Local Kalman controllability at $x_m$ allows us to choose finitely many
overlapping hypoelliptic cylinders, connected in the forward admissible
direction by controlled curves, from $z_\eps^\partial$ to a fixed
neighborhood $\mathscr U_m$ of $x_m$.  The closures of the cylinders lie in a
fixed compact set disjoint from $S_\eps$ and, for all sufficiently large $R$,
inside $\mathscr D_R$.  Applying
\cite[Proposition~7.1 and Theorem~7.1]{Bony1969} on this finite chain, after
the time lift of Lemma~\ref{lem:uniform-local-harnack-holder}, implies that there exists a constant
$C_\eps<\infty$, independent of $R$, such that
\begin{equation}
\label{eq:flatness-exhaustion-harnack-bound}
\sup_{z\in\mathscr U_m}u_{\eps,R}(z)
\le C_\eps\left(u_{\eps,R}\left(z_\eps^\partial\right)+1\right).
\end{equation}
Letting $R\to\infty$ in
\eqref{eq:flatness-exhaustion-harnack-bound} and using
\eqref{eq:flatness-exhaustion-monotonicity} shows that $u_\eps$ is finite on
$\mathscr U_m$.  For every sufficiently large $R$, the bounded stopped
expectation $u_{\eps,R}$ satisfies the interior Poisson equation
$\calL_\eps u_{\eps,R}=-1$ on $\mathscr U_m$.  Equivalently, for every
$\varphi\in C_c^\infty(\mathscr U_m)$,
\begin{align}\label{u:L:varphi:eqn}
\int_{\mathscr U_m}u_{\eps,R}(z)\calL_\eps^*\varphi(z)\,dz
=-\int_{\mathscr U_m}\varphi(z)\,dz.
\end{align}
The bound \eqref{eq:flatness-exhaustion-harnack-bound} is uniform in $R$ on
$\mathscr U_m$.  Hence dominated convergence, together with
$u_{\eps,R}\uparrow u_\eps$, permits passage to the limit in \eqref{u:L:varphi:eqn}.  Thus $\calL_\eps u_\eps=-1$ in
$\mathcal D'(\mathscr U_m)$; local hypoellipticity then
gives
\begin{equation}
\label{eq:flatness-local-poisson-equation}
u_\eps\in C^\infty(\mathscr U_m),
\qquad
\calL_\eps u_\eps=-1\quad\text{in }\mathscr U_m.
\end{equation}

\emph{Step 2: well scaling and a uniform gradient bound.}
Let $R_H>4$ be the fixed radius in
Lemma~\ref{lem:uniform-local-harnack-holder}.  For all sufficiently small
$\eps$,
\[
\overline{x_m+\sqrt\eps B_{R_H}}\subset\mathscr U_m,
\qquad
\left(x_m+\sqrt\eps B_{R_H}\right)\cap S_\eps=\varnothing.
\]
Define
\[
v_\eps(y):=
\frac{u_\eps(x_m+\sqrt\eps y)}{u_\eps(x_m)},
\qquad |y|<R_H,
\qquad
\kappa_\eps:=\frac1{u_\eps(x_m)}.
\]
By \eqref{eq:flatness-local-poisson-equation},
\begin{equation}
\label{eq:flatness-rescaled-poisson-equation}
v_\eps\ge0,
\qquad v_\eps(0)=1,
\qquad \mathscr L_\eps v_\eps=-\kappa_\eps
\quad\text{in }B_{R_H}.
\end{equation}
Lemma~\ref{lem:nondegenerate-mean-transition-time} yields
$0<\kappa_\eps\le c_*^{-1}$.  Applying
Lemma~\ref{lem:uniform-local-harnack-holder} to
\eqref{eq:flatness-rescaled-poisson-equation} therefore implies that
\begin{equation}
\label{eq:flatness-rescaled-gradient-bound}
|v_\eps(y)-1|\le C|y|,
\qquad |y|\le1,
\end{equation}
where $C$ is independent of $\eps$.

\emph{Step 3: restriction to the shrinking well ball.}
If $z\in G_{\eps,\varrho}$ and
$y=(z-x_m)/\sqrt\eps$, then
\[
|y|\le\eps^{\varrho-1/2}=o(1).
\]
Thus \eqref{eq:flatness-rescaled-gradient-bound} implies that, uniformly for
$z\in G_{\eps,\varrho}$,
\[
\left|
\frac{u_\eps(z)}{u_\eps(x_m)}-1
\right|
\le C\eps^{\varrho-1/2},
\]
which proves the asserted flatness estimate.
\end{proof}

\subsection{Local Lyapunov and Kernel Bounds for Committor Localization}
\label{app:local-kernel-committor}

The estimates in this subsection serve the killed-process argument leading
to committor localization.  They are independent of the intrawell-flatness
proof above.

We first establish a local contraction estimate near either minimum.  A
quadratic Lyapunov function for the stable linearization controls returns to
the natural $\sqrt\eps$ scale.

For the remainder of this appendix, set
\[
\mathsf A_\eps^m:=A_\eps,
\qquad
\mathsf A_\eps^s:=S_\eps.
\]

\begin{lemma}[Local hypocoercive Lyapunov functions near the wells]
\label{lem:local-hypocoercive-lyapunov}
There are $r_0,c,C>0$ such that, for each $i\in\{m,s\}$, there is a
positive definite quadratic form
$V_i(Y)=Y^\top P_iY$, $Y=(\theta-i,p,r)$, with
$V_i\asymp|Y|^2$ and, on $B(x_i,r_0)$,
\[
\calL_\eps V_i\le-cV_i+C\eps.
\]
The constants may be chosen simultaneously for the two wells.
\end{lemma}

\begin{proof}
Fix $i\in\{m,s\}$.  Since $i$ is a non-degenerate local minimum,
$H_i=\nabla^2U(i)$ is positive definite.  The linearization $A_i$ at $x_i$
has, in each $H_i$-eigendirection with
eigenvalue $\lambda>0$, characteristic polynomial
\[
\mu^3+\gamma\mu^2+(\gamma^2+\lambda/L)\mu+\gamma\lambda/L,
\]
which is Hurwitz;
so $A_i$ is Hurwitz.  Solve $A_i^\top P_i+P_iA_i=-I$, $P_i>0$, and set
$V_i=Y^\top P_iY$.  Write the drift near $x_i$ as
\[
b(x_i+Y)=A_iY+R_i(Y),
\qquad |R_i(Y)|\le C_R|Y|^2.
\]
If $P_i^{rr}$ denotes the $r$--$r$ block of $P_i$, then direct application
of the generator gives
\[
\begin{aligned}
\calL_\eps V_i
&=Y^\top\left(A_i^\top P_i+P_iA_i\right)Y
  +2Y^\top P_iR_i(Y)
  +\frac{2\gamma\eps}{L}\operatorname{tr}P_i^{rr}\\
&\le-|Y|^2+2\|P_i\|C_R|Y|^3
  +\frac{2\gamma\eps}{L}\operatorname{tr}P_i^{rr}.
\end{aligned}
\]
Choose $r_i>0$ so that $2\|P_i\|C_Rr_i\le1/2$.  Since
$V_i(Y)\le\lambda_{\max}(P_i)|Y|^2$, for $|Y|\le r_i$ we obtain
\[
\calL_\eps V_i
\le-\frac{1}{2\lambda_{\max}(P_i)}V_i
+\frac{2\gamma}{L}\operatorname{tr}(P_i^{rr})\eps.
\]
Taking the minimum of the two radii and adjusting the constants proves the
assertion simultaneously for $i=m,s$.
\end{proof}

\subsubsection{Local Kernel Bounds}

The estimates of Pigato are formulated under a single bound for derivatives
of every order.  The applications in
Lemmas~\ref{lem:rescaled-well-kernel-convergence}
and~\ref{lem:third-order-small-time-density-upper} require a zeroth-order
density bound and finitely many inverse-covariance moments.  We therefore record
explicitly the corresponding finite-order uniform consequence.

\begin{lemma}[Finite-order uniform form of the chain density estimate]
\label{lem:finite-order-pigato}
Fix the chain length $n=3$, the block dimension $d$, and $T<\infty$.
Fix also $p>2$ and an inverse-moment order $q<\infty$.  There is an integer
$N=N(d,p,q)$ with the following property.  Consider a family of smooth
third-order chain diffusions in the block order
$X^1,X^2,X^3$, centered at the deterministic trajectory as in
\cite[(D2)]{Pigato2022}.  Suppose uniformly over the family that:
\begin{enumerate}[label=\textup{(P\arabic*)},leftmargin=2.6em]
\item the initial points and initial drift values are bounded;
\item the weak H\"ormander matrix in hypothesis \textup{(H1)} of
\cite{Pigato2022} has smallest singular value at least $\lambda_0>0$;
\item all positive-order spatial derivatives of orders $1,\ldots,N$, together
with their first time derivatives, are bounded on
$[0,T]\times\R^{3d}$;
\item the diffusion coefficient is bounded together with the derivatives
through the same order.
\end{enumerate}
Let
\[
\mathsf T_t
:=\operatorname{diag}\left(t^{1/2}I_d,t^{3/2}I_d,t^{5/2}I_d\right)
\]
and let $\Gamma_t^{\rm sc}$ be the Malliavin covariance of the centered,
scaled variable $\mathsf T_t^{-1}(X_t-\vartheta_t)$, where
$\vartheta_t$ is the deterministic trajectory from \cite[(D2)]{Pigato2022}.
Then
\begin{equation}
\label{eq:finite-order-pigato-inverse}
\sup_{0<t\le T}\E\!\left[
\lambda_{\min}(\Gamma_t^{\rm sc})^{-q}\right]\le C_q,
\end{equation}
where, for a symmetric non-negative matrix $M$,
$\lambda_{\min}(M):=\inf_{|v|=1}v^\top Mv$
denotes its smallest eigenvalue.
For the same statement, after taking $N$ to be the maximum of the finitely
many orders used in the two conclusions, the transition density
$q(t,\xi,y)$
satisfies
\begin{equation}
\label{eq:finite-order-pigato-density}
q(t,\xi,y)
\le
\frac{C_p t^{-9d/2}}
{1+|\mathsf T_t^{-1}(y-\vartheta_t)|^p},
\qquad 0<t\le T,\quad y\in\R^{3d}.
\end{equation}
The constants depend only on
$d,T,p,q,\lambda_0$ and the finitely many coefficient bounds above.
\end{lemma}

\begin{proof}
We divide the proof into three steps.

\emph{Step 1: inverse covariance moments.}
Let $q_{\rm den}=q_{\rm den}(d,p)$ be a finite inverse-covariance moment order
sufficient for the zeroth-order integration-by-parts estimate used in Step~2,
and set $\bar q=\max\{q,q_{\rm den}\}$.  In the proof of
\cite[Lemma~3.5]{Pigato2022}, the estimate of
$\lambda_{\min}(\Gamma_t^{\rm sc})^{-\bar q}$ is obtained by a finite sequence
of variational-equation estimates and applications of Norris'
semimartingale inequality \cite[Lemma~4.1]{Norris1986}.
The latter converts smallness of the quadratic variation of a semimartingale
into smallness of its drift and martingale coefficients, with a quantitative
probability bound.  We follow its implementation in the proof of
\cite[Lemma~3.5]{Pigato2022}; see also \cite[Section~2.3]{Nualart2006}.
Let
$N_{\rm inv}=N_{\rm inv}(d,\bar q)$ be the largest coefficient-derivative
order occurring in that sequence.  Under \textup{(P1)}--\textup{(P4)}, the
quantitative H\"ormander constant is at least $\lambda_0$ and all coefficient
norms through order $N_{\rm inv}$ are uniformly bounded.  Therefore the
constants in that proof are uniform over the present family, and
\begin{equation}
\label{eq:finite-order-pigato-inverse-proof}
\sup_{0<t\le T}\E\!\left[
\lambda_{\min}(\Gamma_t^{\rm sc})^{-\bar q}\right]
\le C_{\bar q}.
\end{equation}
Since $q\le\bar q$, this also proves
\eqref{eq:finite-order-pigato-inverse}, after changing the constant.

\emph{Step 2: the zeroth-order density estimate.}
Fix $p>2$ and take the density multi-index $\alpha=\varnothing$.
The Malliavin integration-by-parts estimate
\cite[Appendix~A, Lemma~A.1]{Pigato2022} then uses only finitely many
Malliavin--Sobolev norms.  Let
$N_{\rm den}=N_{\rm den}(d,p,\bar q)$ be the largest
coefficient-derivative order needed for those norms and for the moment bound
in the proof of \cite[Theorem~2.1(i)]{Pigato2022}.  Assumptions
\textup{(P1)}--\textup{(P4)}, together with
\eqref{eq:finite-order-pigato-inverse-proof}, give uniformly
\[
q(t,\xi,y)
\le
\frac{C_p}{|\det\mathsf T_t|}
\frac{1}{1+|\mathsf T_t^{-1}(y-\vartheta_t)|^p},
\qquad 0<t\le T.
\]
Since
\[
|\det\mathsf T_t|
=\left(t^{1/2}t^{3/2}t^{5/2}\right)^d=t^{9d/2},
\]
this is exactly \eqref{eq:finite-order-pigato-density}.

\emph{Step 3: choice of the finite regularity order.}
Set
\[
N(d,p,q):=
\max\left\{N_{\rm inv}(d,\bar q),N_{\rm den}(d,p,\bar q)\right\}.
\]
Thus the constants in Steps~1--2 depend on the parameters and the finite
collection of coefficient norms listed in the statement.  This completes the
proof.
\end{proof}

To turn the Lyapunov contraction into a quantitative probability
of hitting the much smaller target ball, we need a local density lower bound.
We first establish convergence of the rescaled cutoff kernel to its
controllable Ornstein--Uhlenbeck limit.

\begin{lemma}[Uniform convergence of a localized rescaled well kernel]
\label{lem:rescaled-well-kernel-convergence}
For $i\in\{m,s\}$, let $A_i$ be the linearized drift matrix at $x_i$ and let
$P_i^{\rm OU}(t,x,y)$ be the transition density of
\begin{equation}
\label{eq:well-ou-process}
dY_t=A_iY_t\,dt+\sqrt{2\gamma/L}\,(0,0,dB_t).
\end{equation}
Fix the common radius $r_0$ of
Lemma~\ref{lem:local-hypocoercive-lyapunov}.  Extend
only the rescaled force smoothly outside $B(0,r_0/\sqrt\eps)$, keeping its
positive-order derivatives uniformly bounded and leaving the linear chain
$W\to V\to Q$ unchanged.  Denote the density of this chain-compatible cutoff
process by
$\bar P_{i,\eps}(t,x,y)$.  The cutoff agrees with the true rescaled process
about $x_i$ up to
exit from $B(0,r_0/\sqrt\eps)$.
For every $0<t_-<t_+<\infty$ and compact $K_0,K_1\subset\R^{3d}$,
\[
\sup_{\substack{t\in[t_-,t_+]\\x\in K_0,\,y\in K_1}}
\left|
\bar P_{i,\eps}(t,x,y)
-P_i^{\rm OU}(t,x,y)
\right|\longrightarrow0
\qquad\text{as $\eps\downarrow0$}.
\]
The convergence holds for both $i=m,s$.
\end{lemma}

\begin{proof}
Fix $i\in\{m,s\}$.  The rescaled process
$Y_t^\eps=(Z_t^\eps-x_i)/\sqrt\eps$ solves
\[
\begin{aligned}
dQ_t&=V_tdt,\\
dV_t&=\left[-L^{-1}\frac{\nabla U(i+\sqrt\eps Q_t)}{\sqrt\eps}
              +\gamma W_t\right]dt,\\
dW_t&=(-\gamma V_t-\gamma W_t)dt+\sqrt{2\gamma/L}\,dB_t.
\end{aligned}
\]
On every fixed ball its drift converges in $C^k$, for every fixed $k$, to
$A_iY$.  Put
\[
 F_{i,\eps}(Q)=\frac{\nabla U(i+\sqrt\eps Q)}{\sqrt\eps}.
\]
Choose $\chi\in C^\infty([0,\infty);[0,1])$ with
$\chi=1$ on $[0,1]$ and $\chi=0$ on $[2,\infty)$, and define
\[
\bar F_{i,\eps}(Q)
=\chi\!\left(\frac{\sqrt\eps|Q|}{r_0}\right)F_{i,\eps}(Q)
{}+\left[1-\chi\!\left(\frac{\sqrt\eps|Q|}{r_0}\right)\right]H_iQ.
\]
Then $\bar F_{i,\eps}=F_{i,\eps}$ on $|Q|\le r_0/\sqrt\eps$, it has at most linear
growth, and
\begin{equation}
\label{eq:uniform-cutoff-force-derivatives}
 \sup_{\substack{i\in\{m,s\}\\\eps\le\eps_0}}
 \left\|D^k\bar F_{i,\eps}\right\|_\infty<\infty,
 \qquad k\ge1.
\end{equation}
Indeed, on the interpolation annulus the physical variable
$i+\sqrt\eps Q$ remains in the fixed compact ball $B(i,2r_0)$,
$DF_{i,\eps}=\nabla^2U(i+\sqrt\eps Q)$, and
\[
D^kF_{i,\eps}=\eps^{(k-1)/2}D^{k+1}U\left(i+\sqrt\eps Q\right).  
\]
Every derivative falling
on the cutoff contributes a factor $\sqrt\eps$; since
$|F_{i,\eps}(Q)|+|H_iQ|\le C\eps^{-1/2}$ on the annulus,
\eqref{eq:uniform-cutoff-force-derivatives} follows.
Define the cutoff drift by
\[
 \bar B_{i,\eps}(Q,V,W)
 =\left(V,-L^{-1}\bar F_{i,\eps}(Q)+\gamma W,
              -\gamma V-\gamma W\right).
\]
Let $Y_t^{\eps,x}$ denote this cutoff process with initial point $x$.
We apply the chain estimates of \cite{Pigato2022} in the block order
$X^1=W$, $X^2=V$, $X^3=Q$.  The diffusion matrix in the first block is
$\sigma_0=\sqrt{2\gamma/L}\,I_d$, while the two chain Jacobians are exactly
\[
J_W(\bar B_{i,\eps})_V=\gamma I_d,
\qquad
J_V(\bar B_{i,\eps})_Q=I_d.
\]
Equivalently, with
$G=(0,0,\sqrt{2\gamma/L}\,I_d)^\top$ in the block order $(Q,V,W)$,
\begin{equation}
\label{eq:well-ou-kalman-rank}
\operatorname{rank}\left[G,A_iG,A_i^2G\right]=3d.
\end{equation}
Thus, the non-degeneracy constant in hypothesis (H1) of that paper can be
chosen uniformly as
\[
\lambda_0
=\sigma_{\min}\!\left(J_V(\bar B_{i,\eps})_Q
J_W(\bar B_{i,\eps})_V\sigma_0\right)
=\gamma\sqrt{2\gamma/L}>0,
\]
where 
$\sigma_{\min}(M):=\inf_{|v|=1}|Mv|$
is the smallest singular value of a square matrix $M$.
For the finite order required by
Lemma~\ref{lem:finite-order-pigato}, the coefficient seminorms are uniform in
$\eps$ by the bounds on $D^k\bar F_{i,\eps}$ above.  The starting points
$x\in K_0$ and their drift values are uniformly bounded.  Since
$t\in[t_-,t_+]$, both the anisotropic scaling matrix $\mathsf T_t$ and its
inverse are bounded.  Hence
\eqref{eq:finite-order-pigato-inverse}, after undoing this bounded scaling,
gives, for every fixed $q<\infty$,
\begin{equation}
\label{eq:uniform-well-malliavin-inverse}
 \sup_{\substack{0<\eps\le\eps_0,\ t\in[t_-,t_+]\\x\in K_0}}
 \E_x\!\left[\lambda_{\min}(\bar\Gamma_{i,t}^\eps)^{-q}\right]<\infty,
\end{equation}
where $\bar\Gamma_{i,t}^\eps$ is the Malliavin covariance of the rescaled cutoff
process.

Let $Y_t^{0,x}$ be the Ornstein--Uhlenbeck process with drift $A_i$ and initial
point $x$, and
couple $Y^{\eps,x}$ and $Y^{0,x}$ with the same Brownian motion.  Taylor's
formula implies that, for every fixed $R,N<\infty$,
\begin{equation}
\label{eq:well-local-drift-convergence}
\max_{0\le j\le N}\sup_{|y|\le R}
\left|D^j\left(\bar B_{i,\eps}(y)-A_iy\right)\right|
\le C_{R,N}\sqrt\eps .
\end{equation}
Indeed, the zeroth-order force remainder is
$\mathcal O(\sqrt\eps|Q|^2)$, its first derivative is
$\mathcal O(\sqrt\eps|Q|)$, and every derivative of order at least two
contains a factor $\eps^{(j-1)/2}$.

Define the common localization time
\[
\tau_R^\eps
:=\inf\left\{t\ge0:
\left|Y_t^{\eps,x}\right|\vee\left|Y_t^{0,x}\right|\ge R\right\}.
\]
Set $\Delta_t^\eps=Y_t^{\eps,x}-Y_t^{0,x}$.  Because the two processes are
driven by the same Brownian motion and have the same constant noise matrix,
the stochastic terms cancel and
\begin{equation}
\label{eq:well-process-difference-equation}
\Delta_t^\eps
=\int_0^t\left[\bar B_{i,\eps}\left(Y_s^{\eps,x}\right)
                 -\bar B_{i,\eps}\left(Y_s^{0,x}\right)\right]ds
 +\int_0^t\left[\bar B_{i,\eps}\left(Y_s^{0,x}\right)-A_iY_s^{0,x}\right]ds.
\end{equation}
Let $J_t^\eps=\partial_xY_t^{\eps,x}$ and let
$G=\left(0,0,\sqrt{2\gamma/L}\,I_d\right)^\top$ in the block order $(Q,V,W)$.
The first variational equations are
\begin{equation}
\label{eq:well-first-variational-equations}
\begin{aligned}
J_t^\eps
&=I+\int_0^tD\bar B_{i,\eps}(Y_s^{\eps,x})J_s^\eps\,ds,\\
D_rY_t^{\eps,x}
&=G+\int_r^tD\bar B_{i,\eps}(Y_s^{\eps,x})D_rY_s^{\eps,x}\,ds,
\qquad 0\le r\le t.
\end{aligned}
\end{equation}
For $k\ge2$, writing $\mathbf r=(r_1,\ldots,r_k)$ and
$r_*=\max_j r_j$, the higher Malliavin derivatives satisfy
\begin{equation}
\label{eq:well-higher-variational-equation}
D_{\mathbf r}^kY_t^{\eps,x}
=\int_{r_*}^tD\bar B_{i,\eps}(Y_s^{\eps,x})
                 D_{\mathbf r}^kY_s^{\eps,x}\,ds
 +\int_{r_*}^t\mathcal R_{k,\eps}(s,\mathbf r)\,ds,
\end{equation}
where $\mathcal R_{k,\eps}$ is the finite Fa\`a di Bruno sum of terms
\[
D^\ell\bar B_{i,\eps}(Y_s^{\eps,x})
\left[D_{\mathbf r_{B_1}}^{|B_1|}Y_s^{\eps,x},\ldots,
      D_{\mathbf r_{B_\ell}}^{|B_\ell|}Y_s^{\eps,x}\right],
\qquad \ell\ge2,
\]
indexed by partitions $\{B_1,\ldots,B_\ell\}$ of
$\{1,\ldots,k\}$.  The OU process satisfies the corresponding equations
with $\bar B_{i,\eps}$ replaced by $A_i(\cdot)$.

On $\{\tau_R^\eps>t_+\}$, subtracting these equations, using
\eqref{eq:well-local-drift-convergence}, and applying Gr\"onwall's inequality
inductively in $k$ gives, for every fixed $N,p$,
\begin{equation}
\label{eq:well-local-malliavin-difference}
\sup_{\substack{t\in[0,t_+]\\x\in K_0}}
\E_x\!\left[\mathbf1_{\{\tau_R^\eps>t_+\}}
\left(
|\Delta_t^\eps|^p
+\sum_{k=1}^N
\left\|D^kY_t^{\eps,x}-D^kY_t^{0,x}\right\|_{\mathfrak H^{\otimes k}}^p
\right)\right]
\le C_{R,N,p}\eps^{p/2},
\end{equation}
where $\mathfrak H:=L^2([0,t_+];\R^d)$.  Moreover, the cutoff drifts are
uniformly Lipschitz with at most linear growth, and their derivatives through
each fixed positive order are uniformly bounded.  The
Burkholder--Davis--Gundy and Gr\"onwall inequalities applied to the SDE and
\eqref{eq:well-first-variational-equations}--
\eqref{eq:well-higher-variational-equation} therefore give
\begin{equation}
\label{eq:well-uniform-variational-moments}
\sup_{\substack{0<\eps\le\eps_0\\x\in K_0}}
\left\{
\E_x\!\left[\sup_{s\le t_+}|Y_s^{\eps,x}|^p
             +\sup_{s\le t_+}|J_s^\eps|^p\right]
+\sup_{t\le t_+}\sum_{k=1}^N
\E_x\!\left[\left\|D^kY_t^{\eps,x}\right\|_{\mathfrak H^{\otimes k}}^p\right]
\right\}
\le C_{N,p},
\end{equation}
and the same estimate holds for $Y^{0,x}$.
Letting first $\eps\downarrow0$ in
the preceding localized estimates and then $R\uparrow\infty$, using H\"older's
inequality, the uniform moment bounds, and
\[
\sup_{0<\eps\le\eps_0}\sup_{x\in K_0}
\Pbb_x\left(\tau_R^\eps\le t_+\right)\le C_aR^{-a}
\qquad\text{for every fixed }a<\infty,
\]
we obtain, for every fixed $N,p$,
\[
\sup_{\substack{t\in[t_-,t_+]\\x\in K_0}}
\left\|Y_t^{\eps,x}-Y_t^{0,x}\right\|_{\mathbb D^{N,p}}
\le C_{R,N,p}\sqrt\eps
+C_{N,2p}
\sup_{x\in K_0}\Pbb_x(\tau_R^\eps\le t_+)^{1/(2p)}.
\]
Consequently,
\begin{equation}
\label{eq:well-malliavin-sobolev-convergence}
\sup_{\substack{t\in[t_-,t_+]\\x\in K_0}}
\left\|Y_t^{\eps,x}-Y_t^{0,x}\right\|_{\mathbb D^{N,p}}
\longrightarrow0
\qquad\text{as $\eps\downarrow0$}.
\end{equation}

The OU covariance is uniformly non-degenerate on $[t_-,t_+]$.  Combining
\eqref{eq:uniform-well-malliavin-inverse},
\eqref{eq:well-malliavin-sobolev-convergence}, and the density-comparison
estimate \cite[Appendix~A, equation~(A.2)]{Pigato2022}, with
$F=Y_t^{\eps,x}$ and $G=Y_t^{0,x}$, yields
\[
\sup_{\substack{t\in[t_-,t_+]\\x\in K_0\\y\in\R^{3d}}}
\left|\bar P_{i,\eps}(t,x,y)-P_i^{\rm OU}(t,x,y)\right|
\longrightarrow0
\qquad\text{as $\eps\downarrow0$}.
\]
This proves the stated compact-$y$ convergence for fixed $i$.  Since there
are only two wells, all constants may be chosen simultaneously.
\end{proof}

\begin{remark}
Undoing the well scaling contributes the Jacobian $\eps^{-3d/2}$.  The lemma
concerns the cutoff density.  The next lemma applies this convergence to an
integrated local event and transfers the estimate to the original process by
subtracting the exit probability.
\end{remark}

\begin{lemma}[Scaled local hypoelliptic density lower bound]
\label{lem:scaled-local-density-lower-bound}
There are $r_0,t_0>0$ such that for every $C_0<\infty$ there are
$c(C_0),c_0(C_0),\eps_0(C_0)>0$ such that, for $i\in\{m,s\}$,
$\eps<\eps_0(C_0)$, every $z\in B(x_i,C_0\sqrt\eps)$, and every Borel
$E\subset B(x_i,C_0\sqrt\eps)$,
\begin{equation}
\label{eq:scaled-local-minorization}
\Pbb_z\left(Z_{t_0}^\eps\in E,\ t_0<\tau_{\partial B(x_i,r_0)}\right)
\ge c(C_0)\eps^{-3d/2}\operatorname{Leb}(E)-e^{-c_0/\eps}.
\end{equation}
In particular, for $\varrho\in(1/2,1]$,
\[
\inf_{z\in B(x_i,\eps^\varrho)}
\Pbb_z\left(
Z_{t_0}^\eps\in\mathsf A_\eps^i,
\ t_0<\tau_{\partial B(x_i,r_0)}
\right)
\ge c\eps^{3d/2}.
\]
\end{lemma}

\begin{proof}
Fix $i\in\{m,s\}$.
The covariance at time $t_0$ of the OU process
\eqref{eq:well-ou-process} is
\[
\Gamma_{i,t_0}=\int_0^{t_0}e^{sA_i}GG^\top e^{sA_i^\top}\,ds,
\]
where $G=\left(0,0,\sqrt{2\gamma/L}\,I_d\right)^\top$.  The rank identity
\eqref{eq:well-ou-kalman-rank} is the Kalman controllability condition, and
therefore $\Gamma_{i,t_0}>0$ for every $t_0>0$.
Hence,
	$P_i^{\rm OU}(t_0,x,y)$ is strictly positive, and its infimum for
	$|x|,|y|\le C_0$ is positive.  The uniform kernel convergence in
	Lemma~\ref{lem:rescaled-well-kernel-convergence} implies that
	$\bar P_{i,\eps}(t_0,x,y)\ge c(C_0)$ for $|x|,|y|\le C_0$.  Write
	$\Pbb_z^{i,\eps,\mathrm{cut}}$ for the law, in the original variables, of
	the corresponding cutoff process started from $z$.  For
$z=x_i+\sqrt\eps x$ and
\[
E_\eps:=\left\{y\in\R^{3d}:x_i+\sqrt\eps y\in E\right\},
\]
we have $|x|\le C_0$, $E_\eps\subset B(0,C_0)$, and
\[
\Pbb_z^{i,\eps,\mathrm{cut}}(Z_{t_0}^\eps\in E)
=\int_{E_\eps}\bar P_{i,\eps}(t_0,x,y)\,dy
\ge c(C_0)\operatorname{Leb}(E_\eps)
=c(C_0)\eps^{-3d/2}\operatorname{Leb}(E).
\]
The cutoff and true processes coincide up to
$\tau_{\partial B(x_i,r_0)}$.  The deterministic flow from the starting set
stays in $B(x_i,r_0/2)$ on $[0,t_0]$, so
Lemma~\ref{lem:uniform-fw-tracking} implies that
\[
\sup_{z\in B(x_i,C_0\sqrt\eps)}
\Pbb_z\left(\tau_{\partial B(x_i,r_0)}\le t_0\right)
\le e^{-c_0/\eps}.
\]
Consequently,
\[
\Pbb_z\left(Z_{t_0}^\eps\in E,
t_0<\tau_{\partial B(x_i,r_0)}\right)
\ge
c(C_0)\eps^{-3d/2}\operatorname{Leb}(E)-e^{-c_0/\eps},
\]
which is \eqref{eq:scaled-local-minorization}.  If
$z\in B(x_i,\eps^\varrho)$ with $\varrho>1/2$, then
$|(z-x_i)/\sqrt\eps|\le\eps^{\varrho-1/2}\le1$ for small $\eps$.
Taking $E=\mathsf A_\eps^i$ and using
\[
\operatorname{Leb}(\mathsf A_\eps^i)=\operatorname{Leb}(B_1)\eps^{3d}
\]
in \eqref{eq:scaled-local-minorization}, and then absorbing
$e^{-c_0/\eps}=o(\eps^{3d/2})$, gives the final lower bound
$c\eps^{3d/2}$ uniformly on $B(x_i,\eps^\varrho)$.  Taking the minimum of the
two well constants makes the estimate simultaneous in $i$.
\end{proof}

\subsection{Pigato-Type Density Bounds and Lee--Ramil--Seo Committor Localization}
\label{app:pigato-lrs}

Recall from Subsection~\ref{sec:notation} the components $W_m,W_s$, the
momentum flip $\Theta_*$, and the killed domains $D_\eps^m,D_\eps^s$.
The stopped time-reversal and committor-localization mechanism culminating in
Proposition~\ref{prop:third-order-lrs-committor-localization} adapts
the underdamped argument of Lee--Ramil--Seo \cite{LeeRamilSeo2026} to the
third-order chain $r\to p\to\theta$, with related trajectory estimates developed
in \cite{LeeRamilSeo2023Cutoff}; the required localized density estimate is
supplied by the chain estimates of \cite{Pigato2022}.

We first record the momentum-reversal relation between the
forward and adjoint dynamics.  It transfers estimates for the forward
committor to the adjoint committor and preserves the two Hamiltonian
sublevel components.

\begin{proposition}[Hamiltonian invariance and momentum flip]
\label{prop:adjoint-lyapunov-dissipation}
$H\circ\Theta_*=H$, $\Theta_*W_i=W_i$, and the forward/adjoint committors satisfy
$h_\eps^*(z)=h_\eps(\Theta_*z)$.
\end{proposition}

\begin{proof}
$H$ is even in $(p,r)$, so $H\circ\Theta_*=H$ and $\Theta_*$ preserves each
component of $W$.  
By conjugating the forward generator by $\Theta_*$,  flipping the
$p$-transport and leaving the $r$-diffusion, we get
\[
\calL_\eps^* f(z)=\calL_\eps(f\circ\Theta_*)(\Theta_*z).
\]
Since $A_\eps,S_\eps$
are $\Theta_*$-invariant, uniqueness gives $h_\eps^*=h_\eps\circ\Theta_*$.
\end{proof}

The stopped time-reversal argument uses an auxiliary dynamics
whose $r$-friction has the opposite sign.  Its positive divergence produces
the exponential Jacobian factor in the subsequent Feynman--Kac estimate.

\begin{lemma}[Anti-dissipative generator and Jacobian exponent]
\label{lem:anti-dissipative-generator-jacobian}
The auxiliary process
\begin{align*}
&d\widetilde\theta=-\widetilde p\,dt,
\\
&d\widetilde p=(L^{-1}\nabla U(\widetilde\theta)-\gamma\widetilde r)\,dt,
\\
&d\widetilde r=(\gamma\widetilde p+\gamma\widetilde r)\,dt+\sqrt{2\gamma\eps/L}\,dB_{t},
\end{align*}
which we denote by
$\widetilde Z_t=(\widetilde\theta_t,\widetilde p_t,\widetilde r_t)$,
has a unique non-explosive strong solution.  Its drift term $\widetilde b$
given by
\[
\widetilde b(\theta,p,r)
=\left(-p,\ L^{-1}\nabla U(\theta)-\gamma r,\ \gamma p+\gamma r\right),
\]
and $\operatorname{div}\widetilde b=\gamma d$.
\end{lemma}

\begin{proof}
Local existence and pathwise uniqueness hold because the coefficients are
locally Lipschitz.  A direct computation gives
\[
\widetilde\calL_\eps H
=\gamma L|r|^2+\gamma d\eps
\le 2\gamma H+\gamma d\eps,
\]
where $U\ge0$.  It\^o's formula, stopped on Hamiltonian sublevels, followed
by Gr\"onwall's inequality and monotone convergence excludes explosion.
Only $\gamma r$ contributes to the divergence, giving $\gamma d$.  The
Lebesgue Jacobian of the
auxiliary flow over time $t$ is $e^{\gamma d\,t}$.
\end{proof}

Conjugating the adjoint generator by the Gibbs weight identifies
the auxiliary generator plus a constant potential.  Stopping the resulting
Feynman--Kac functional at the well boundary gives a pointwise committor
bound.

\begin{lemma}[Stopped $h$-transform supersolution estimate]
\label{lem:lrs-stopped-time-reversal}
For the auxiliary process $\widetilde Z$ of
Lemma~\ref{lem:anti-dissipative-generator-jacobian}, write
\[
\widetilde\tau_D:=\inf\left\{t\ge0:\widetilde Z_t\in D\right\},
\qquad
\widetilde\E_z[\cdot]:=\E\left[\,\cdot\mid\widetilde Z_0=z\right].
\]
With
\[
\widetilde\zeta_\eps^m=\widetilde\tau_{\partial W_m}\wedge\widetilde\tau_{A_\eps},
\qquad
\widetilde\zeta_\eps^s=\widetilde\tau_{\partial W_s}\wedge\widetilde\tau_{S_\eps},
\]
for $z\in W_m$,
\[
1-h_\eps(z)\le
e^{(H(z)-U(\sigma))/\eps}
\widetilde\E_{\Theta_*z}\left[e^{\gamma d\,\widetilde\zeta_\eps^m}\right],
\]
and for $z\in W_s$,
\[
h_\eps(z)\le
e^{(H(z)-U(\sigma))/\eps}
\widetilde\E_{\Theta_*z}\left[e^{\gamma d\,\widetilde\zeta_\eps^s}\right].
\]
\end{lemma}

\begin{proof}
For $g=e^{H/\eps}v$, direct differentiation implies that
\[
\nabla_r g=e^{H/\eps}\left(\nabla_rv+\frac{Lr}{\eps}v\right),
\qquad
\Delta_r g=e^{H/\eps}\left(
\Delta_rv+\frac{2Lr}{\eps}\cdot\nabla_rv
+\frac{Ld}{\eps}v+\frac{L^2|r|^2}{\eps^2}v\right).
\]
Substituting these identities, together with the analogous first-derivative
formulas in $\theta$ and $p$, into $\calL_\eps^*g$ yields
\begin{align*}
e^{-H/\eps}\calL_\eps^*(e^{H/\eps}v)
&=-p\cdot\nabla_\theta v
+(L^{-1}\nabla U-\gamma r)\cdot\nabla_pv \\
&\qquad\qquad +(\gamma p+\gamma r)\cdot\nabla_rv
+\frac{\gamma\eps}{L}\Delta_rv+\gamma d\,v.
\end{align*}
Indeed, the diffusion cross term $2\gamma r\cdot\nabla_rv$ changes the adjoint
$r$-drift $\gamma p-\gamma r$ into $\gamma p+\gamma r$, while the zero-order
terms satisfy
\[
-\frac{\gamma L}{\eps}|r|^2v
+\frac{\gamma\eps}{L}
\left(\frac{Ld}{\eps}+\frac{L^2|r|^2}{\eps^2}\right)v
=\gamma d\,v.
\]
Thus
\[
\calL_\eps^*\left(e^{H/\eps}v\right)
=e^{H/\eps}\left(\widetilde\calL_\eps v+\gamma d\,v\right).
\]
With
$\calL_\eps(f\circ\Theta_*)=(\calL_\eps^*f)\circ\Theta_*$, the truncated
Feynman--Kac functional
$V_{m,T}(t,y)=\widetilde\E_y e^{\lambda(\widetilde\zeta_\eps^m\wedge(T-t))}$,
$\lambda=\gamma d$, solves
\[
\partial_tV_{m,T}+\widetilde\calL_\eps V_{m,T}+\lambda V_{m,T}=0,
\]
so
$G_{m,T}(t,z)=e^{(H-U(\sigma))/\eps}V_{m,T}(t,\Theta_*z)$
satisfies
\[
(\partial_t+\calL_\eps)G_{m,T}=0,
\]
i.e.\ it is space--time harmonic for $\partial_t+\calL_\eps$.

Optional stopping at $\tau_{\partial W_m}\wedge\tau_{A_\eps}\wedge(T-t)$ and
$G_{m,T}=1$ on the $\{\tau_{\partial W_m}<\tau_{A_\eps}\}$ boundary give
$\Pbb_z(\tau_{\partial W_m}<\tau_{A_\eps})\le G_{m,T}$; let $T\to\infty$.
Since
\[
1-h_\eps(z)=\Pbb_z(\tau_{S_\eps}<\tau_{A_\eps})\le\Pbb_z(\tau_{\partial W_m}<\tau_{A_\eps}),
\]
the $W_m$ bound follows; $W_s$ is identical.  The auxiliary moment is finite by
Proposition~\ref{prop:third-order-lrs-exponential-stopping}.
\end{proof}

To estimate the auxiliary stopping-time moment, we relate its
killed transition density to that of the original dissipative process.  The
factor $e^{\gamma dt}$ is exactly the Jacobian exponent computed above.

\begin{lemma}[Third-order absorbed density duality]
\label{lem:third-order-absorbed-density-duality}
For the open sets
$D=W_m\setminus\overline{A_\eps}$ or
$W_s\setminus\overline{S_\eps}$, the killed densities of
the original and auxiliary processes satisfy
$p_t^{D,\eps}(z,y)=e^{\gamma d t}\widetilde p_t^{D,\eps}(y,z)$.
\end{lemma}

\begin{proof}
The original drift $b$ has $\operatorname{div}b=-\gamma d$ and
$\widetilde b=-b$, so the Lebesgue adjoint is
$\calL_\eps^\dagger=\widetilde\calL_\eps+\gamma d$ on
$C_c^\infty(D)$.  To pass from this interior identity to killed semigroups,
choose smooth open sets $D_n$ with compact closures in $D$ such that
$\overline{D_n}\subset D_{n+1}$ and $\bigcup_nD_n=D$.

We localize the coefficients before using whole-space semigroups, in the same
spirit as the regularized absorbed-kernel argument of
\cite[Section~8]{LeeRamilSeo2026}.  Extend $\nabla U$ from a neighborhood of
the configuration projection of $\overline{D_n}$ to a smooth map $F_n$ with
at most linear growth and bounded derivatives of every positive order.  Set
\[
b_n(\theta,p,r)
=\left(p,-L^{-1}F_n(\theta)+\gamma r,-\gamma p-\gamma r\right),
\qquad
\widetilde b_n=-b_n.
\]
These globally Lipschitz chain-compatible drifts agree with $b,\widetilde b$
on a neighborhood of $\overline{D_n}$ and satisfy
\[
\operatorname{div}b_n=-\gamma d,
\qquad
\operatorname{div}\widetilde b_n=\gamma d.
\]
Consequently, their diffusions are non-explosive and their generators obey
$\calL_{\eps,n}^\dagger=\widetilde\calL_{\eps,n}+\gamma d$ on the whole
space.

The minimal semigroup on $D_n$ is obtained from the whole-space cutoff
process by smooth killing potentials $k_{n,j}$ that vanish on increasing
interior cores and diverge on the complementary boundary layers.  Use the
same scalar killing potential for the original and auxiliary cutoff
generators, and denote the corresponding penalized semigroups by
$P_t^{(n,j)}$ and $\widetilde P_t^{(n,j)}$.  For every non-negative
$f,g\in C_c^\infty(D_n)$, the whole-space integration-by-parts identity is
\begin{equation}
\label{eq:penalized-whole-space-duality}
\int_{\R^{3d}} f P_t^{(n,j)}g\,dz
=e^{\gamma dt}
\int_{\R^{3d}} g\widetilde P_t^{(n,j)}f\,dz.
\end{equation}
Because $f$ and $g$ are supported in $D_n$, the two integrals in
\eqref{eq:penalized-whole-space-duality} may equivalently be restricted to
$D_n$.  Letting $j\to\infty$ and using the standard penalization convergence
to the minimal killed semigroups gives
\[
\int_{D_n} f P_t^{D_n}g\,dz
=e^{\gamma dt}\int_{D_n}g\widetilde P_t^{D_n}f\,dz.
\]
Here the cutoff processes can be replaced by the original and auxiliary
processes because their coefficients agree up to the exit time from $D_n$.
Finally, letting $D_n\uparrow D$ yields
\[
\int_D f P_t^Dg\,dz
=e^{\gamma dt}\int_D g\widetilde P_t^Df\,dz.
\]
This exhaustion argument applies to an arbitrary open set $D$.  Both killed semigroups have
interior smooth densities by H\"ormander hypoellipticity; the preceding identity
proves the asserted kernel equality almost everywhere, and interior continuity
gives the pointwise version on $D\times D$.
\end{proof}

The density duality reduces the auxiliary stopping estimate to
an exit-tail estimate for the original process.  Near the saddle, an exact
integrated coordinate reduces the required bottleneck estimate to a
one-dimensional Brownian survival bound.

\begin{lemma}[One-sided force on the saddle sublevel lobes]
\label{lem:saddle-lobe-force-sign}
Orient $e_\sigma$ so that the negative unstable lobe belongs to $W_m$.
There is a neighborhood $\mathcal N_\sigma$ of $z_\sigma$ such that, writing
$q_1=\langle\theta-\sigma,e_\sigma\rangle$, every
$z=(\theta,p,r)\in(W_m\cup W_s)\cap\mathcal N_\sigma$ satisfies
\[
q_1<0\ \text{ on }W_m,
\qquad
q_1>0\ \text{ on }W_s,
\]
and
\begin{equation}
\label{eq:saddle-lobe-force-sign}
q_1\,\langle\nabla U(\theta),e_\sigma\rangle<0.
\end{equation}
Moreover, in the Hessian eigenbasis,
\begin{equation}
\label{eq:saddle-lobe-cone-bound}
|q'|+|p|+|r|\le C|q_1|
\qquad\text{on }(W_m\cup W_s)\cap\mathcal N_\sigma.
\end{equation}
\end{lemma}

\begin{proof}
Let $q'=(q_2,\ldots,q_d)$ denote the stable configuration coordinates.  The
Taylor expansion at $z_\sigma$ implies that
\[
H-U(\sigma)
=-\frac{\lambda_\sigma}{2}q_1^2
{}+\frac12\sum_{j=2}^d\lambda_j^\sigma q_j^2
{}+\frac L2\left(|p|^2+|r|^2\right)
{}+\mathcal O(|q|^3).
\]
On $H<U(\sigma)$, after shrinking the neighborhood so that the cubic
remainder can be absorbed into the quadratic terms, this implies
\[
|q'|^2+|p|^2+|r|^2\le Cq_1^2.
\]
In particular $q_1\ne0$, since $z_\sigma$ itself belongs to the boundary
$\{H=U(\sigma)\}$.  The two signs of $q_1$ give the two local components of
the sublevel; the stated assignment follows from the orientation of
$e_\sigma$.  Finally,
\[
\langle\nabla U(\theta),e_\sigma\rangle
=-\lambda_\sigma q_1+\mathcal O(|q|^2)
=-\lambda_\sigma q_1+\mathcal O(q_1^2),
\]
where the last equality uses \eqref{eq:saddle-lobe-cone-bound}.  A further
reduction of $\mathcal N_\sigma$ proves
\eqref{eq:saddle-lobe-force-sign}.
\end{proof}

Near the saddle the deterministic drift is not uniformly separated from
zero, so the compact-sublevel entrance argument in
Lemma~\ref{lem:original-compact-entrance-well-ball} is insufficient.
The next lemma supplies a uniform polynomial-time escape estimate even for
starting points arbitrarily close to the saddle bottleneck.

\begin{lemma}[Polynomial escape from the saddle bottleneck]
\label{lem:original-saddle-bottleneck-escape}
For the neighborhood $\mathcal N_\sigma$ of
Lemma~\ref{lem:saddle-lobe-force-sign}, there is $C<\infty$
such that, for $i\in\{m,s\}$,
\[
 \sup_{z\in W_i\cap\mathcal N_\sigma}
 \Pbb_z\!\left(
 \tau_{\partial W_i}\wedge\tau_{\mathcal N_\sigma^c}>t
 \right)
 \le \frac{C}{\sqrt{\eps t}},
 \qquad t>0.
\]
In particular, for every fixed $M>1$, the probability of reaching
$\partial W_i$ or leaving $\mathcal N_\sigma$ before time $\eps^{-M}$ tends to
one uniformly over $W_i\cap\mathcal N_\sigma$.
\end{lemma}

\begin{proof}
Introduce the exact integrated coordinate
\[
 \Xi(\theta,p,r)
 :=\theta-\sigma+\frac{1}{\gamma}p+\frac1\gamma r .
\]
The exact equations of motion give
\begin{equation}
\label{eq:integrated-saddle-coordinate}
 d\Xi_t
 =-\frac{1}{\gamma L}\nabla U(\theta_t)\,dt
   +\sqrt{\frac{2\eps}{\gamma L}}\,dB_t .
\end{equation}
Set $X_t=\langle\Xi_t,e_\sigma\rangle$ and
$\beta_t=\langle B_t,e_\sigma\rangle$.  The latter is a standard
one-dimensional Brownian motion.  Up to
\[
 \zeta_i:=\tau_{\partial W_i}\wedge\tau_{\mathcal N_\sigma^c},
\]
the drift of $X$ in \eqref{eq:integrated-saddle-coordinate} is non-positive
for $i=m$ and non-negative for $i=s$ by
Lemma~\ref{lem:saddle-lobe-force-sign}.  Moreover, $|X_t|\le C_0$ on
$\{t<\zeta_i\}$, with $C_0$ independent of $\eps$ and the starting point.
Thus, on $\{\zeta_m>t\}$,
\[
 \inf_{0\le u\le t}\beta_u
 \ge-\frac{C_1}{\sqrt\eps},
\]
while on $\{\zeta_s>t\}$,
\[
 \sup_{0\le u\le t}\beta_u
 \le\frac{C_1}{\sqrt\eps}.
\]
The reflection principle and the Brownian scaling property yield that
\[
 \Pbb\!\left(\inf_{u\le t}\beta_u\ge-a\right)
 =\Pbb\!\left(\sup_{u\le t}\beta_u\le a\right)
 \le C\min\!\left\{1,\frac{a}{\sqrt t}\right\}.
\]
Taking $a=C_1/\sqrt\eps$ proves the estimate.  Substitution of
$t=\eps^{-M}$ proves the last assertion.
\end{proof}

Away from the saddle and the energy boundary, the process starts in a compact
subset of one sublevel component.  Dissipation of the Hamiltonian then drives
it toward the corresponding well, and finite-time tracking makes this
entrance probabilistically uniform.

\begin{lemma}[Deterministic entrance from compact sublevels]
\label{lem:original-compact-entrance-well-ball}
For each $i\in\{m,s\}$, let $B_i$ be a sufficiently small ball centered at
$x_i$ such that $\overline{B_i}\subset W_i$.  For every compact
$Q^i\subset W_i$, there exist
$T_Q<\infty$, $c>0$, and $\eps_0>0$ such that
\[
\inf_{0<\eps\le\eps_0}\inf_{z\in Q^i}
\Pbb_z\!\left(\tau_{B_i}\le T_Q<\tau_{\partial W_i}\right)\ge c.
\]
\end{lemma}

\begin{proof}
Let $\Phi_t$ denote the original zero-noise flow and set
$h_Q:=\max_{z\in Q^i}H(z)<U(\sigma)$.  Since
\[
\frac{d}{dt}H(\Phi_tz)=-\gamma L|r_t|^2\le0,
\]
every trajectory starting in $Q^i$ remains in the compact set
\[
K_Q^i:=\overline W_i\cap\{H\le h_Q\}\subset W_i.
\]
On the set where $dH/dt=0$ one has $r=0$.  A trajectory that remains in
$\{r=0\}$ must also satisfy, successively,
\[
0=\dot r=-\gamma p,
\qquad
0=\dot p=-L^{-1}\nabla U(\theta),
\]
and is therefore an equilibrium $(j,0,0)$ with $j\in\operatorname{Crit}(U)$.
By Assumption~\ref{ass:double-well}, the only such point in
$W_i$ is $x_i$.  LaSalle's invariance principle consequently gives
\[
\operatorname{dist}(\Phi_tz,x_i)\longrightarrow0
,\qquad z\in Q^i,
\]
as $t\rightarrow\infty$.
Following the energy-sublevel entrance argument of
\cite[Lemmas~10.3--10.4]{LeeRamilSeo2026}, choose a smaller concentric ball
$B_i'$ with $\overline{B_i'}\subset B_i$.  The isolation of $x_i$ as the
unique minimum of $H$ in $W_i$ gives $\kappa_i>0$ such that
\[
\overline W_i\cap\{H\le U(i)+2\kappa_i\}\subset B_i'.
\]
For $z\in Q^i$, define the deterministic entrance time
\[
S_i(z):=\inf\left\{t\ge0:
H(\Phi_tz)\le U(i)+\kappa_i\right\}.
\]
The convergence above implies $S_i(z)<\infty$ for every $z\in Q^i$.
Moreover, after the first entrance the energy is strictly below the threshold
at every positive later time.  Indeed, if it remained equal to
$U(i)+\kappa_i$ on a non-trivial interval, then
$r=0$, followed successively by $p=0$ and $\nabla U(\theta)=0$, throughout
that interval.  The trajectory would be the equilibrium $x_i$, whose energy
is $U(i)$, a contradiction.

We next check the dependence of $S_i$ on the initial point, as in the cited
underdamped argument.  If $S_i(z)>0$ and $0<a<S_i(z)$, compactness of the
time interval and continuity of the flow give, for all $z'$ close to $z$,
\[
H(\Phi_tz')>U(i)+\kappa_i,
\qquad 0\le t\le S_i(z)-a.
\]
The strict inequality after entrance also gives
$H(\Phi_{S_i(z)+a}z)<U(i)+\kappa_i$, and hence the same inequality for
$z'$ close to $z$.  Thus
$|S_i(z')-S_i(z)|<a$.  The same upper bound, with
$0\le S_i(z')<a$, applies when $S_i(z)=0$.  Hence $S_i$ is continuous on
$Q^i$.  Compactness now yields
\[
T_*:=\sup_{z\in Q^i}S_i(z)<\infty.
\]
Set $T_Q=T_*+1$.  Energy monotonicity and the strict post-entrance inequality
give
\[
\Phi_{T_Q}(Q^i)\subset
\{H\le U(i)+\kappa_i\}\subset B_i'.
\]
Furthermore,
\[
C_Q^i:=\Phi_{[0,T_Q]}(Q^i)
\subset K_Q^i\subset W_i
\]
is compact.
Choose $\rho>0$ smaller than both
$\operatorname{dist}(B_i',B_i^c)$ and
$\operatorname{dist}(C_Q^i,\partial W_i)$.  By
Lemma~\ref{lem:uniform-fw-tracking},
\[
\inf_{z\in Q^i}
\Pbb_z\left(
\sup_{0\le t\le T_Q}|Z_t^\eps-\Phi_tz|\le\rho
\right)
\ge1-e^{-c_1/\eps}.
\]
On this event the process does not hit $\partial W_i$ before $T_Q$ and lies
in $B_i$ at time $T_Q$.  For all sufficiently small $\eps$ the last
probability is at least, for example, $c=1/2$, proving the claim.
\end{proof}

The remaining part of the energy boundary is compact and separated from the
saddle neighborhood.  On this non-saddle part, strict dissipation after a
uniform time gives either an exit or entrance into a lower-energy compact
set.

\begin{lemma}[Entrance from non-saddle boundary collars]
\label{lem:original-nonsaddle-boundary-collar}
After choosing $\mathcal N_\sigma$ small, for each $i$ there are an interior
collar
\[
\mathcal C_i
:=\left\{z\in W_i:\dist\left(z,K_i^\partial\right)<\delta_i\right\},
\qquad
K_i^\partial:=\partial W_i\setminus\mathcal N_\sigma,
\]
for some $\delta_i>0$, a compact
$Q_\partial^i\subset W_i$, and $T_\partial,c_\partial>0$ such that for
$z\in D_\eps^i\cap\mathcal C_i$,
\[
\Pbb_z\left(\tau_{\partial W_i}\le T_\partial\ \text{or}\ \tau_{Q_\partial^i}\le T_\partial\wedge\tau_{\partial W_i}\right)\ge c_\partial.
\]
\end{lemma}

\begin{proof}
$K_i^\partial=\partial W_i\setminus\mathcal N_\sigma$ is compact and contains no
equilibrium (an invariant point on $H=U(\sigma)$ with $dH/dt=0$ has $r=0$, and hence
$p=0$, $\nabla U(\theta)=0$, i.e.\ the saddle, which is in $\mathcal N_\sigma$).
Since 
\[
\frac{dH}{dt}=-\gamma L|r|^2\le0, 
\]
LaSalle's invariance principle plus compactness imply that there exist some $T_\partial$,
$\delta_\partial>0$ such that 
\[
H(\Phi_{T_\partial}z)\le U(\sigma)-4\delta_\partial,
\qquad\text{on
$K_i^\partial$},
\]
Next, we set $Q_\partial^i=\overline W_i\cap\{H\le U(\sigma)-2\delta_\partial\}$.
By continuity of $(t,z)\mapsto\Phi_t(z)$, after decreasing the
collar width $\delta_i$ we have, for every $z\in\mathcal C_i$, the deterministic
alternative
\[
\tau_{\partial W_i}^{0}(z)\le T_\partial
\quad\text{or}\quad
\Phi_{T_\partial}(z)\in
\left\{H\le U(\sigma)-3\delta_\partial\right\}
\subset\operatorname{int}Q_\partial^i,
\]
where $\tau_{\partial W_i}^{0}$ is the deterministic exit time.  Choose
$\eta>0$ smaller than the distance from
$\{H\le U(\sigma)-3\delta_\partial\}$ to
$W_i\setminus Q_\partial^i$.  Lemma~\ref{lem:uniform-fw-tracking}, applied on
the compact closure of the collar, then shows that with probability at least
$1-e^{-c/\eps}$ the diffusion either crosses $\partial W_i$ before
$T_\partial$ or belongs to $Q_\partial^i$ at time $T_\partial$ without an
earlier exit.  For small $\eps$ this is bounded below by a constant
$c_\partial>0$, completing the proof.
\end{proof}

The preceding entrance lemmas reach a fixed well
neighborhood.  The next estimate combines Lyapunov contraction with the local
density lower bound to hit the shrinking ball on a logarithmic time scale.

	\begin{lemma}[Logarithmic hit of the shrinking well ball]
	\label{lem:original-log-hit-shrinking-ball}
	For each $i\in\{m,s\}$ and a sufficiently small fixed well ball $B_i$,
	\[
	\inf_{z\in\overline{B_i}}
	\Pbb_z\left(\tau_{\mathsf A_\eps^i}\le C\log(1/\eps)\right)
	\ge c\eps^{3d/2}.
\]
The constants can be chosen simultaneously for the two wells.
\end{lemma}

	\begin{proof}
	Fix $i\in\{m,s\}$.  Let $V_i$ and $r_0$ be given by
	Lemma~\ref{lem:local-hypocoercive-lyapunov}.  Choose
	$0<a_-<a_+$ so that
	\[
	\overline{B_i}\subset\operatorname{int}K_i^-\subset K_i^-
	\subset\operatorname{int}K_i^+\subset K_i^+\subset B(x_i,r_0),
	\qquad
	K_i^\pm:=\{z:V_i(z-x_i)\le a_\pm\}.
	\]
	After decreasing $B_i$ if necessary, these inclusions hold simultaneously for
	the two wells.  Write the drift as $b(x_i+Y)=A_iY+R_i(Y)$ near $x_i$.
	After decreasing $r_0$, choose a smooth cutoff $\chi_i$ that equals one on
	$B(0,r_0)$ and is supported in a slightly larger ball on which
	$|R_i(Y)|\le C|Y|^2$, and define
	$\bar b_i(x_i+Y):=A_iY+\chi_i(Y)R_i(Y)$.  The support radius can be chosen
	so small that the Lyapunov-equation calculation in
	Lemma~\ref{lem:local-hypocoercive-lyapunov} gives, globally,
	$\bar{\calL}_{i,\eps}V_i\le-cV_i+C\eps$.  Let $\bar Z$ be the resulting
	cutoff process, denote its law and expectation by
	$\overline{\Pbb}_z^{\,i,\eps}$ and $\overline{\E}_z^{\,i,\eps}$, and couple
	$\bar Z$ and $Z^\eps$ by driving them with the same Brownian motion.  Denote
	the resulting joint law by $\mathbb Q_z^{i,\eps}$.
	Dynkin's formula implies that
	\[
	\overline{\E}_z^{\,i,\eps} V_i(\bar Z_t-x_i)
	\le e^{-ct}V_i(z-x_i)+C\eps .
	\]

	We next make the logarithmic-time localization explicit.  For the zero-noise
	flow $\Phi_t$, the same Lyapunov inequality with $\eps=0$ gives
	\[
	V_i(\Phi_tz-x_i)\le e^{-ct}V_i(z-x_i),
	\qquad z\in K_i^-,\quad 0\le t\le1.
	\]
	Consequently, the paths $\Phi_{[0,1]}(K_i^-)$ lie in $K_i^-$ and
	$\Phi_1(K_i^-)$ is a positive distance inside $K_i^-$.  Choose $\rho>0$ so
	small that every $\rho$-tube around such a path stays in $K_i^+$ and its
	time-one endpoint lies in $K_i^-$.  The finite-horizon tracking estimate of
	Lemma~\ref{lem:uniform-fw-tracking}, whose proof applies verbatim to the
	globally Lipschitz cutoff drift, gives $a>0$ such that, uniformly over
	$z\in K_i^-$,
	\[
	\overline{\Pbb}_z^{\,i,\eps}\left(
	\sup_{0\le t\le1}|\bar Z_t-\Phi_tz|>\rho\right)
	\le e^{-a/\eps}.
	\]
	Iterating at the integer times by the strong Markov property under this
	coupling shows that
	\begin{equation}
	\label{eq:logarithmic-well-coupling}
	\sup_{z\in\overline{B_i}}
	\mathbb Q_z^{i,\eps}\left(\bar Z_t\ne Z_t^\eps
	\text{ for some }t\le n\right)
	\le ne^{-a/\eps},
	\qquad n\in\mathbb N.
	\end{equation}
	Indeed, on every successful block the path remains in
	$K_i^+\subset B(x_i,r_0)$, where the two drifts agree, and its endpoint
	returns to $K_i^-$.  This closes the induction.

	Take
	$t_1=\lceil C_1\log(1/\eps)\rceil$, with $C_1$ large enough that
	$e^{-ct_1}\sup_{K_i^-}V_i\le\eps$.  Since $V_i\asymp|z-x_i|^2$, first
	choosing $C_0$ large and then $\eps$ small gives
	\[
	\inf_{z\in\overline{B_i}}
	\overline{\Pbb}_z^{\,i,\eps}\left(
	\bar Z_{t_1}\in B(x_i,C_0\sqrt\eps)\right)\ge\frac34.
	\]
	Combining this with \eqref{eq:logarithmic-well-coupling} and
	$t_1e^{-a/\eps}=o(1)$ yields
	\[
	\inf_{z\in\overline{B_i}}
	\Pbb_z\left(Z_{t_1}^\eps\in B(x_i,C_0\sqrt\eps)\right)\ge\frac12
	\]
	for every sufficiently small $\eps$.  For such $\eps$,
	$\mathsf A_\eps^i=B(x_i,\eps)\subset B(x_i,C_0\sqrt\eps)$.  Then
	Lemma~\ref{lem:scaled-local-density-lower-bound} implies that
\[
\inf_{y\in B(x_i,C_0\sqrt\eps)}
\Pbb_y\!\left(
Z_{t_0}^\eps\in\mathsf A_\eps^i,
t_0<\tau_{\partial B(x_i,r_0)}
\right)
\ge c\eps^{3d/2}.
	\]
	Consequently, the strong Markov property at $t_1$ yields
	\[
	\inf_{z\in\overline{B_i}}
	\Pbb_z\!\left(\tau_{\mathsf A_\eps^i}\le t_1+t_0\right)
	\ge \frac{c}{2}\eps^{3d/2}.
	\]
Since $t_0$ is fixed, increasing $C$ if necessary gives
$t_1+t_0\le C\log(1/\eps)$ for all sufficiently small $\eps$.  Taking the
minimum of the two constants proves the simultaneous assertion.
\end{proof}

The saddle, collar, compact-interior, and local-well estimates
cover every starting point of the killed domain.  Concatenating them yields a
uniform one-window absorption probability, from which an exponential tail
follows by the strong Markov property.

\begin{proposition}[Original killed exit tail]
\label{prop:original-killed-exit-tail}
There are $C,\alpha,\beta_0>0$ such that
for $i\in\{m,s\}$,
\[
\sup_{z\in D_\eps^i}
\Pbb_z(\tau_{\partial D_\eps^i}>t)
\le Ce^{-\alpha t\eps^{\beta_0}}.
\]
\end{proposition}

\begin{proof}
We divide the proof into two steps.

\emph{Step 1: one-window absorption.}
Fix a small well ball $B_i$ around $x_i$ as in
Lemmas~\ref{lem:original-compact-entrance-well-ball} and
\ref{lem:original-log-hit-shrinking-ball}.  Fix $M>1$, set
$T_\eps=\eps^{-M}$, and let
\[
K_i^{\rm int}
:=\overline W_i\setminus
\left(\mathcal N_\sigma\cup\mathcal C_i\cup B_i\right).
\]
After slightly enlarging the three open sets,
$K_i^{\rm int}\subset W_i$ is compact and
\[
\overline W_i
\subset \mathcal N_\sigma\cup\mathcal C_i\cup B_i\cup K_i^{\rm int}.
\]
Let $\rho_i:=\tau_{\partial W_i}\wedge\tau_{B_i}$.
Lemma~\ref{lem:original-saddle-bottleneck-escape} gives, uniformly on
$W_i\cap\mathcal N_\sigma$, probability at least $1/2$ of reaching
$\partial W_i$ or $\mathcal N_\sigma^c$ before $T_\eps/3$, for any sufficiently small
$\eps$.  If the latter event occurs without an exit, the entrance point lies
in $\mathcal C_i\cup B_i\cup K_i^{\rm int}$.  On $\mathcal C_i$,
Lemma~\ref{lem:original-nonsaddle-boundary-collar} reaches either
$\partial W_i$ or a compact subset of $W_i$ with uniformly positive
probability.  On that compact subset and on $K_i^{\rm int}$,
Lemma~\ref{lem:original-compact-entrance-well-ball} reaches $B_i$ in a fixed
time with uniformly positive probability.  The strong Markov property,
together with the same estimates for starting points outside
$\mathcal N_\sigma$, therefore yields $c_0>0$ such that
\begin{equation}
\label{eq:killed-gate-or-well}
\inf_{z\in D_\eps^i}
\Pbb_z\!\left(\rho_i\le T_\eps\right)\ge c_0.
\end{equation}

If $\rho_i=\tau_{\partial W_i}$, then
$\tau_{\partial D_\eps^i}=\rho_i$.  On
$\{\rho_i=\tau_{B_i}<\infty\}$, continuity of the trajectories gives
$Z_{\tau_{B_i}}^\eps\in\overline{B_i}$.  Hence, if
$\rho_i=\tau_{B_i}$, the strong Markov property and
Lemma~\ref{lem:original-log-hit-shrinking-ball} give conditional probability
at least $c_1\eps^{3d/2}$ of hitting $B(x_i,\eps)$ in an additional
$C_1\log(1/\eps)$ units of time.  Since this additional time is at
most $T_\eps$ for small $\eps$, \eqref{eq:killed-gate-or-well} yields
$c_2>0$ such that
\begin{equation}
\label{eq:killed-one-window-success}
\inf_{z\in D_\eps^i}
\Pbb_z\!\left(
\tau_{\partial D_\eps^i}\le 2T_\eps
\right)
\ge c_2\eps^{3d/2}.
\end{equation}

\emph{Step 2: iteration.}
Put $\widehat T_\eps=2T_\eps$.  Applying the strong Markov property at
$\widehat T_\eps,2\widehat T_\eps,\ldots$ and using
\eqref{eq:killed-one-window-success} yields
\[
\sup_{z\in D_\eps^i}
\Pbb_z\left(\tau_{\partial D_\eps^i}>n\widehat T_\eps\right)
\le\left(1-c_2\eps^{3d/2}\right)^n
\le e^{-c_2n\eps^{3d/2}}.
\]
For $t\ge\widehat T_\eps$, take
$n=\lfloor t/\widehat T_\eps\rfloor$ to obtain
\[
\sup_{z\in D_\eps^i}\Pbb_z\left(\tau_{\partial D_\eps^i}>t\right)
\le C\exp\left(-\frac12 c_2t\eps^{M+3d/2}\right).
\]
After increasing $C$ this also holds for $0\le t<\widehat T_\eps$.
The assertion follows with $\beta_0=M+3d/2$.
\end{proof}

To transfer this exit tail through the killed-density duality,
we also need a short-time upper bound on the original killed kernel.  The
anisotropic powers $t^{1/2},t^{3/2},t^{5/2}$ reflect the three levels of the
noise-propagation chain.

\begin{lemma}[Third-order small-time density upper bound]
\label{lem:third-order-small-time-density-upper}
For $D=W_m\setminus\overline{A_\eps}$ or
$W_s\setminus\overline{S_\eps}$ and $0<t\le1$,
\[
\sup_{z,y\in D}p_t^{D,\eps}(z,y)\le C\eps^{-3d/2}t^{-9d/2};
\]
in particular for
$\eps\le t\le1$,
\[
\sup_{z,y\in D}p_t^{D,\eps}(z,y)\le C\eps^{-6d}.
\]
\end{lemma}

\begin{proof}
The compact set $\overline W_m\cup\overline W_s$ is contained in a bounded
region.  Extend only $\nabla U$ outside its configuration projection to a
smooth map with bounded derivatives of every positive order and at most
linear growth.  Keeping all linear couplings unchanged gives a globally
Lipschitz, chain-compatible drift $\bar b$ which agrees with the original
drift on a neighborhood of $\overline W_m\cup\overline W_s$.  Let
$\bar Z_t^{\eps,z}$ be the corresponding unrestricted process and let
$\phi_t(z)$ solve 
\[
\dot\phi_t=\bar b(\phi_t), 
\qquad\phi_0=z.
\]
Center at the deterministic trajectory and normalize the noise by setting
\[
 Y_t^{\eps,z}=\frac{\bar Z_t^{\eps,z}-\phi_t(z)}{\sqrt\eps}.
\]
Then $Y_0^{\eps,z}=0$ and
\[
 dY_t^{\eps,z}=B_{\eps,z}(t,Y_t^{\eps,z})\,dt+G\,dB_t,
 \qquad
 B_{\eps,z}(t,y)
 =\frac{\bar b(\phi_t(z)+\sqrt\eps y)-\bar b(\phi_t(z))}{\sqrt\eps},
\]
where $G=(0,0,\sqrt{2\gamma/L}\,I_{d})^\top$ in the order
$(\theta,p,r)$.  Reorder the blocks as $X^1=r,X^2=p,X^3=\theta$.
Translations and the common factor $\sqrt\eps$ preserve the triangular
dependence, and the two chain Jacobians are exactly
\[
 J_{x^1}(B_{\eps,z})_2=\gamma I_{d},
 \qquad
 J_{x^2}(B_{\eps,z})_3=I_{d}.
\]
Moreover $B_{\eps,z}(t,0)=0$ for every $t\in[0,1]$.  Taylor's formula and the bounded derivatives of
$\bar b$ give, uniformly in $z\in\overline W_m\cup\overline W_s$,
$0<\eps\le\eps_0$, and $0\le t\le1$,
\[
 \|D_yB_{\eps,z}(t,\cdot)\|_\infty\le C,
 \qquad
 \|D_y^kB_{\eps,z}(t,\cdot)\|_\infty
 \le C_k\eps^{(k-1)/2},\quad k\ge2,
\]
with the analogous time-derivative bounds required in Pigato's hypotheses.
We verify the uniform weak H\"ormander constant before applying
Lemma~\ref{lem:finite-order-pigato}.
In the block order $X^1=r$, $X^2=p$, $X^3=\theta$, its hypothesis (H1)
contains the matrix
\[
J_{x^2}(B_{\eps,z})_3J_{x^1}(B_{\eps,z})_2G_r
=I_d(\gamma I_d)\sqrt{2\gamma/L}\,I_d,
\qquad
G_r=\sqrt{2\gamma/L}\,I_d.
\]
Its smallest singular value is the fixed number
$\lambda_0=\gamma\sqrt{2\gamma/L}>0$.  The initial point is zero,
$B_{\eps,z}(0,0)=0$, the diffusion coefficient is constant and bounded, and
the finite collection of coefficient seminorms in
Lemma~\ref{lem:finite-order-pigato} is uniform by the preceding bounds.  Since
$B_{\eps,z}(t,0)=0$, the deterministic transport path is identically zero.
If $q_{\eps,z}(t,0,\cdot)$ denotes the density, then
\eqref{eq:finite-order-pigato-density} gives, for one fixed $p>2$,
\[
q_{\eps,z}(t,0,y)
\le
\frac{Ct^{-9d/2}}
{1+|\operatorname{diag}
(t^{-1/2}I_d,t^{-3/2}I_d,t^{-5/2}I_d)y|^p},
\qquad 0<t\le1,
\]
uniformly in $y,z,$ and $\eps$.  Taking the supremum in $y$ yields the
asserted $Ct^{-9d/2}$ bound.

The change of variables back to physical coordinates is exact:
\[
 \bar p_t^\eps(z,y)
 =\eps^{-3d/2}q_{\eps,z}
 \left(t,0,\frac{y-\phi_t(z)}{\sqrt\eps}\right).
\]
It follows that $\bar p_t^\eps(z,y)\le
C\eps^{-3d/2}t^{-9d/2}$.  Up to the exit time from $D$ the cutoff and original
processes coincide.  Hence, for every Borel set $E\subset D$,
\[
\int_Ep_t^{D,\eps}(z,y)\,dy
\le\int_E\bar p_t^\eps(z,y)\,dy.
\]
Thus $p_t^{D,\eps}\le\bar p_t^\eps$ almost everywhere; interior continuity of
the two hypoelliptic densities gives the pointwise inequality on $D\times D$.
Finally $t\ge\eps$ implies
$\eps^{-3d/2}t^{-9d/2}\le\eps^{-6d}$.
\end{proof}

We can now combine the killed-density duality, the original
exit tail, and the small-time kernel bound.  Their combination controls the
exponential moment of the auxiliary stopping time by a polynomial in
$\eps^{-1}$.

\begin{proposition}[Third-order Lee--Ramil--Seo exponential stopping estimate]
\label{prop:third-order-lrs-exponential-stopping}
There are $C,M_{\rm com}>0$ such that, for all sufficiently small $\eps$,
\[
\sup_{z\in W_m}\widetilde\E_z\left[e^{\gamma d\,\widetilde\zeta_\eps^m}\right]
+\sup_{z\in W_s}\widetilde\E_z \left[e^{\gamma d\,\widetilde\zeta_\eps^s}\right]
\le C\eps^{-M_{\rm com}}.
\]
\end{proposition}

\begin{proof}
Work in $D=W_m\setminus\overline{A_\eps}$, $\lambda=\gamma d$.  For $t\ge\eps$, by
Lemma~\ref{lem:third-order-absorbed-density-duality}
$e^{\lambda t}\widetilde p_t^{D,\eps}(z,y)=p_t^{D,\eps}(y,z)$, and
Chapman--Kolmogorov equation plus Lemma~\ref{lem:third-order-small-time-density-upper}
over the last $\eps$ units imply that
\[
p_t^{D,\eps}(y,z)\le C\eps^{-6d}\Pbb_y(\tau_{\partial D}>t-\eps).
\]
Integrating
in $y$ over $D$ (finite volume) and using
Proposition~\ref{prop:original-killed-exit-tail},
\[
\int_D p_t^{D,\eps}(y,z)\,dy\le C\eps^{-6d}e^{-\alpha(t-\eps)\eps^{\beta_0}}.
\]
Using
\[
\widetilde\E_z\left[e^{\lambda\widetilde\zeta_\eps^m}\right]
=1+\lambda\int_0^\infty e^{\lambda t}
\widetilde\Pbb_z\left(\widetilde\zeta_\eps^m>t\right)\,dt
\]
and bounding the interval $[0,\eps]$ by $e^{\lambda\eps}$, integration over
$t\ge\eps$ gives
\[
\widetilde\E_z \left[e^{\lambda\widetilde\zeta_\eps^m}\right]
\le C\eps^{-6d-\beta_0}. 
\]
The estimate is uniform up to the killed boundary by
monotone interior approximation.  The $s$-well proof is identical.
Thus one may take $M_{\rm com}=6d+\beta_0$.
\end{proof}

Substituting the auxiliary exponential-moment estimate into the
stopped $h$-transform bound gives the desired energetic localization of the
forward committor.  Momentum reversal then supplies the adjoint estimate.

\begin{proposition}[Third-order Lee--Ramil--Seo committor localization]
\label{prop:third-order-lrs-committor-localization}
There are $C,M_{\rm com}>0$ such that for every sufficiently small $\eps$,
\[
1-h_\eps(z)\le C\eps^{-M_{\rm com}}e^{(H(z)-U(\sigma))/\eps}\qquad\text{ on $W_m$},
\]
and
\[
h_\eps(z)\le C\eps^{-M_{\rm com}}e^{(H(z)-U(\sigma))/\eps}\qquad\text{ on $W_s$}.
\]
The same holds
for $h_\eps^*$.
\end{proposition}

\begin{proof}
Combine Lemma~\ref{lem:lrs-stopped-time-reversal} with
Proposition~\ref{prop:third-order-lrs-exponential-stopping} and
$\Theta_*W_i=W_i$.  The corresponding two estimates for the adjoint
committor $h_\eps^*$ follow from $h_\eps^*=h_\eps\circ\Theta_*$
(Proposition~\ref{prop:adjoint-lyapunov-dissipation}).
\end{proof}

\section{Prefactor Expansions and Numerical Details}
\label{app:prefactor-numerics}
\label{app:supplementary}

\subsection{Asymptotic Size of the Prefactor Improvement}
\label{app:prefactor-expansions}

Remark~\ref{rem:prefactor-comparison} shows that
$\mu_{\sigma}>\nu_{\sigma,L}$.  To quantify the difference, introduce the
dimensionless damping parameter
\[
\rho_\sigma:=\frac{\gamma}{\sqrt{a_\sigma}}
    =\gamma\sqrt{\frac{L}{\lambda_\sigma}},
\]
and write $M(\rho_\sigma):=\mu_\sigma/\sqrt{a_\sigma}$ and
$N(\rho_\sigma):=\nu_{\sigma,L}/\sqrt{a_\sigma}$.  Then $M,N$ satisfy
\[
    M^3+\rho_\sigma M^2+(\rho_\sigma^2-1)M-\rho_\sigma=0,
    \qquad
    N=\frac{\sqrt{\rho_\sigma^2+4}-\rho_\sigma}{2}.
\]
Expanding the positive roots at $\rho_\sigma=0$ gives
\[
    M(\rho_\sigma)
    =1-\frac12\rho_\sigma^2+\frac12\rho_\sigma^3
      +\mathcal O(\rho_\sigma^4),
      \qquad
    N(\rho_\sigma)
    =1-\frac12\rho_\sigma+\frac18\rho_\sigma^2
      +\mathcal O(\rho_\sigma^4),
\]
and hence
\begin{equation}\label{eq:small-rho-prefactor-gain}
    \mu_\sigma-\nu_{\sigma,L}
    =\frac{\gamma}{2}
     -\frac{5\gamma^2}{8\sqrt{a_\sigma}}
     +\frac{\gamma^3}{2a_\sigma}
     +\mathcal O\!\left(\frac{\gamma^4}{a_\sigma^{3/2}}\right).
\end{equation}
This is the regime of small $\gamma$, or of small $L$ with
$\gamma,\lambda_\sigma$ fixed.  At the other extreme,
\[
    M(\rho_\sigma)
      =\rho_\sigma^{-1}-\rho_\sigma^{-5}
       +\rho_\sigma^{-7}+\mathcal O(\rho_\sigma^{-9}),
    \qquad
    N(\rho_\sigma)
      =\rho_\sigma^{-1}-\rho_\sigma^{-3}
       +2\rho_\sigma^{-5}-5\rho_\sigma^{-7}
       +\mathcal O(\rho_\sigma^{-9}),
\]
so that
\begin{equation}\label{eq:large-rho-prefactor-gain}
    \mu_\sigma-\nu_{\sigma,L}
    =\frac{a_\sigma^2}{\gamma^3}
     -\frac{3a_\sigma^3}{\gamma^5}
     +\mathcal O\!\left(\frac{a_\sigma^4}{\gamma^7}\right)
    =\frac{\lambda_\sigma^2}{\gamma^3L^2}
     -\frac{3\lambda_\sigma^3}{\gamma^5L^3}
     +\mathcal O\!\left(\frac{\lambda_\sigma^4}{\gamma^7L^4}\right).
\end{equation}
This is the regime of large $\gamma$, or of large $L$ with the other
parameters fixed.  Thus, the strict improvement persists for every finite
positive parameter choice, but its relative size tends to zero in both
extreme regimes; it is most visible at intermediate values of $\rho_\sigma$.

\subsection{Complete Numerical Details}
\label{app:numerical-details}

The double-well model, theoretical constants, and overview figure are given in
Section~\ref{sec:numerics}; see in particular
\eqref{eq:numerical-potential}, \eqref{eq:numerical-theory-constants}, and
Figure~\ref{fig:numerical-prefactor-position}.  We record here the simulation
scheme, estimators, complete table, and step-size check.

\subsubsection{Discretization and Estimators}

The third-order Langevin diffusion starts from $(-1,0,0)$ and the underdamped
Langevin diffusion starts from $(-1,0)$.  Both are simulated by a symmetric
splitting method: the force and position substeps are explicit, and the linear
Ornstein--Uhlenbeck substep is sampled exactly.  Unless stated otherwise, the
time step is $\Delta t=0.005$.  For a first-passage time $\tau_j$, $j\in
\{2,3\}$, we estimate the normalized prefactor by
\begin{equation}
\label{eq:numerical-prefactor-estimator}
\widehat C_j(\eps)
=\overline\tau_j\exp\!\left(-\frac{1}{4\eps}\right),
\end{equation}
where $\overline\tau_j$ is the sample mean.  Error bars are pointwise $95\%$
normal confidence intervals based on independent trajectories; confidence
intervals for ratios use the delta method.  No trajectory used for
Table~\ref{tab:numerical-position} or
Figure~\ref{fig:numerical-prefactor-position} was right-censored.

We use the fixed right-well target
\begin{equation}
\label{eq:numerical-position-target}
\tau_j^{\mathrm{pos}}
=\inf\{t\geq0:\theta_t\geq0.9\},
\qquad j\in\{2,3\}.
\end{equation}
This target cleanly records passage through the saddle followed by entry into
the right well and provides a finite-temperature diagnostic of the saddle
prefactor.  Table~\ref{tab:numerical-position} reports $20{,}000$
independent trajectories per model and temperature.

\begin{table}[b]
\centering
\caption{Estimated normalized prefactors for the fixed right-well target
\eqref{eq:numerical-position-target}.  Parentheses give a pointwise $95\%$
confidence interval for the ratio.}
\label{tab:numerical-position}
\begin{tabular}{cccc}
\toprule
$\eps$ & $\widehat C_3(\eps)$ & $\widehat C_2(\eps)$
& $\overline\tau_3/\overline\tau_2$ \\
\midrule
$0.12$ & $7.563$ & $8.380$ & $0.903\;(0.887,0.919)$ \\
$0.10$ & $7.280$ & $8.273$ & $0.880\;(0.864,0.896)$ \\
$0.08$ & $6.983$ & $7.991$ & $0.874\;(0.857,0.890)$ \\
$0.07$ & $6.721$ & $7.949$ & $0.845\;(0.829,0.862)$ \\
$0.06$ & $6.540$ & $7.912$ & $0.827\;(0.811,0.842)$ \\
$0.05$ & $6.324$ & $7.676$ & $0.824\;(0.808,0.840)$ \\
\midrule
Eyring--Kramers limit & $5.886$ & $7.189$ & $0.819$ \\
\bottomrule
\end{tabular}
\end{table}

The normalized prefactors move toward their predicted values as the
temperature decreases.  More importantly for the comparison, the empirical
ratio approaches $\nu/\mu$: at $\eps=0.06$ it is $0.827\;(0.811,0.842)$ and
at $\eps=0.05$ it is $0.824\;(0.808,0.840)$, compared with the predicted value
$0.819$.  Both confidence intervals contain the theoretical ratio.  As a basic
step-size check, we also repeated the simulations at $\eps=0.07$ and $0.06$
with $\Delta t=0.0025$ and $1000$ paths per model.  The estimated ratios were
respectively $0.889\;(0.813,0.965)$ and $0.899\;(0.822,0.976)$; these intervals
overlap their $\Delta t=0.005$ counterparts.  At the present Monte Carlo
precision, the refinement therefore does not reveal a statistically resolved
time-step bias.

Overall, the experiment supports the two features of
Theorem~\ref{thm:EK} and Remark~\ref{rem:prefactor-comparison}: both models
have the same Arrhenius exponent, while the larger third-order unstable rate
produces a smaller leading metastable prefactor.

\clearpage
\bibliographystyle{alpha}
\bibliography{bibtex}

@book{BovierDenHollander2015,
  author    = {Bovier, Anton and den Hollander, Frank},
  title     = {Metastability: A Potential-Theoretic Approach},
  publisher = {Springer},
  year      = {2015}
}

@article{AvelinJulinViitasaari2023,
  author  = {Avelin, Benny and Julin, Vesa and Viitasaari, Lauri},
  title   = {Geometric Characterization of the {Eyring--Kramers} Formula},
  journal = {Communications in Mathematical Physics},
  volume  = {404},
  number  = {1},
  pages   = {401--437},
  year    = {2023},
  doi     = {10.1007/s00220-023-04845-z}
}

@article{Berglund2013Kramers,
  author  = {Berglund, Nils},
  title   = {{Kramers'} Law: Validity, Derivations and Generalisations},
  journal = {Markov Processes and Related Fields},
  volume  = {19},
  number  = {3},
  pages   = {459--490},
  year    = {2013}
}

@article{BonyLePeutrecMichel2025,
  author  = {Bony, Jean-Fran\c{c}ois and Le Peutrec, Dorian and Michel, Laurent},
  title   = {{Eyring--Kramers} Law for {Fokker--Planck} Type Differential Operators},
  journal = {Journal of the European Mathematical Society},
  volume  = {27},
  number  = {11},
  pages   = {4347--4398},
  year    = {2025},
  doi     = {10.4171/JEMS/1461}
}

@misc{Delande2026KFP,
  author        = {Delande, Lo{\"i}s},
  title         = {Hypocoercivity and Metastability of Degenerate {KFP} Equations at Low Temperature},
  year          = {2026},
  eprint        = {2601.04784},
  archivePrefix = {arXiv},
  primaryClass  = {math.AP},
  note          = {arXiv:2601.04784v2}
}

@article{BovierEckhoffGayrardKlein2004,
  author  = {Bovier, Anton and Eckhoff, Michael and Gayrard, V\'{e}ronique and Klein, Markus},
  title   = {Metastability in Reversible Diffusion Processes. {I}. {Sharp} Asymptotics for Capacities and Exit Times},
  journal = {Journal of the European Mathematical Society},
  volume  = {6},
  number  = {4},
  pages   = {399--424},
  year    = {2004},
  doi     = {10.4171/JEMS/14}
}

@article{HerauHitrikSjostrand2008,
  author  = {H\'{e}rau, Fr\'{e}d\'{e}ric and Hitrik, Michael and Sj\"{o}strand, Johannes},
  title   = {Tunnel Effect for {Kramers--Fokker--Planck} Type Operators},
  journal = {Annales Henri Poincar\'{e}},
  volume  = {9},
  number  = {2},
  pages   = {209--274},
  year    = {2008},
  doi     = {10.1007/s00023-008-0355-y}
}

@article{LandimSeo2018,
  author  = {Landim, Claudio and Seo, Insuk},
  title   = {Metastability of Non-reversible Random Walks in a Potential Field: The {Eyring--Kramers} Transition Rate Formula},
  journal = {Communications on Pure and Applied Mathematics},
  volume  = {71},
  number  = {2},
  pages   = {203--266},
  year    = {2018},
  doi     = {10.1002/cpa.21723}
}

@article{LeeSeo2022,
  author  = {Lee, Jungkyoung and Seo, Insuk},
  title   = {Non-reversible Metastable Diffusions with {Gibbs} Invariant Measure. {I}. {Eyring--Kramers} Formula},
  journal = {Probability Theory and Related Fields},
  volume  = {182},
  number  = {3--4},
  pages   = {849--903},
  year    = {2022},
  doi     = {10.1007/s00440-021-01102-z}
}

@misc{LeeRamilSeo2023Cutoff,
  author = {Lee, Seungwoo and Ramil, Mouad and Seo, Insuk},
  title  = {Asymptotic Stability and Cut-off Phenomenon for the Underdamped {Langevin} Dynamics},
  year   = {2023},
  note   = {arXiv:2311.18263}
}

@article{LelievreRamilReygner2022KFP,
  author  = {Leli\`{e}vre, Tony and Ramil, Mouad and Reygner, Julien},
  title   = {A Probabilistic Study of the Kinetic {Fokker--Planck} Equation in Cylindrical Domains},
  journal = {Journal of Evolution Equations},
  volume  = {22},
  pages   = {Paper No. 38},
  year    = {2022},
  doi     = {10.1007/s00028-022-00796-5}
}

@article{LelievreRamilReygner2022QSD,
  author  = {Leli\`{e}vre, Tony and Ramil, Mouad and Reygner, Julien},
  title   = {Quasi-stationary Distribution for the {Langevin} Process in Cylindrical Domains, Part {I}: Existence, Uniqueness and Long-time Convergence},
  journal = {Stochastic Processes and their Applications},
  volume  = {144},
  pages   = {173--201},
  year    = {2022},
  doi     = {10.1016/j.spa.2021.11.005}
}

@misc{LelievreRamilReygner2022Hill,
  author = {Leli\`{e}vre, Tony and Ramil, Mouad and Reygner, Julien},
  title  = {Estimation of Statistics of Transitions and {Hill} Relation for {Langevin} Dynamics},
  year   = {2022},
  note   = {arXiv:2206.13264}
}

@article{BramantiZhu2013,
  author  = {Bramanti, Marco and Zhu, Maochun},
  title   = {{$L^p$} and {Schauder} Estimates for Nonvariational Operators Structured on {H{\"o}rmander} Vector Fields with Drift},
  journal = {Analysis \& PDE},
  volume  = {6},
  number  = {8},
  pages   = {1793--1855},
  year    = {2013},
  doi     = {10.2140/apde.2013.6.1793}
}

@article{BouchetReygner2016,
  author  = {Bouchet, Freddy and Reygner, Julien},
  title   = {Generalisation of the {Eyring--Kramers} Transition Rate Formula to Irreversible Diffusion Processes},
  journal = {Annales Henri Poincar{\'e}},
  volume  = {17},
  number  = {12},
  pages   = {3499--3532},
  year    = {2016},
  doi     = {10.1007/s00023-016-0507-4}
}

@misc{DangGurbuzbalabanIslamYaoZhu2025,
  author = {Dang, Thanh and G{\"u}rb{\"u}zbalaban, Mert and Islam, Mohammad Rafiqul and Yao, Nian and Zhu, Lingjiong},
  title  = {High-Order {Langevin} {Monte Carlo} Algorithms},
  year   = {2025},
  note   = {arXiv:2508.17545v1}
}

@article{Hormander1967,
  author  = {H{\"o}rmander, Lars},
  title   = {Hypoelliptic Second Order Differential Equations},
  journal = {Acta Mathematica},
  volume  = {119},
  pages   = {147--171},
  year    = {1967}
}

@article{LandimMarianiSeo2019,
  author  = {Landim, Claudio and Mariani, Mauro and Seo, Insuk},
  title   = {{Dirichlet's} and {Thomson's} Principles for Non-selfadjoint Elliptic Operators with Application to Non-reversible Metastable Diffusion Processes},
  journal = {Archive for Rational Mechanics and Analysis},
  volume  = {231},
  number  = {2},
  pages   = {887--938},
  year    = {2019},
  doi     = {10.1007/s00205-018-1291-8}
}

@article{LePeutrecMichelNectoux2024,
  author  = {Le Peutrec, Dorian and Michel, Laurent and Nectoux, Boris},
  title   = {Exit Time and Principal Eigenvalue of Non-reversible Elliptic Diffusions},
  journal = {Communications in Mathematical Physics},
  volume  = {405},
  pages   = {202},
  year    = {2024},
  doi     = {10.1007/s00220-024-05032-4}
}

@misc{LeeRamilSeo2026,
  author = {Lee, Seungwoo and Ramil, Mouad and Seo, Insuk},
  title  = {Eyring--{K}ramers law for the underdamped {L}angevin process},
  year   = {2026},
  note   = {arXiv:2503.12610v2}
}

@article{Bony1969,
  author  = {Bony, Jean-Michel},
  title   = {Principe du maximum, in{\'e}galit{\'e} de {Harnack} et unicit{\'e} du probl{\`e}me de {Cauchy} pour les op{\'e}rateurs elliptiques d{\'e}g{\'e}n{\'e}r{\'e}s},
  journal = {Annales de l'Institut Fourier},
  volume  = {19},
  number  = {1},
  pages   = {277--304},
  year    = {1969},
  doi     = {10.5802/aif.319}
}

@article{Monmarche2018,
  author  = {Monmarch{\'e}, Pierre},
  title   = {Hypocoercivity in Metastable Settings and Kinetic Simulated Annealing},
  journal = {Probability Theory and Related Fields},
  volume  = {172},
  pages   = {1215--1248},
  year    = {2018},
  doi     = {10.1007/s00440-018-0828-y}
}

@article{Monmarche2023AlmostSure,
  author  = {Monmarch{\'e}, Pierre},
  title   = {Almost Sure Contraction for Diffusions on {$\mathbb{R}^d$}: Application to Generalized {Langevin} Diffusions},
  journal = {Stochastic Processes and their Applications},
  volume  = {161},
  pages   = {316--349},
  year    = {2023},
  doi     = {10.1016/j.spa.2023.04.006}
}

@article{MouMaWainwrightBartlettJordan2021,
  author  = {Mou, Wenlong and Ma, Yi-An and Wainwright, Martin J. and Bartlett, Peter L. and Jordan, Michael I.},
  title   = {High-Order {Langevin} Diffusion Yields an Accelerated {MCMC} Algorithm},
  journal = {Journal of Machine Learning Research},
  volume  = {22},
  number  = {42},
  pages   = {1--41},
  year    = {2021}
}

@article{Pigato2022,
  author  = {Pigato, Paolo},
  title   = {Density estimates and short-time asymptotics for a hypoelliptic diffusion process},
  journal = {Stochastic Processes and their Applications},
  volume  = {145},
  pages   = {117--142},
  year    = {2022},
  doi     = {10.1016/j.spa.2021.11.012}
}

@incollection{Norris1986,
  author    = {Norris, James R.},
  title     = {Simplified {Malliavin} Calculus},
  booktitle = {S{\'e}minaire de Probabilit{\'e}s XX, 1984/85},
  series    = {Lecture Notes in Mathematics},
  volume    = {1204},
  pages     = {101--130},
  publisher = {Springer},
  address   = {Berlin},
  year      = {1986},
  doi       = {10.1007/BFb0075716}
}

@book{Nualart2006,
  author    = {Nualart, David},
  title     = {The Malliavin Calculus and Related Topics},
  edition   = {Second},
  publisher = {Springer},
  address   = {Berlin},
  year      = {2006}
}

@article{RothschildStein1976,
  author  = {Rothschild, Linda Preiss and Stein, Elias M.},
  title   = {Hypoelliptic Differential Operators and Nilpotent Groups},
  journal = {Acta Mathematica},
  volume  = {137},
  number  = {3--4},
  pages   = {247--320},
  year    = {1976},
  doi     = {10.1007/BF02392419}
}

@article{high-order-Liu-2025,
  title={The {P}icard-{L}agrange framework for high-order {L}angevin {M}onte {C}arlo},
  author={Mahajan, Jaideep and Zhang, Kaihong and Liang, Feng and Liu, Jingbo},
  journal={arXiv:2510.18242},
  year={2025}
}

@article{chiang1987diffusion,
  title={Diffusion for global optimization in $\mathbb{R}^n$},
  author={Chiang, Tzuu-Shuh and Hwang, Chii-Ruey and Sheu, Shuenn Jyi},
  journal={SIAM Journal on Control and Optimization},
  volume={25},
  number={3},
  pages={737--753},
  year={1987},
  publisher={SIAM}
}

@article{GGZ,
  title={Global Convergence of {S}tochastic {G}radient {H}amiltonian {M}onte {C}arlo for Non-Convex Stochastic Optimization: Non-Asymptotic Performance Bounds and Momentum-Based Acceleration},
  author={Gao, Xuefeng and G\"{u}rb\"{u}zbalaban, Mert and Zhu, Lingjiong},
  journal={Operations Research},
  volume={70},
  number={5},
  pages={2931-2947},
  year={2022}
}

@article{JianfengLu,
	Author = {Cao, Yu and Lu, Jianfeng and Wang, Lihan},
	Journal = {Archive for Rational Mechanics and Analysis},
    volume={247},
    number={90},
    pages={1-34},
	Title = {On explicit {$L^{2}$}-convergence rate estimate for underdamped {L}angevin dynamics},
	Year = {2023}}

@article{DM2017,
  title={Non-asymptotic convergence analysis for the {U}nadjusted {L}angevin {A}lgorithm},
  author={Durmus, Alain and Moulines, Eric},
  journal={Annals of Applied Probability},
  volume={27},
  number={3},
  pages={1551-1587},
  year={2017}
}

@book{gelman1995bayesian,
  title={Bayesian Data Analysis},
  author={Gelman, Andrew and Carlin, John B and Stern, Hal S and Rubin, Donald B},
  year={1995},
  publisher={Chapman \& Hall/CRC Press}
}

@article{stuart2010inverse,
  title={Inverse problems: A {B}ayesian perspective},
  author={Stuart, Andrew M},
  journal={Acta Numerica},
  volume={19},
  pages={451--559},
  year={2010},
  publisher={Cambridge University Press}
}

@article{andrieu2003introduction,
  title={An introduction to {MCMC} for machine learning},
  author={Andrieu, Christophe and De Freitas, Nando and Doucet, Arnaud and Jordan, Michael I},
  journal={Machine Learning},
  volume={50},
  number={1},
  pages={5--43},
  year={2003},
  publisher={Springer}
}

@article{DistMCMC19,
	Author = {G\"urb\"uzbalaban, Mert and Gao, Xuefeng and Hu, Yuanhan and Zhu, Lingjiong},
	Journal = {Journal of Machine Learning Research},
	Volume={22},
	Number={239},
	Pages={1-69},
	Title = {Decentralized stochastic gradient {L}angevin dynamics and {H}amiltonian {M}onte {C}arlo},
	Year = 2021}

@article{GIWZ2024,
	Author = {G\"{u}rb\"{u}zbalaban, Mert and Islam, Mohammad Rafiqul and Wang, Xiaoyu and Zhu, Lingjiong},
	Journal = {arXiv preprint arXiv:2412.01993},
	Title = {Generalized {EXTRA} stochastic gradient {L}angevin dynamics},
	Year = {2024}}

@inproceedings{Raginsky,
  title={Non-convex learning via Stochastic Gradient {L}angevin Dynamics: a nonasymptotic analysis},
  author={Raginsky, Maxim and Rakhlin, Alexander and Telgarsky, Matus},
  booktitle={Proceedings of the 2017 Conference on Learning Theory},
  volume={65},
  organization={PMLR},
  pages={1674--1703},
  year={2017}
}

@article{Chau2019,
  title={On stochastic gradient {L}angevin dynamics with dependent data streams: the fully non-convex case},
  author={Chau, Ngoc Huy and Moulines, \'{E}ric and R\'{a}sonyi, Miklos and Sabanis, Sotirios and Zhang, Ying},
  journal={SIAM Journal of Mathematics of Data Science},
  volume={3},
  number={3},
  pages={959-986},
  year={2021}
}

@article{teh2016consistency,
  title={Consistency and fluctuations for stochastic gradient {L}angevin dynamics},
  author={Teh, Yee Whye and Thiery, Alexandre H and Vollmer, Sebastian J},
  journal={Journal of Machine Learning Research},
  volume={17},
  number={1},
  pages={193--225},
  year={2016},
  publisher={JMLR. org}
}

@article{Dalalyan,
  title={Theoretical guarantees for approximate sampling from smooth and log-concave densities},
  author={Dalalyan, Arnak S},
  journal={Journal of the Royal Statistical Society: Series B (Statistical Methodology)},
  volume={79},
  number={3},
  pages={651--676},
  year={2017},
  publisher={Wiley Online Library}
}

@article{Zhang2019,
  title={Nonasymptotic estimates for {S}tochastic {G}radient {L}angevin {D}ynamics under local conditions in nonconvex optimization},
  author={Zhang, Ying and Akyildiz, \"{O}mer Deniz and Damoulas, Theodoros and Sabanis, Sotirios},
  journal={Applied Mathematics \& Optimization},
  volume={87},
  pages={25},
  year={2023}
}

@article{Eberle,
  title={Couplings and quantitative contraction rates for {L}angevin dynamics},
  author={Eberle, Andreas and Guillin, Arnaud and Zimmer, Raphael},
  journal={Annals of Probability},
  volume={47},
  number={4},
  pages={1982-2010},
  year={2019}
}

@article{stroock-langevin-spectrum,
  title={Asymptotics of the spectral gap with applications to the theory of simulated annealing},
  author={Holley, Richard A and Kusuoka, Shigeo and Stroock, Daniel W},
  journal={Journal of Functional Analysis},
  volume={83},
  number={2},
  pages={333--347},
  year={1989},
  publisher={Academic Press}
}

@article{DK2017,
  title={User-friendly guarantees for the {L}angevin {M}onte {C}arlo with inaccurate gradient},
  author={Dalalyan, Arnak S. and Karagulyan, Avetik G.},
  journal={Stochastic Processes and their Applications},
  volume={129},
  number={12},
  pages={5278-5311},
  year={2019}
}

@inproceedings{xu2018global,
  title={Global convergence of {L}angevin dynamics based algorithms for nonconvex optimization},
  author={Xu, Pan and Chen, Jinghui and Zou, Difan and Gu, Quanquan},
  booktitle={Advances in Neural Information Processing Systems},
  pages={3122--3133},
  volume={31},
  publisher = {Curran Associates, Inc.},
  year={2018}
}

@article{Chau2022,
  title={Stochastic {G}radient {H}amiltonian {M}onte {C}arlo for non-convex learning},
  author={Chau, Ngoc Huy and R\'{a}sonyi, Mikl\'{o}s},
  journal={Stochastic Processes and their Applications},
  volume={149},
  pages={341-368},
  year={2022}
}

@book{FreidlinWentzell2012,
  author    = {Freidlin, Mark I. and Wentzell, Alexander D.},
  title     = {Random Perturbations of Dynamical Systems},
  series    = {Grundlehren der mathematischen Wissenschaften},
  volume    = {260},
  edition   = {Third},
  publisher = {Springer},
  address   = {Berlin and Heidelberg},
  year      = {2012},
  doi       = {10.1007/978-3-642-25847-3}
}

@book{Gantmacher1959,
  author    = {Gantmacher, Felix R.},
  title     = {The Theory of Matrices},
  volume    = {2},
  publisher = {Chelsea Publishing Company},
  address   = {New York},
  year      = {1959}
}

@misc{HeLiWangZhu2026WeakEquilibrium,
  author = {He, Ping and Li, Xiaodan and Wang, Yingli and Zhu, Lingjiong},
  title  = {Weak Equilibrium Measures and Capacity--Hitting Identities for the Hypoelliptic Third-Order {L}angevin Diffusion},
  year   = {2026},
  note   = {arXiv:2607.19752}
}

@inproceedings{GGZ2,
	Author = {Gao, Xuefeng and G\"{u}rb\"{u}zbalaban, Mert and Zhu, Lingjiong},
	publisher = {Curran Associates, Inc.},
	Booktitle={Advances in Neural Information Processing Systems (NeurIPS)},
	Volume={33},
	Title = {Breaking Reversibility Accelerates {L}angevin Dynamics for Global Non-Convex Optimization},
	Year = {2020}}

\end{document}